\documentclass{article}

\usepackage{times}
\usepackage{graphicx} 
\usepackage{subfigure}

\usepackage{natbib}

\usepackage{algorithm}
\usepackage{algorithmic}
\usepackage[algo2e,linesnumbered,lined,boxed,vlined]{algorithm2e}

\usepackage{amsfonts}
\usepackage{amssymb}

\usepackage{hyperref}

\usepackage[accepted]{icml2016}

\usepackage{comment}
\usepackage{amsthm}
\usepackage{mathtools}
\usepackage{tabularx}
\usepackage{multirow}

\newtheorem{theorem}{Theorem}
\newtheorem{lemma}{Lemma}

\newtheorem{corollary}{Corollary}
\newtheorem{definition}{Definition}

\newtheorem{remark}{Remark}

\def\b{\ensuremath\boldsymbol}

\usepackage{lastpage}
\usepackage{fancyhdr}
\usepackage{url}
\icmltitlerunning{}

\begin{document}


\twocolumn[
\icmltitle{KKT Conditions, First-Order and Second-Order Optimization, and Distributed Optimization: Tutorial and Survey}

\icmlauthor{Benyamin Ghojogh}{bghojogh@uwaterloo.ca}
\icmladdress{Department of Electrical and Computer Engineering, 
\\Machine Learning Laboratory, University of Waterloo, Waterloo, ON, Canada}
\icmlauthor{Ali Ghodsi}{ali.ghodsi@uwaterloo.ca}
\icmladdress{Department of Statistics and Actuarial Science \& David R. Cheriton School of Computer Science, 
\\Data Analytics Laboratory, University of Waterloo, Waterloo, ON, Canada}
\icmlauthor{Fakhri Karray}{karray@uwaterloo.ca}
\icmladdress{Department of Electrical and Computer Engineering, 
\\Centre for Pattern Analysis and Machine Intelligence, University of Waterloo, Waterloo, ON, Canada}
\icmlauthor{Mark Crowley}{mcrowley@uwaterloo.ca}
\icmladdress{Department of Electrical and Computer Engineering, 
\\Machine Learning Laboratory, University of Waterloo, Waterloo, ON, Canada}

\icmlkeywords{Tutorial}

\vskip 0.3in
]

\begin{abstract}
This is a tutorial and survey paper on Karush-Kuhn-Tucker (KKT) conditions, first-order and second-order numerical optimization, and distributed optimization. After a brief review of history of optimization, we start with some preliminaries on properties of sets, norms, functions, and concepts of optimization. Then, we introduce the optimization problem, standard optimization problems (including linear programming, quadratic programming, and semidefinite programming), and convex problems. We also introduce some techniques such as eliminating inequality, equality, and set constraints, adding slack variables, and epigraph form. We introduce Lagrangian function, dual variables, KKT conditions (including primal feasibility, dual feasibility, weak and strong duality, complementary slackness, and stationarity condition), and solving optimization by method of Lagrange multipliers. Then, we cover first-order optimization including gradient descent, line-search, convergence of gradient methods, momentum, steepest descent, and backpropagation. Other first-order methods are explained, such as accelerated gradient method, stochastic gradient descent, mini-batch gradient descent, stochastic average gradient, stochastic variance reduced gradient, AdaGrad, RMSProp, and Adam optimizer, proximal methods (including proximal mapping, proximal point algorithm, and proximal gradient method), and constrained gradient methods (including projected gradient method, projection onto convex sets, and Frank-Wolfe method). We also cover non-smooth and $\ell_1$ optimization methods including lasso regularization, convex conjugate, Huber function, soft-thresholding, coordinate descent, and subgradient methods. Then, we explain second-order methods including Newton's method for unconstrained, equality constrained, and inequality constrained problems. We explain the interior-point method, barrier methods, Wolfe conditions for line-search, fast solving system of equations (including decomposition methods and conjugate gradient), and quasi-Newton's method (including BFGS, LBFGS, Broyden, DFP, and SR1 methods). The sequential convex programming for non-convex optimization is also introduced. Finally, we explain distributed optimization including alternating optimization, dual decomposition methods, augmented Lagrangian, and alternating direction method of multipliers (ADMM). We also introduce some techniques for using ADMM for many constraints and variables. 
\end{abstract}



\onecolumn
\tableofcontents


\twocolumn

\section{Introduction}

\textbf{-- KKT conditions and numerical optimization:}
Numerical optimization has application in various fields of science. Many of the optimization methods can be explained in terms of the Karush-Kuhn-Tucker (KKT) conditions \cite{kjeldsen2000contextualized}, proposed in \cite{karush1939minima,kuhn1951nonlinear}. 
The KKT conditions discuss the primal and dual problems with primal and dual variables, respectively, where the dual minimum is a lower-bound on the primal optimum.
In an unconstrained problem, if setting the gradient of a cost function to zero gives a closed-form solution, the optimization is done; however, if we do not have a closed-form solution, we should use numerical optimization which finds the solution iteratively and gradually. Besides, if the optimization is constrained, constrained numerical optimization should be used. 
The numerical optimization methods can be divided into first-order and second-order methods. 

\hfill\break
\textbf{-- History of first-order optimization:}
The first-order methods are based on gradient while the second-order methods make use of Hessian or approximation of Hessian as well as gradient. 
The most well-known first-order method is gradient descent, first suggested by Cauchy in 1874 \cite{lemarechal2012cauchy} and Hadamard in 1908 \cite{hadamard1908memoire}, whose convergence was later analyzed in \cite{curry1944method}. Backpropagation, for training neural networks, was proposed in \cite{rumelhart1986learning} and it is gradient descent used with chain rule. 
It is found out in 1980's that gradient descent is not optimal in convergence rate. Therefore, Accelerated Gradient Method (AGM) was proposed by Nesterov \cite{nesterov1983method,nesterov1988approach,nesterov2005smooth} which had an optimal convergence rate in gradient methods. 
Stochastic methods were also proposed for large volume optimization when we have a dataset of points. They randomly sample points or batches of points for use in gradient methods. Stochastic Gradient Descent (SGD), first proposed in \cite{robbins1951stochastic}, was first used for machine learning in \cite{bottou1998online}. Stochastic Average Gradient (SAG) \cite{roux2012stochastic}, Stochastic Variance Reduced Gradient (SVRG) \cite{johnson2013accelerating} are two other example methods in this category. 
Some techniques, such as AdaGrad \cite{duchi2011adaptive}, Root Mean Square Propagation (RMSProp) \cite{tieleman2012lecture}, Adaptive Moment Estimation (Adam) \cite{kingma2014adam}, have also been proposed for adaptive learning rate in stochastic optimization. 

\hfill\break
\textbf{-- History of proximal methods:}
Another family of optimization methods are the proximal methods \cite{parikh2014proximal} which are based on the Moreau-Yosida regularization \cite{moreau1965proximite,yosida1965functional}. Some proximal methods are the proximal point algorithm \cite{rockafellar1976monotone} and the proximal gradient method \cite{nesterov2013gradient}. The proximal mapping can also be used for constrained gradient methods such as projected gradient method \cite{iusem2003convergence}. Another effective first-order method for constrained problems is the Frank-Wolfe method \cite{frank1956algorithm}. 

\hfill\break
\textbf{-- History of non-smooth optimization:}
Optimization of non-smooth functions is also very important especially because of use of $\ell_1$ norm for sparsity in many applications. Some techniques for $\ell_1$ norm optimization are $\ell_1$ norm approximation by Huber function \cite{huber1992robust}, soft-thresholding which is the proximal mapping of $\ell_1$ norm, and coordinate descent \cite{wright2015coordinate} which can be used for $\ell_1$ norm optimization \cite{wu2008coordinate}. Subgradient methods, including stochastic subgradient method \cite{shor1998nondifferentiable} and projected subgradient method \cite{alber1998projected}, can also be used for non-smooth optimization. 

\hfill\break
\textbf{-- History of second-order optimization:}
Second-order methods use Hessian, or inverse of Hessian, or their approximations. The family of second-order methods can be named the Newton's methods which are based on the Newton-Raphson method \cite{stoer2013introduction}. 
Constrained second-order methods can be solved using the interior-point method, first proposed in \cite{dikin1967iterative}. The interior-point method is also called the barrier methods \cite{boyd2004convex,nesterov2018lectures} and Sequential Unconstrained Minimization Technique (SUMT) \cite{fiacco1967sequential}.
Interior-point method is a very powerful method and is often the main method of solving optimization problems in optimization toolboxes such as CVX \cite{grant2009cvx}. 

The second-order methods are usually faster than the first-order methods because of using the Hessian information. However, computation of Hessian or approximation of Hessian in second-order methods is time-consuming and difficult. This might be the reason for why most machine learning algorithms, such as backpropagation for neural networks \cite{rumelhart1986learning}, use first-order methods. Although, note that some few machine learning algorithms, such as logistic regression and Sammon mapping \cite{sammon1969nonlinear}, use second-order optimization. 

The update of solution in either first-order or second-order methods can be stated as a system of linear equations. 
For large-scale optimization, the Newton's method becomes slow and intractable. 
Therefore, decomposition methods \cite{golub2013matrix}, conjugate gradient method \cite{hestenes1952methods}, and nonlinear conjugate gradient method \cite{fletcher1964function,polak1969note,hestenes1952methods,dai1999nonlinear} can be used for approximation of solution to the system of equations. 
Truncated Newton's methods \cite{nash2000survey}, used for large scale optimization, usually use conjugate gradient.
Another approach for approximating Newton's method for large-scale data is the quasi-Newton's method {\citep[Chapter 6]{nocedal2006numerical}} which approximates the Hessian or inverse Hessian matrix. 
The well-known algorithms for quasi-Newton's method are Broyden-Fletcher-Goldfarb-Shanno (BFGS) \cite{fletcher1987practical,dennis1996numerical}, limited-memory BFGS (LBFGS) \cite{nocedal1980updating,liu1989limited}, Davidon-Fletcher-Powell (DFP) \cite{davidon1991variable,fletcher1987practical}, Broyden method \cite{broyden1965class}, and Symmetric Rank-one (SR1) \cite{conn1991convergence}.

\hfill\break
\textbf{-- History of line-search:}
Both first-order and second-order optimization methods have a step size parameter to move toward their descent direction. This step size can be calculated at every iteration using line-search methods. Well-known line-search methods are the backtracking or Armijo line-search \cite{armijo1966minimization} and the Wolfe conditions \cite{wolfe1969convergence}. 

\hfill\break
\textbf{-- Standard problems:}
The terms ``programming" and ``program" are sometimes used to mean ``optimization" and "optimization problem", respectively, in the literature. 
Convex optimization or convex programming started to develop since 1940's \cite{tikhomirov1996evolution}. There exist some standard forms for convex problems which are linear programming, quadratic programming, quadratically constrained quadratic programming, second-order cone programming, and Semidefinite Programming (SDP). An important method for solving linear programs was the simplex method proposed in 1947 \cite{dantzig1983reminiscences}. SDP is also important because the standard convex problems can be stated as special cases of SDP and then may be solved using the interior-point method. 

\hfill\break
\textbf{-- History of non-convex optimization:}
There also exist methods for non-convex optimization. These methods are either local or global methods. The local methods are faster but find a local solution depending on the initial solution. The global methods, however, find the global solution but are slower. Examples for local and global non-convex methods are Sequential Convex Programming (SCP) \cite{dinh2010local} and branch and bound \cite{land1960automatic}, respectively. SCP uses trust region \cite{conn2000trust} and it solves a sequence of convex approximations of the problem. It is related to Sequential Quadratic Programming (SQP) \cite{boggs1995sequential} which is used for constrained nonlinear optimization. 
The branch and bound methods use a binary tree structure for optimizing on a non-convex cost function. 

\hfill\break
\textbf{-- History of distributed optimization:}
Distributed optimization has two benefits. First, it makes the problem able to run in parallel on several servers. Secondly, it can be used to solve problems with multiple optimization variables. Especially, for the second reason, it has been widely used in machine learning and signal processing. 
Two most well-known distributed optimization approaches are alternating optimization \cite{jain2017non,li2019alternating} and Alternating Direction Method of Multipliers (ADMM) \cite{gabay1976dual,glowinski1976finite,boyd2011distributed}.
Alternating optimization alternates between optimizing over variables one-by-one, iteratively. 
ADMM is based on dual decomposition \cite{dantzig1960decomposition,benders1962partitioning,everett1963generalized} and augmented Lagrangian \cite{hestenes1969multiplier,powell1969method}. ADMM has also been generalized for multiple variables and constraints \cite{giesen2016distributed,giesen2019combining}. 

\hfill\break
\textbf{-- History of iteratively decreasing the feasible set:}
Cutting-plane methods remove a part of feasible point at every iteration where the removed part does not contain the minimizer. The feasible set gets smaller and smaller until it converges to the solution. The most well-known cutting-plane method is the Analytic Center Cutting-Plane Method (ACCPM) \cite{goffin1993computation,nesterov1995cutting,atkinson1995cutting}. 
Ellipsoid method \cite{shor1977cut,yudin1976informational,yudin1977evaluation,yudin1977optimization} has a similar idea but it removes half of an ellipsoid around the current solution at every iteration. 
The ellipsoid method was initially applied to liner programming \cite{khachiyan1979polynomial}. 


\hfill\break
\textbf{-- History of other optimization approaches:}
There exist some other approaches for optimization. In this paper, for brevity, we do not explain the theory of these other approaches and we merely focus on the classical optimization. 
Riemannian optimization \cite{absil2009optimization,boumal2020introduction} is the extension of Euclidean optimization to the cases where the optimization variable lies on a possibly curvy Riemannian manifold \cite{hosseini2020recent,hu2020brief} such as the symmetric positive definite \cite{sra2015conic}, quotient \cite{lee2013quotient}, Grassmann \cite{bendokat2020grassmann}, and Stiefel \cite{edelman1998geometry} manifolds.

Metaheuristic optimization \cite{talbi2009metaheuristics}, in the field of soft computing, is a a family of methods finding the optimum of a cost function using efficient, and not brute-force, search. 
They use both local and global searches  for exploitation and exploration of the cost function, respectively. 
They can be used in highly non-convex optimization with many constraints, where classical optimization is a little difficult and slow to perform. 
These methods contain nature-inspired optimization \cite{yang2010nature}, evolutionary computing \cite{simon2013evolutionary}, and particle-based optimization. 
Two fundamental metaheuristic methods are genetic algorithm \cite{holland1992adaptation} and particle swarm optimization \cite{kennedy1995particle}.

\hfill\break
\textbf{-- Important books on optimization:}
Some important books on optimization are Boyd's book \cite{boyd2004convex}, Nocedal's book \cite{nocedal2006numerical}, Nesterov's books \cite{nesterov1998introductory,nesterov2003introductory,nesterov2018lectures} (The book \cite{nesterov2003introductory} is a good book on first-order methods), Beck's book \cite{beck2017first}, and some other books \cite{dennis1996numerical,avriel2003nonlinear,chong2004introduction,bubeck2014convex,jain2017non}, etc. 

In this paper, we introduce and explain these optimization methods and approaches. 

\section*{Required Background for the Reader}

This paper assumes that the reader has general knowledge of calculus and linear algebra. 

\section{Notations and Preliminaries}\label{section_preliminaries}

\subsection{Preliminaries on Sets and Norms}

\begin{definition}[Interior and boundary of set]
Consider a set $\mathcal{D}$ in a metric space $\mathcal{R}^d$. The point $\b{x} \in \mathcal{D}$ is an interior point of the set if:
\begin{align*}
\exists \epsilon > 0\,\,\,\, \text{such that}\,\,\,\, \{\b{y}\, |\, \|\b{y} - \b{x}\|_2 \leq \epsilon\} \subseteq \mathcal{D}.
\end{align*}
The interior of the set, denoted by $\textbf{int}(\mathcal{D})$, is the set containing all the interior points of the set. 
The closure of the set is defined as $\textbf{cl}(\mathcal{D}) := \mathbb{R}^d \setminus \textbf{int}(\mathbb{R}^d \setminus \mathcal{D})$.
The boundary of set is defined as $\textbf{bd}(\mathcal{D}) := \textbf{cl}(\mathcal{D}) \setminus \textbf{int}(\mathcal{D})$.
An open (resp. closed) set does not (resp. does) contain its boundary. 
The closure of set can be defined as the smallest closed set containing the set. In other words, the closure of set is the union of interior and boundary of the set. 
\end{definition}

\begin{definition}[Convex set and convex hull]
A set $\mathcal{D}$ is a convex set if it completely contains the line segment between any two points in the set $\mathcal{D}$:
\begin{align*}
\forall \b{x}, \b{y} \in \mathcal{D}, 0 \leq t \leq 1 \implies t \b{x} + (1 - t) \b{y} \in \mathcal{D}.
\end{align*}
The convex hull of a (not necessarily convex) set $\mathcal{D}$ is the smallest convex set containing the set $\mathcal{D}$. If a set is convex, it is equal to its convex hull. 
\end{definition}

\begin{definition}[Minimum, maximum, infimum, and supremum]
A minimum and maximum of a function $f: \mathbb{R}^d \rightarrow \mathbb{R}$, $f: \b{x} \mapsto f(\b{x})$, with domain $\mathcal{D}$, are defined as:
\begin{align*}
& \min_{\b{x}} f(\b{x}) \leq f(\b{y}),\,\, \forall \b{y} \in \mathcal{D}, \\
& \max_{\b{x}} f(\b{x}) \geq f(\b{y}),\,\, \forall \b{y} \in \mathcal{D}, 
\end{align*}
respectively. The minimum and maximum of a function belong to the range of function. Infimum and supremum are the lower-bound and upper-bound of function, respectively:
\begin{align*}
& \inf_{\b{x}} f(\b{x}) := \max\{\b{z} \in \mathbb{R}\, |\, \b{z} \leq f(\b{x}), \forall \b{x} \in \mathcal{D}\}, \\
& \sup_{\b{x}} f(\b{x}) := \min\{\b{z} \in \mathbb{R}\, |\, \b{z} \geq f(\b{x}), \forall \b{x} \in \mathcal{D}\}.
\end{align*}
Depending on the function, the infimum and supremum of a function may or may not belong to the range of function. Fig. \ref{figure_min_max_inf_sup} shows some examples for minimum, maximum, infimum, and supremum. The minimum and maximum of a function are also the infimum and supremum of function, respectively, but the converse is not necessarily true. 
If the minimum and maximum of function are minimum and maximum in the entire domain of function, they are the global minimum and global maximum, respectively. 
See Fig. \ref{figure_StationaryPoints} for examples of global minimum and maximum. 
\end{definition}

\begin{figure}[!t]
\centering
\includegraphics[width=3.2in]{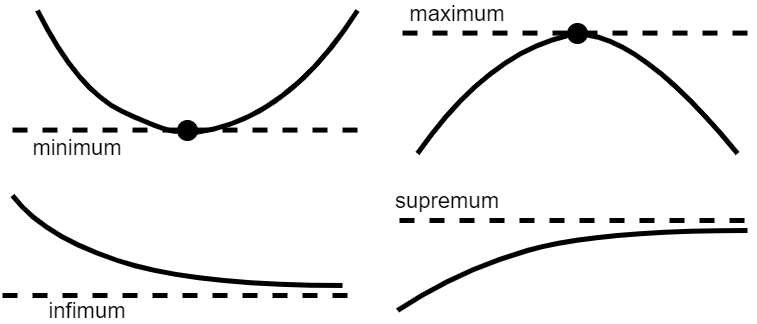}
\caption{Minimum, maximum, infimum, and supremum of example functions.}
\label{figure_min_max_inf_sup}
\end{figure}

\begin{figure}[!t]
\centering
\includegraphics[width=3.2in]{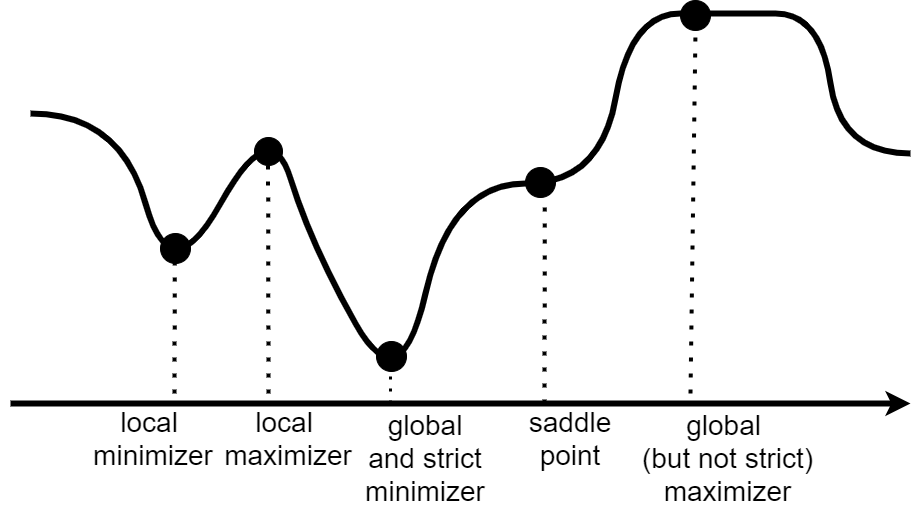}
\caption{Examples for stationary points such as local and global extreme points, strict and non-strict extreme points, and saddle point.}
\label{figure_StationaryPoints}
\end{figure}

\begin{lemma}[Inner product]
Consider two vectors $\b{x} = [x_1, \dots, x_d]^\top \in \mathbb{R}^d$ and $\b{y} = [y_1, \dots, y_d]^\top \in \mathbb{R}^d$. Their inner product, also called dot product, is:
\begin{align*}
& \langle \b{x}, \b{y} \rangle = \b{x}^\top \b{y} = \sum_{i=1}^d x_i\, y_i.
\end{align*}
We also have inner product between matrices $\b{X}, \b{Y} \in \mathbb{R}^{d_1 \times d_2}$. Let $\b{X}_{ij}$ denote the $(i,j)$-th element of matrix $\b{X}$. The inner product of $\b{X}$ and $\b{Y}$ is:
\begin{align*}
& \langle \b{X}, \b{Y} \rangle = \textbf{tr}(\b{X}^\top \b{Y}) = \sum_{i=1}^{d_1} \sum_{j=1}^{d_2} \b{X}_{i,j}\, \b{Y}_{i,j},
\end{align*}
where $\textbf{tr}(.)$ denotes the trace of matrix. 
\end{lemma}

\begin{definition}[Norm]
A function $\|\cdot\|: \mathbb{R}^d \rightarrow \mathbb{R}$, $\|\cdot\|: \b{x} \mapsto \|\b{x}\|$ is a norm if it satisfies:
\begin{enumerate}
\item $\|\b{x}\| \geq 0, \forall \b{x}$
\item $\|a \b{x}\| = |a|\, \|\b{x}\|, \forall \b{x}$ and all scalars $a$
\item $\|\b{x}\| = 0$ if and only if $\b{x} = 0$
\item Triangle inequality: $\|\b{x} + \b{y}\| \leq \|\b{x}\| + \|\b{y}\|$.
\end{enumerate}
\end{definition}

\begin{definition}[Important norms]
Some important norms for a vector $\b{x} = [x_1, \dots, x_d]^\top$ are as follows. 
The $\ell_p$ norm is:
\begin{align*}
& \|\b{x}\|_p := \big(|x_1|^p + \dots + |x_d|^p\big)^{1/p},
\end{align*}
where $p \geq 1$ and $|.|$ denotes the absolute value. 
Two well-known $\ell_p$ norms are $\ell_1$ norm and $\ell_2$ norm (also called the Euclidean norm) with $p=1$ and $p=2$, respectively.
The $\ell_\infty$ norm, also called the infinity norm, the maximum norm, or the Chebyshev norm, is:
\begin{align*}
& \|\b{x}\|_\infty := \max\{|x_1| + \dots + |x_d|\}. 
\end{align*}
For the matrix $\b{X} \in \mathbb{R}^{d_1 \times d_2}$, the $\ell_p$ norm is:
\begin{align*}
\|\b{X}\|_p := \sup_{\b{y} \neq 0} \frac{\|\b{X}\b{y}\|_p}{\|\b{y}\|_p}.
\end{align*}
A special case for this is the $\ell_2$ norm, also called the spectral norm or the Euclidean norm. The spectral norm is related to the largest singular value of matrix:
\begin{align*}
\|\b{X}\|_2 = \sup_{\b{y} \neq 0} \frac{\|\b{X}\b{y}\|_2}{\|\b{y}\|_2} = \sqrt{\lambda_{\text{max}}(\b{X}^\top \b{X})} = \sigma_{\text{max}}(\b{X}),
\end{align*}
where $\lambda_{\text{max}}(\b{X}^\top \b{X})$ and $\sigma_{\text{max}}(\b{X})$ denote the largest eigenvalue of $\b{X}^\top \b{X}$ and the largest singular value of $\b{X}$, respectively.
Other specaial cases are the maximum-absolute-column-sum norm ($p=1$) and the maximum-absolute-row-sum norm ($p=\infty$):
\begin{align*}
& \|\b{X}\|_1 = \sup_{\b{y} \neq 0} \frac{\|\b{X}\b{y}\|_1}{\|\b{y}\|_1} = \max_{1 \leq j \leq d_2} \sum_{i=1}^{d_1} |\b{X}_{i,j}|, \\
& \|\b{X}\|_\infty = \sup_{\b{y} \neq 0} \frac{\|\b{X}\b{y}\|_\infty}{\|\b{y}\|_\infty} = \max_{1 \leq i \leq d_1} \sum_{j=1}^{d_2} |\b{X}_{i,j}|.
\end{align*}
The formulation of the Frobenius norm for a matrix is similar to the formulation of $\ell_2$ norm for a vector:
\begin{align*}
\|\b{X}\|_F := \sqrt{\sum_{i=1}^{d_1} \sum_{j=1}^{d_2} \b{X}_{i,j}^2},
\end{align*}
where $\b{X}_{ij}$ denotes the $(i,j)$-th element of $\b{X}$.

The $\ell_{2,1}$ norm of matrix $\b{X}$ is:
\begin{align*}
& \|\b{X}\|_{2,1} := \sum_{i=1}^{d_1} \sqrt{\sum_{j=1}^{d_2} \b{X}_{i,j}^2}.
\end{align*}

The Schatten $\ell_p$ norm of matrix $\b{X}$ is:
\begin{align*}
\|\b{X}\|_p := \bigg(\sum_{i=1}^{\min(d_1, d_2)} \big(\sigma_i(\b{X})\big)^p\bigg)^{1/p},
\end{align*}
where $\sigma_i(\b{X})$ denotes the $i$-th singular value of $\b{X}$.
A special case of the Schatten norm, with $p=1$, is called the nuclear norm or the trace norm \cite{fan1951maximum}:
\begin{align*}
\|\b{X}\|_* := \sum_{i=1}^{\min(d_1, d_2)} \sigma_i(\b{X}) = \textbf{tr}\Big(\!\sqrt{\b{X}^\top \b{X}}\,\Big),
\end{align*}
which is summation of the singular values of matrix. 
Note that similar to use of $\ell_1$ norm of vector for sparsity, the nuclear norm is also used to impose sparsity on matrix. 
\end{definition}

\begin{lemma}
We have:
\begin{align*}
& \|\b{x}\|_2^2 = \b{x}^\top \b{x} = \langle \b{x}, \b{x} \rangle, \\
& \|\b{X}\|_F^2 = \textbf{tr}(\b{X}^\top \b{X}) = \langle \b{X}, \b{X} \rangle, 
\end{align*}
which are convex and in quadratic forms. 
\end{lemma}

\begin{definition}[Unit ball]\label{definition_unit_ball}
The unit ball for a norm $\|\cdot\|$ is:
\begin{align*}
& \mathcal{B} := \{\b{x} \in \mathbb{R}^d\, |\, \|\b{x}\| \leq 1 \}.
\end{align*}
The unit balls for some of the norms are shown in Fig. \ref{figure_unit_balls}.
\end{definition}

\begin{figure}[!t]
\centering
\includegraphics[width=3.2in]{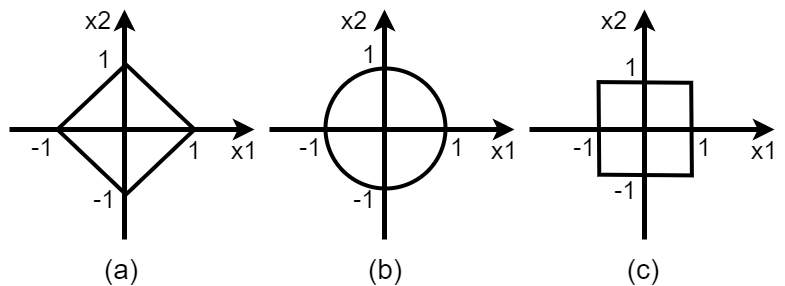}
\caption{The unit balls, in $\mathbb{R}^2$, for (a) $\ell_1$ norm, (b) $\ell_2$ norm, and (c) $\ell_\infty$ norm.}
\label{figure_unit_balls}
\end{figure}

\begin{definition}[Dual norm]
Let $\|.\|$ be a norm on $\mathbb{R}^d$. Its dual norm is:
\begin{align}
\|\b{x}\|_* := \sup\{\b{x}^\top \b{y}\, |\, \|\b{y}\| \leq 1\}. 
\end{align}
Note that the notation $\|\cdot\|_*$ should not be confused with the the nuclear norm despite of similarity of notations.
\end{definition}

\begin{lemma}[H{\"o}lder's \cite{holder1889ueber} and Cauchy-Schwarz inequalities \cite{steele2004cauchy}]
Let $p,q \in [1, \infty]$ and:
\begin{align}\label{Holder_inequality_1p_1q_equals_1}
\frac{1}{p} + \frac{1}{q} = 1.
\end{align}
These $p$ and $q$ are called the H{\"o}lder conjugates of each other. According to the H{\"o}lder's inequality, for functions $f(.)$ and $g(.)$, we have $\|fg\|_1 \leq \|f\|_p \|g\|_q$. 
A corollary of the H{\"o}lder's inequality is that Eq. (\ref{Holder_inequality_1p_1q_equals_1}) holds if the norms $\|.\|_p$ and $\|.\|_q$ are dual of each other. H{\"o}lder's inequality states that:
\begin{align*}
|\b{x}^\top \b{y}| \leq \|\b{x}\|_p \|\b{x}\|_q,
\end{align*}
where $p$ and $q$ satisfy Eq. (\ref{Holder_inequality_1p_1q_equals_1}). A special case of the H{\"o}lder's inequality is the Cauchy-Schwarz inequality, stated as $|\b{x}^\top \b{y}| \leq \|\b{x}\|_2 \|\b{x}\|_2$.
\end{lemma}

According to Eq. (\ref{Holder_inequality_1p_1q_equals_1}), we have:
\begin{align}\label{equation_dual_norm_calculation}
&\|\cdot\|_p \implies \|\cdot\|_* = \|\cdot\|_{p / (p-1)}, \quad \forall p \in [1, \infty].
\end{align}
For example, the dual norm of $\|.\|_2$ is $\|.\|_2$ again and the dual norm of $\|.\|_1$ is $\|.\|_\infty$.

\begin{definition}[Cone and dual cone]
A set $\mathcal{K} \subseteq \mathbb{R}^d$ is a cone if:
\begin{enumerate}
\item it contains the origin, i.e., $\b{0} \in \mathcal{K}$,
\item $\mathcal{K}$ is a convex set,
\item for each $\b{x} \in \mathcal{K}$ and $\lambda \geq 0$, we have $\lambda \b{x} \in \mathcal{K}$.
\end{enumerate}
The dual cone of a cone $\mathcal{K}$ is:
\begin{align*}
\mathcal{K}^* := \{\b{y}\, |\, \b{y}^\top \b{x} \geq 0, \forall \b{x} \in \mathcal{K}\}.
\end{align*}
An example cone and its dual are depicted in Fig. \ref{figure_dual_cone}-a.
\end{definition}

\begin{figure}[!t]
\centering
\includegraphics[width=2.8in]{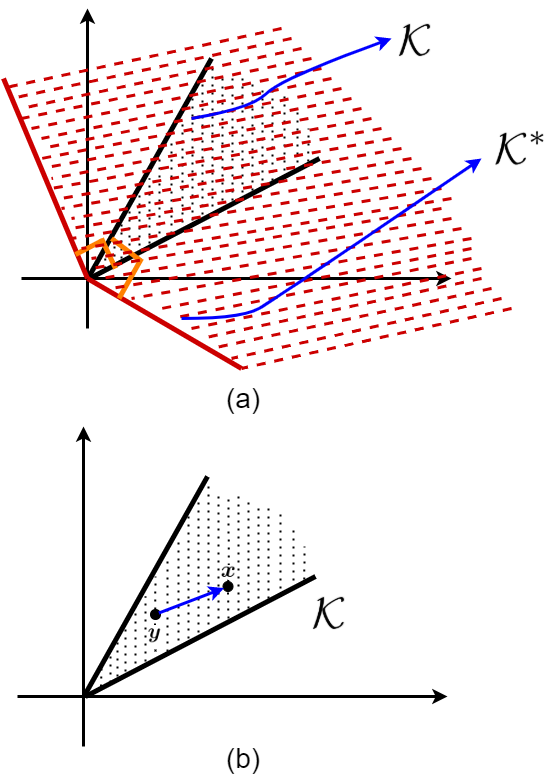}
\caption{(a) A cone $\mathcal{K}$ and its dual cone $\mathcal{K}^*$. Note that the borders of dual cone are perpendicular to the borders of the cone as shown in this figure. (b) An example for the generalized inequality $\b{x} \succeq_\mathcal{K} \b{y}$. As it is shown, the vector $(\b{x} - \b{y})$ belongs to the cone $\mathcal{K}$.}
\label{figure_dual_cone}
\end{figure}

\begin{definition}[Proper cone \cite{boyd2004convex}]
A convex cone $\mathcal{K} \subseteq \mathbb{R}^d$ is a proper cone if:
\begin{enumerate}
\item $\mathcal{K}$ is closed, i.e., it contains its boundary, 
\item $\mathcal{K}$ is solid, i.e., its interior is non-empty, 
\item $\mathcal{K}$ is pointed, i.e., it contains no line. In other words, it is not a two-sided cone around the origin.
\end{enumerate}
\end{definition}

\begin{definition}[Generalized inequality \cite{boyd2004convex}]\label{definition_generalized_inequality}
A generalized inequality, defined by a proper cone $\mathcal{K}$, is:
\begin{align*}
\b{x} \succeq_\mathcal{K} \b{y} \iff \b{x} - \b{y} \in \mathcal{K}.
\end{align*}
This means $\b{x} \succeq_\mathcal{K} \b{y} \iff \b{x} - \b{y} \in \textbf{int}(\mathcal{K})$. 
Note that $\b{x} \succeq_\mathcal{K} \b{y}$ can also be stated as $\b{x} - \b{y} \succeq_\mathcal{K} \b{0}$.
An example for a generalized inequality is shown in Fig. \ref{figure_dual_cone}-b.
\end{definition}

\begin{definition}[Important examples for generalized inequality]\label{definition_generalized_inequality_examples}
The generalized inequality defined by the non-negative orthant, $\mathcal{K} = \mathbb{R}_+^d$, is the default inequality for vectors $\b{x} = [x_1, \dots, x_d]^\top$, $\b{y} = [y_1, \dots, y_d]^\top$:
\begin{align*}
\b{x} \succeq \b{y} \iff \b{x} \succeq_{\mathbb{R}_+^d} \b{y}.
\end{align*}
It means component-wise inequality:
\begin{align*}
\b{x} \succeq \b{y} \iff x_i \geq y_i, \quad \forall i \in \{1, \dots, d\}.
\end{align*}
The generalized inequality defined by the positive definite cone, $\mathcal{K} = \mathbb{S}_+^d$, is the default inequality for symmetric matrices $\b{X}, \b{Y} \in \mathbb{S}^d$:
\begin{align*}
\b{X} \succeq \b{Y} \iff \b{X} \succeq_{\mathbb{S}_+^d} \b{Y}.
\end{align*}
It means $(\b{X} - \b{Y})$ is positive semi-definite. Note that if the inequality is strict, i.e. $\b{X} \succ \b{Y}$, it means that $(\b{X} - \b{Y})$ is positive definite.
In conclusion, $\b{x} \succeq \b{0}$ means all elements of vector $\b{x}$ are non-negative and $\b{X} \succeq \b{0}$ means the matrix $\b{X}$ is positive semi-definite. 
\end{definition}


\subsection{Preliminaries on Functions}

\begin{definition}[Fixed point]\label{definition_fixed_point}
A fixed point of a function $f(.)$ is a point $\b{x}$ which is mapped to itself by the function, i.e., $f(\b{x}) = \b{x}$.
\end{definition}

\begin{definition}[Convex function]
A function $f(.)$ with domain $\mathcal{D}$ is convex if:
\begin{align}\label{equation_convex_function}
f\big(\alpha \b{x} + (1-\alpha) \b{y}\big) \leq \alpha f(\b{x}) + (1-\alpha) f(\b{y}),
\end{align}
$\forall \b{x}, \b{y} \in \mathcal{D}$, where $\alpha \in [0,1]$. Eq. (\ref{equation_convex_function}) is depicted in Fig. \ref{figure_convex_function}-a.

Moreover, if the function $f(.)$ is differentiable, it is convex if:
\begin{align}\label{equation_convex_function_firstDerivative}
f(\b{x}) \geq f(\b{y}) + \nabla f(\b{y})^\top (\b{x} - \b{y}),
\end{align}
$\forall \b{x}, \b{y} \in \mathcal{D}$. 
Eq. (\ref{equation_convex_function_firstDerivative}) is depicted in Fig. \ref{figure_convex_function}-b.

Moreover, if the function $f(.)$ is twice differentiable, it is convex if its second-order derivative is positive semi-definite:
\begin{align}\label{equation_convex_function_secondDerivative}
\nabla^2 f(\b{x}) \succeq \b{0},
\end{align}
$\forall \b{x} \in \mathcal{D}$. 
\end{definition}
Each of the Eqs. (\ref{equation_convex_function}), (\ref{equation_convex_function_firstDerivative}), and (\ref{equation_convex_function_secondDerivative}) is a definition for the convex function. 
Note that if $\geq$ is changed to $\leq$ in Eqs. (\ref{equation_convex_function}) and (\ref{equation_convex_function_firstDerivative}) or if $\succeq$ is changed to $\preceq$ in Eq. (\ref{equation_convex_function_secondDerivative}), the function is \textit{concave}. 

\begin{figure}[!t]
\centering
\includegraphics[width=3.2in]{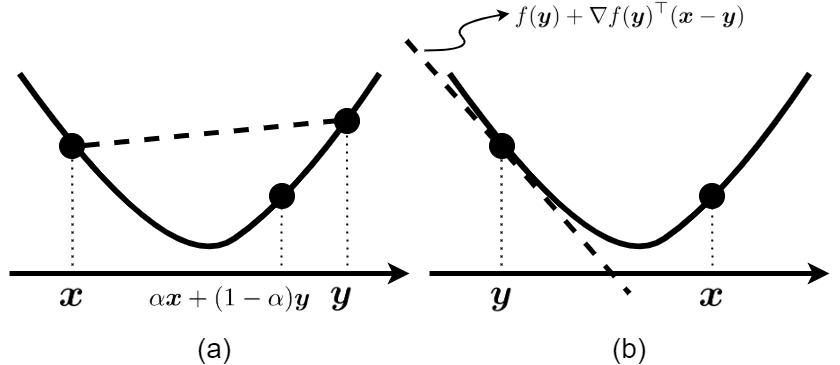}
\caption{Two definitions for the convex function: (a) Eq. (\ref{equation_convex_function}) meaning that the function value for the point $\alpha \b{x} + (1-\alpha) \b{y}$ is less than or equal to the hyper-line $\alpha f(\b{x}) + (1-\alpha) f(\b{y})$, and (b) Eq. (\ref{equation_convex_function_firstDerivative}) meaning that the function value for $\b{x}$, i.e. $f(\b{x})$, falls above the hyper-line $f(\b{y}) + \nabla f(\b{y})^\top (\b{x} - \b{y}), \forall \b{x}, \b{y} \in \mathcal{D}$.}
\label{figure_convex_function}
\end{figure}

\begin{definition}[Strongly convex function]
A differential function $f(.)$ with domain $\mathcal{D}$ is $\mu$-strongly convex if:
\begin{align}\label{equation_strongly_convex_function_firstDerivative}
f(\b{x}) \geq f(\b{y}) + \nabla f(\b{y})^\top (\b{x} - \b{y}) + \frac{\mu}{2} \|\b{x} - \b{y}\|_2^2,
\end{align}
$\forall \b{x}, \b{y} \in \mathcal{D}$ and $\mu > 0$.

Moreover, if the function $f(.)$ is twice differentiable, it is $\mu$-strongly convex if its second-order derivative is positive semi-definite:
\begin{align}\label{equation_strongly_convex_function_secondDerivative}
\b{y}^\top \nabla^2 f(\b{x}) \b{y} \geq \mu \|\b{y}\|_2^2,
\end{align}
$\forall \b{x}, \b{y} \in \mathcal{D}$ and $\mu > 0$.
A strongly convex function has a unique minimizer. See Fig. \ref{figure_stronglyConvex} for difference of convex and strongly convex functions. 
\end{definition}

\begin{figure}[!t]
\centering
\includegraphics[width=3.2in]{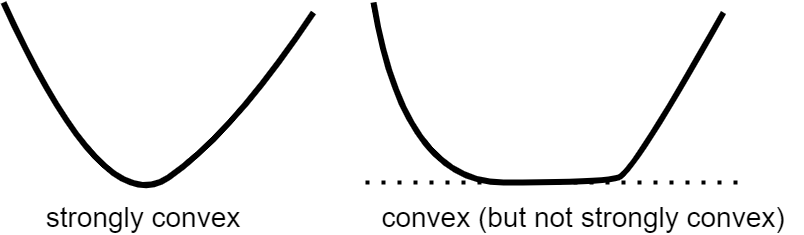}
\caption{Comparison of strongly convex and convex functions. The strongly convex function has only one strict minimizer while the convex function can have multiple minimizers with equal function values.}
\label{figure_stronglyConvex}
\end{figure}


\begin{definition}[H{\"o}lder and Lipschitz smoothness]
A function $f(.)$ with domain $\mathcal{D}$ belongs to a H{\"o}lder space $H(\alpha, L)$, with smoothness parameter $\alpha$ and the radius $L$ for ball (as the space can be seen as a ball), if:
\begin{align}
|f(\b{x}) - f(\b{y})| \leq L\, \|\b{x} - \b{y}\|_2^{\alpha}, \quad \forall \b{x}, \b{y} \in \mathcal{D}.
\end{align}
The H{\"o}lder space relates to local smoothness. A function in this space is called H{\"o}lder smooth (or H{\"o}lder continuous). 
A function is Lipschitz smooth (or Lipschitz continuous) if it is H{\"o}lder smooth with $\alpha=1$:
\begin{align}
|f(\b{x}) - f(\b{y})| \leq L\, \|\b{x} - \b{y}\|_2, \quad \forall \b{x}, \b{y} \in \mathcal{D}.
\end{align}
The parameter $L$ is called the Lipschitz constant.
A function with Lipschitz smoothness (with Lipschitz constant $L$) is called $L$-smooth. 
\end{definition}
H{\"o}lder and Lipschitz smoothness are used in many convergence and correctness proofs for optimization (e.g., see \cite{liu2021smooth}). 

The following lemma, which is based on the fundamental theorem of calculus, is widely used in proofs of optimization methods. 
\begin{lemma}[Fundamental theorem of calculus for multivariate functions]
Consider a differentiable function $f(.)$ with domain $\mathcal{D}$.
For any $\b{x}, \b{y} \in \mathcal{D}$, we have:
\begin{equation}\label{equation_fundamental_theorem_calculus}
\begin{aligned}
&f(\b{y}) = f(\b{x}) + \nabla f(\b{x})^\top (\b{y} - \b{x}) \\
&~~~~~~+ \int_0^1 \Big(\nabla f\big(\b{x} + t(\b{y} - \b{x})\big) - \nabla f(\b{x}) \Big)^\top (\b{y} - \b{x}) dt \\
&~~~~~= f(\b{x}) + \nabla f(\b{x})^\top (\b{y} - \b{x}) + o(\b{y} - \b{x}),
\end{aligned}
\end{equation}
where $o(.)$ is the small-$o$ complexity.
\end{lemma}

\begin{lemma}[Corollary of the fundamental theorem of calculus]\label{lemma_fundamental_theorem_calculus_corollary}
Consider a differentiable function $f(.)$, with domain $\mathcal{D}$, whose gradient is $L$-smooth:
\begin{align}\label{equation_gradient_L_smooth}
|\nabla f(\b{x}) - \nabla f(\b{y})| \leq L\, \|\b{x} - \b{y}\|_2, \quad \forall \b{x}, \b{y} \in \mathcal{D}.
\end{align}
For any $\b{x}, \b{y} \in \mathcal{D}$, we have:
\begin{equation}\label{equation_fundamental_theorem_calculus_Lipschitz}
\begin{aligned}
&f(\b{y}) \leq f(\b{x}) + \nabla f(\b{x})^\top (\b{y} - \b{x}) + \frac{L}{2} \|\b{y} - \b{x}\|_2^2.
\end{aligned}
\end{equation}
\end{lemma}
\begin{proof}
Proof is available in Appendix \ref{app_fundamental_theorem_calculus_corollary}.
\end{proof}

The following lemma is useful for proofs of convergence of first-order methods. 
\begin{lemma}\label{lemma_function_and_gradient_difference_bounds}
Consider a convex and differentiable function $f(.)$, with domain $\mathcal{D}$, whose gradient is $L$-smooth (see Eq. (\ref{equation_gradient_L_smooth})). We have:
\begin{align}
f(\b{y}) - f(\b{x}) \leq &\,\nabla f(\b{y})^\top (\b{y} - \b{x}) \nonumber \\
&- \frac{1}{2 L} \|\nabla f(\b{y}) - \nabla f(\b{x})\|_2^2, \label{equation_lemma_fy_fx_grady_y_x}
\end{align}
\begin{align}
& \big( \nabla f(\b{y}) - \nabla f(\b{x}) \big)^\top (\b{y} - \b{x}) \geq \frac{1}{L} \|\nabla f(\b{y}) - \nabla f(\b{x})\|_2^2. \label{equation_lemma_gradfy_gradfx_y_x}
\end{align}
\end{lemma}
\begin{proof}
Proof is available in Appendix \ref{app_function_and_gradient_difference_bounds}.
\end{proof}

\subsection{Preliminaries on Optimization}

\begin{definition}[Local and global minimizers]
A point $\b{x} \in \mathcal{D}$ is a local minimizer of function $f(.)$ if and only if:
\begin{align}\label{equation_local_minimizer}
\exists\, \epsilon > 0 : \forall \b{y} \in \mathcal{D},\, \|\b{y} - \b{x}\|_2 \leq \epsilon \implies f(\b{x}) \leq f(\b{y}),
\end{align}
meaning that in an $\epsilon$-neighborhood of $\b{x}$, the value of function is minimum at $\b{x}$.
A point $\b{x} \in \mathcal{D}$ is a global minimizer of function $f(.)$ if and only if:
\begin{align}\label{equation_global_minimizer}
f(\b{x}) \leq f(\b{y}), \quad \forall \b{y} \in \mathcal{D}.
\end{align}
See Fig. \ref{figure_StationaryPoints} for examples of local minimizer and maximizer.
\end{definition}

\begin{definition}[Strict minimizers]
In Eqs. (\ref{equation_local_minimizer}) and (\ref{equation_global_minimizer}), if we have $f(\b{x}) < f(\b{y})$ rather than $f(\b{x}) \leq f(\b{y})$, the minimizer is a strict local and global minimizer, respectively. 
See Fig. \ref{figure_StationaryPoints} for examples of strict/non-strict minimizer and maximizer.
\end{definition}

\begin{lemma}[Minimizer in convex function]\label{lemma_global_min_convex_function}
In a convex function, any local minimizer is a global minimizer. 
\end{lemma}
\begin{proof}
Proof is available in Appendix \ref{appendix_lemma_global_min_convex_function}.
\end{proof}

\begin{corollary}\label{corollary_only_one_min_convex_function}
In a convex function, there exists only one local minimizer which is the global minimizer. As an imagination, a convex function is like a multi-dimensional bowl with only one minimizer. 
\end{corollary}

\begin{lemma}[Gradient of a convex function at the minimizer point]\label{lemma_minimizer_gradient_zero}
When the function $f(.)$ is convex and differentiable, a point $\b{x}^*$ is a minimizer if and only if $\nabla f(\b{x}^*) = \b{0}$.
\end{lemma}
\begin{proof}
Proof is available in Appendix \ref{appendix_lemma_minimizer_gradient_zero}.
\end{proof}

\begin{definition}[Stationary, extremum, and saddle points]
In a general (not-necessarily-convex) function $f(.)$, a point $\b{x}^*$ is a stationary if and only if $\nabla f(\b{x}^*) = \b{0}$. 
By passing through a saddle point, the sign of the second derivative flips to the opposite sign.  
Minimizer and maximizer points (locally or globally) minimize and maximize the function, respectively. A saddle point is neither minimizer nor maximizer, although the gradient at a saddle point is zero. 
Both minimizer and maximizer are also called the extremum points. 
As Fig. \ref{figure_StationaryPoints} shows, some of  stationary point can be either a minimizer, a maximizer, or a saddle point of function.
\end{definition}

\begin{lemma}[First-order optimality condition {\citep[Theorem 1.2.1]{nesterov2018lectures}}]\label{lemma_first_order_optimality_condition}
If $\b{x}^*$ is a local minimizer for a differentiable function $f(.)$, then:
\begin{align}\label{equation_first_order_optimality_condition}
\nabla f(\b{x}^*) = \b{0}.
\end{align}
Note that if $f(.)$ is convex, this equation is a necessary and sufficient condition for a minimizer. 
\end{lemma}
\begin{proof}
Proof is available in Appendix \ref{app_first_order_optimality_condition}.
\end{proof}
Note that if setting the derivative to zero, i.e. Eq. (\ref{equation_first_order_optimality_condition}), gives a closed-form solution for $\b{x}^*$, the optimization is done. Otherwise, one should start with some random initialized solution and iteratively update it using the gradient. First-order or second-order methods can be used for iterative optimization (see Sections \ref{section_first_order_methods} and \ref{section_second_order_methods}). 

\begin{definition}[Arguments of minimization and maximization]
In the domain of function, the point which minimizes (resp. maximizes) the function $f(.)$ is the argument for the minimization (resp. maximization) of function. The minimizer and maximizer of function are denoted by $\arg\min_{\b{x}} f(\b{x})$ and $\arg\max_{\b{x}} f(\b{x})$, respectively. 
\end{definition}

\begin{remark}
We can convert convert maximization to minimization and vice versa: 
\begin{equation}\label{equation_max_min_conversion}
\begin{aligned}
& \underset{\b{x}}{\text{maximize}}\,\,\, f(\b{x}) = -\,\underset{\b{x}}{\text{minimize}}\,\,\, \big(\!\!-\!f(\b{x})\big), \\
& \underset{\b{x}}{\text{minimize}}\,\,\, f(\b{x}) = -\,\underset{\b{x}}{\text{maximize}}\,\,\, \big(\!\!-\!f(\b{x})\big).
\end{aligned}
\end{equation}
We can have similar conversions for the arguments of maximization and minimization but we the sign of optimal value of function is not important in argument, we do not have the negative sign before maximization and minimization:
\begin{equation}
\begin{aligned}
& \arg\max_{\b{x}} f(\b{x}) = \arg\min_{\b{x}} \big(\!\!-\!f(\b{x})\big), \\
& \arg\min_{\b{x}} f(\b{x}) = \arg\max_{\b{x}} \big(\!\!-\!f(\b{x})\big).
\end{aligned}
\end{equation}
\end{remark}

\begin{definition}[Convergence rates]\label{definition_convergence_rate}
If any sequence $\{\epsilon_0, \epsilon_1, \dots, \epsilon_{k}, \epsilon_{k+1}, \dots\}$ converges, its convergence rate has one of the following cases:
\begin{align}
\lim_{k \rightarrow \infty} \frac{\epsilon_{k+1}}{\epsilon_k} = 
\left\{
    \begin{array}{ll}
        0 & \mbox{superlinear rate}, \\
        \in (0,1) & \mbox{linear rate}, \\
        1 & \mbox{sublinear rate}.
    \end{array}
\right.
\end{align}
\end{definition}

\subsection{Preliminaries on Derivative}

\begin{remark}[Dimensionality of derivative]\label{remark_dimensionality_of_derivative}
Consider a function $f: \mathbb{R}^{d_1} \rightarrow \mathbb{R}^{d_2}$, $f: \b{x} \mapsto f(\b{x})$.
Derivative of function $f(\b{x}) \in \mathbb{R}^{d_2}$ with respect to (w.r.t.) $\b{x} \in \mathbb{R}^{d_1}$ has dimensionality $(d_1 \times d_2)$. This is because tweaking every element of $\b{x} \in \mathbb{R}^{d_1}$ can change every element of $f(\b{x}) \in \mathbb{R}^{d_2}$. The $(i,j)$-th element of the $(d_1 \times d_2)$-dimensional derivative states the amount of change in the $j$-th element of $f(\b{x})$ resulted by changing the $i$-th element of $\b{x}$. 

Note that one can use a transpose of the derivative as the derivative. This is okay as long as the dimensionality of other terms in equations of optimization coincide (i.e., they are all transposed). In that case, the dimensionality of derivative is $(d_2 \times d_1)$ where the $(i,j)$-th element of derivative states the amount of change in the $i$-th element of $f(\b{x})$ resulted by changing the $j$-th element of $\b{x}$. 

Some examples of derivatives are as follows. 
\begin{itemize}\setlength\itemsep{0.01em}
\item If the function is $f: \mathbb{R} \rightarrow \mathbb{R}, f: x \mapsto f(x)$, the derivative $(\partial f(x) / \partial x) \in \mathbb{R}$ is a scalar because changing the scalar $\b{x}$ can change the scalar $f(\b{x})$.
\item If the function is $f: \mathbb{R}^{d} \rightarrow \mathbb{R}, f: \b{x} \mapsto f(\b{x})$, the derivative $(\partial f(\b{x}) / \partial \b{x}) \in \mathbb{R}^d$ is a vector because changing every element of the vector $\b{x}$ can change the scalar $f(\b{x})$.
\item If the function is $f: \mathbb{R}^{d_1 \times d_2} \rightarrow \mathbb{R}, f: \b{X} \mapsto f(\b{X})$, the derivative $(\partial f(\b{X}) / \partial \b{X}) \in \mathbb{R}^{d_1 \times d_2}$ is a matrix because changing every element of the matrix $\b{X}$ can change the scalar $f(\b{X})$.
\item If the function is $f: \mathbb{R}^{d_1} \rightarrow \mathbb{R}^{d_2}, f: \b{x} \mapsto f(\b{x})$, the derivative $(\partial f(\b{x}) / \partial \b{x}) \in \mathbb{R}^{d_1 \times d_2}$ is a matrix because changing every element of the vector $\b{x}$ can change every element of the vector $f(\b{x})$.
\item If the function is $f: \mathbb{R}^{d_1 \times d_2} \rightarrow \mathbb{R}^{d_3}, f: \b{X} \mapsto f(\b{X})$, the derivative $(\partial f(\b{X}) / \partial \b{X})$ is a $(d_1 \times d_2 \times d_3)$-dimensional tensor because changing every element of the matrix $\b{X}$ can change every element of the vector $f(\b{X})$.
\item If the function is $f: \mathbb{R}^{d_1 \times d_2} \rightarrow \mathbb{R}^{d_3 \times d_4}, f: \b{X} \mapsto f(\b{X})$, the derivative $(\partial f(\b{X}) / \partial \b{X})$ is a $(d_1 \times d_2 \times d_3 \times d_4)$-dimensional tensor because changing every element of the matrix $\b{X}$ can change every element of the matrix $f(\b{X})$.
\end{itemize}
In other words, the derivative of a scalar w.r.t. a scalar is a scalar.
The derivative of a scalar w.r.t. a vector is a vector.
The derivative of a scalar w.r.t. a matrix is a matrix.
The derivative of a vector w.r.t. a vector is a matrix.
The derivative of a vector w.r.t. a matrix is a rank-3 tensor.
The derivative of a matrix w.r.t. a matrix is a rank-4 tensor.
\end{remark}

\begin{definition}[Gradient, Jacobian, and Hessian]
Consider a function $f: \mathbb{R}^{d} \rightarrow \mathbb{R}$, $f: \b{x} \mapsto f(\b{x})$. In optimizing the function $f$, the derivative of function w.r.t. its variable $\b{x}$ is called the gradient, denoted by:
\begin{align*}
\nabla f(\b{x}) := \frac{\partial f(\b{x})}{\partial \b{x}} \in \mathbb{R}^d.
\end{align*}
The second derivative of function w.r.t. to its derivative is called the Hessian matrix, denoted by 
\begin{align*}
\b{B} = \nabla^2 f(\b{x}) := \frac{\partial^2 f(\b{x})}{\partial \b{x}^2} \in \mathbb{R}^{d \times d}. 
\end{align*}
The Hessian matrix is symmetric. If the function is convex, its Hessian matrix is positive semi-definite. 

If the function is multi-dimensional, i.e., $f: \mathbb{R}^{d_1} \rightarrow \mathbb{R}^{d_2}$, $f: \b{x} \mapsto f(\b{x})$, the gradient becomes a matrix:
\begin{align*}
& \b{J} := \Big[\frac{\partial f}{\partial x_1}, \dots, \frac{\partial f}{\partial x_{d_1}}\Big]^\top\! = 
\begin{bmatrix}
\frac{\partial f_1}{\partial x_1} & \dots & \frac{\partial f_{d_2}}{\partial x_{d_1}}\\
\vdots & \ddots & \vdots \\
\frac{\partial f_{1}}{\partial x_{d_1}} & \dots & \frac{\partial f_{d_2}}{\partial x_{d_1}}
\end{bmatrix}
\!\!\in \mathbb{R}^{d_1 \times d_2},
\end{align*}
where $\b{x} = [x_1, \dots, x_{d_1}]^\top$ and $f(\b{x}) = [f_1, \dots, f_{d_2}]^\top$. 
This matrix derivative is called the Jacobian matrix. 
\end{definition}

\begin{corollary}[Technique for calculating derivative]
According to the size of derivative, we can easily calculate the derivatives. For finding the correct derivative for multiplications of matrices (or vectors), one can temporarily assume some dimensionality for every matrix and find the correct of matrices in the derivative. 
Let $\b{X} \in \mathbb{R}^{a \times b}$,
An example for calculating derivative is:
\begin{align}\label{equation_derivative_trace_wrt_matrix}
& \mathbb{R}^{a \times b} \ni \frac{\partial}{\partial \b{X}} \big(\textbf{tr}(\b{A}\b{X}\b{B})\big) = \b{A}^\top \b{B}^\top = (\b{BA})^\top.
\end{align}
This is calculated as explained in the following. We assume $\b{A} \in \mathbb{R}^{c \times a}$ and $\b{B} \in \mathbb{R}^{b \times c}$ so that we can have the matrix multiplication $\b{A}\b{X}\b{B}$ and its size is $\b{A}\b{X}\b{B} \in \mathbb{R}^{c \times c}$ because the argument of trace should be a square matrix. 
The derivative $\partial (\textbf{tr}(\b{A}\b{X}\b{B})) / \partial \b{X}$ has size $\mathbb{R}^{a \times b}$ because $\textbf{tr}(\b{A}\b{X}\b{B})$ is a scalar and $\b{X}$ is $(a \times b)$-dimensional. 
We know that the derivative should be a kind of multiplication of $\b{A}$ and $\b{B}$ because $\textbf{tr}(\b{A}\b{X}\b{B})$ is linear w.r.t. $\b{X}$. Now, we should find their order in multiplication. 
Based on the assumed sizes of $\b{A}$ and $\b{B}$, we see that $\b{A}^\top \b{B}^\top$ is the desired size and these matrices can be multiplied to each other. Hence, this is the correct derivative. 
\end{corollary}

\begin{lemma}[Derivative of matrix w.r.t. matrix]\label{lemma_deriavtive_matrix_wrt_matrix}
As explained in Remark \ref{remark_dimensionality_of_derivative}, the derivative of a matrix w.r.t. another matrix is a tensor. 
Working with tensors is difficult; hence, we can use Kronecker product for representing tensor as matrix.
This is the Magnus-Neudecker convention \cite{magnus1985matrix} in which all matrices are vectorized.
For example, if $\b{X} \in \mathbb{R}^{a \times b}$, $\b{A} \in \mathbb{R}^{c \times a}$, and $\b{B} \in \mathbb{R}^{b \times d}$, we have:
\begin{align}\label{equation_derivative_matrix_wrt_matrix}
\mathbb{R}^{(cd) \times (ab)} \ni \frac{\partial}{\partial \b{X}} (\b{A} \b{X} \b{B}) = \b{B}^\top \otimes \b{A},
\end{align}
where $\otimes$ denotes the Kronecker product. 
\end{lemma}

\begin{remark}[Chain rule in matrix derivatives]
When having composite functions (i.e., function of function), we use chain rule for derivative. When we have derivative of matrix w.r.t. matrix, this chain rule can get difficult but we can do it by checking compatibility of dimensions in matrix multiplications. We should use Lemma \ref{lemma_deriavtive_matrix_wrt_matrix} and vectorization technique in which the matrix is vectorized. Let $\textbf{vec}(.)$ denote vectorization of a $\mathbb{R}^{a \times b}$ matrix to a $\mathbb{R}^{ab}$ vector. Also, let $\textbf{vec}^{-1}_{a \times b}(.)$ be de-vectorization of a $\mathbb{R}^{ab}$ vector to a $\mathbb{R}^{a \times b}$ matrix.

For the purpose of tutorial, here we calculate derivative by chain rule as an example:
\begin{align*}
& f(\b{S}) = \textbf{tr}(\b{A}\b{S}\b{B}),\,\, \b{S} = \b{C}\widehat{\b{M}}\b{D},\,\, \widehat{\b{M}} = \frac{\b{M}}{\|\b{M}\|_F^2},
\end{align*}
where $\b{A} \in \mathbb{R}^{c \times a}$, $\b{S} \in \mathbb{R}^{a \times b}$, $\b{B} \in \mathbb{R}^{b \times c}$, $\b{C} \in \mathbb{R}^{a \times d}$, $\widehat{\b{M}} \in \mathbb{R}^{d \times d}$, $\b{D} \in \mathbb{R}^{d \times b}$, and $\b{M} \in \mathbb{R}^{d \times d}$. 
We have:
\begin{align*}
&\mathbb{R}^{a \times b} \ni \frac{\partial f(\b{S})}{\partial \b{S}} \overset{(\ref{equation_derivative_trace_wrt_matrix})}{=} (\b{BA})^\top. \\
&\mathbb{R}^{ab \times d^2} \ni \frac{\partial \b{S}}{\partial \widehat{\b{M}}} \overset{(\ref{equation_derivative_matrix_wrt_matrix})}{=} \b{D}^\top \otimes \b{C}, \\
&\mathbb{R}^{d^2 \times d^2} \ni \frac{\partial \widehat{\b{M}}}{\partial \b{M}} \overset{(a)}{=} \frac{1}{\|\b{M}\|_F^4} \big(\|\b{M}\|_F^2 \b{I}_{d^2} - 2 \b{M} \otimes \b{M}\big) \\
&~~~~~~~~~~~~~~~~~~~~~~~~~~~ = \frac{1}{\|\b{M}\|_F^2} \big(\b{I}_{d^2} - \frac{2}{\|\b{M}\|_F^2} \b{M} \otimes \b{M}\big),
\end{align*}
where $(a)$ is because of the formula for the derivative of fraction and $\b{I}_{d^2}$ is a $(d^2 \times d^2)$-dimensional identity matrix.
finally, by chain rule, we have:
\begin{align*}
& \mathbb{R}^{d \times d} \ni \frac{\partial f}{\b{M}} = \textbf{vec}^{-1}_{d \times d}\Big( \big(\frac{\partial \widehat{\b{M}}}{\partial \b{M}}\big)^\top \big(\frac{\partial \b{S}}{\partial \widehat{\b{M}}}\big)^\top  \textbf{vec}\big(\frac{\partial f(\b{S})}{\partial \b{S}}\big)\Big).
\end{align*}
Note that the chain rule in matrix derivatives usually is stated right to left in matrix multiplications while transpose is used for matrices in multiplication.
\end{remark}

More formulas for matrix derivatives can be found in the matrix cookbook \cite{petersen2012matrix} and similar resources.
Here, we discussed only real derivatives. When working with complex data (with imaginary part), we need complex derivative. The reader can refer to \cite{hjorungnes2007complex} and {\citep[Chapter 7, Complex Derivatives]{chong2021MH}} for techniques in complex derivatives. 

\section{Optimization Problems}

\subsection{Standard Problems}


Here, we review the standard forms for convex optimization and we explain why these forms are important. 
Note that the term ``programming" refers to solving optimization problems.

\subsubsection{General Optimization Problem}

Consider the function $f: \mathbb{R}^d \rightarrow \mathbb{R}$, $f: \b{x} \mapsto f(\b{x})$. Let the domain of function be $\mathcal{D}$ where $\b{x} \in \mathcal{D}, \b{x} \in \mathbb{R}^d$. 

Consider the following unconstrained minimization of a cost function $f(.)$:
\begin{equation}\label{equation_optimization_problem_unconstrained}
\begin{aligned}
& \underset{\b{x}}{\text{minimize}}
& & f(\b{x}),
\end{aligned}
\end{equation}
where $\b{x}$ is called the \textit{optimization variable} and the function $f(.)$ is called the \textit{objective function} or the \textit{cost function}.
This is an unconstrained problem where the optimization variable $\b{x}$ needs only be in the domain of function, i.e., $\b{x} \in \mathcal{D}$, while minimizing the function $f(.)$. 

The optimization problem can be constrained where the optimization variable $\b{x}$ should satisfy some equality and/or inequality constraints, in addition to being in the domain of function, while minimizing the function $f(.)$.
Consider a constrained optimization problem where we want to minimize the function $f(\b{x})$ while satisfying $m_1$ inequality constraints and $m_2$ equality constraint:
\begin{equation}\label{equation_optimization_problem}
\begin{aligned}
& \underset{\b{x}}{\text{minimize}}
& & f(\b{x}) \\
& \text{subject to}
& & y_i(\b{x}) \leq 0, \; i \in \{1, \ldots, m_1\}, \\
& & & h_i(\b{x}) = 0, \; i \in \{1, \ldots, m_2\},
\end{aligned}
\end{equation}
where $f(\b{x})$ is the objective function, every $y_i(\b{x}) \leq 0$ is an inequality constraint, and every $h_i(\b{x}) = 0$ is an equality constraint. 
Note that if some of the inequality constraints are not in the form $y_i(\b{x}) \leq 0$, we can restate them as:
\begin{align*}
& y_i(\b{x}) \geq 0 \implies -y_i(\b{x}) \leq 0, \\
& y_i(\b{x}) \leq c \implies y_i(\b{x}) - c \leq 0. 
\end{align*}
Therefore, all inequality constraints can be written in the form $y_i(\b{x}) \leq 0$. 
Furthermore, according to Eq. (\ref{equation_max_min_conversion}), if the optimization problem (\ref{equation_optimization_problem}) is a maximization problem rather than minimization, we can convert it to maximization by multiplying its objective function to $-1$:
\begin{equation}\label{equation_convert_maximize_to_minimize}
\begin{aligned}
& \underset{\b{x}}{\text{maximize}} ~~~~~ f(\b{x}) \\
& \text{subject to constraints}
\end{aligned}
\quad\equiv\quad\,\,\\
\begin{aligned}
& \underset{\b{x}}{\text{minimize}} ~~~~~ -f(\b{x}) \\
& \text{subject to constraints}
\end{aligned}
\end{equation}

\begin{definition}[Feasible point]
The point $\b{x}$ for the optimization problem (\ref{equation_optimization_problem}) is feasible if:
\begin{equation}\label{equation_feasible_point}
\begin{aligned}
& \b{x} \in \mathcal{D}, \text{  and} \\
& y_i(\b{x}) \leq 0, \quad \forall i \in \{1, \ldots, m_1\}, \text{  and} \\
& h_i(\b{x}) = 0, \quad \forall i \in \{1, \ldots, m_2\}.
\end{aligned}
\end{equation}
\end{definition}

The constrained optimization problem can also be stated as:
\begin{equation}\label{equation_constraint_set_optimization_problem}
\begin{aligned}
& \underset{\b{x}}{\text{minimize}}
& & f(\b{x}) \\
& \text{subject to}
& & \b{x} \in \mathcal{S},
\end{aligned}
\end{equation}
where $\mathcal{S}$ is the feasible set of constraints. 

\subsubsection{Convex Optimization Problem}

A convex optimization problem is of the form:
\begin{equation}
\begin{aligned}
& \underset{\b{x}}{\text{minimize}}
& & f(\b{x}) \\
& \text{subject to}
& & y_i(\b{x}) \leq 0, \; i \in \{1, \ldots, m_1\}, \\
& & & \b{Ax} = \b{b},
\end{aligned}
\end{equation}
where the functions $f(.)$ and $y_i(.), \forall i$ are all convex functions and the equality constraints are affine functions. The feasible set of a convex problem is a convex set. 

\subsubsection{Linear Programming}

A linear programming problem is of the form:
\begin{equation}
\begin{aligned}
& \underset{\b{x}}{\text{minimize}}
& & \b{c}^\top \b{x} + d \\
& \text{subject to}
& & \b{Gx} \preceq \b{h}, \\
& & & \b{Ax} = \b{b},
\end{aligned}
\end{equation}
where the objective function and equality constraints are affine functions. 
The feasible set of a linear programming problem is a a polyhedron set while the cost is planar (affine). 
A survey on linear programming methods is available in the book \cite{dantzig1963linear}. 
One of the well-known methods for solving linear programming is the \textit{simplex method}, initially appeared in 1947 \cite{dantzig1983reminiscences}. 
Simplex method moves between the vertices of a simplex, until convergence, for minimizing the objective function.
It is efficient and its proposal was a breakthrough in the field of optimization.

\subsubsection{Quadratic Programming}

A quadratic programming problem is of the form:
\begin{equation}
\begin{aligned}
& \underset{\b{x}}{\text{minimize}}
& & (1/2)\b{x}^\top \b{P} \b{x} + \b{q}^\top \b{x} + r \\
& \text{subject to}
& & \b{Gx} \preceq \b{h}, \\
& & & \b{Ax} = \b{b},
\end{aligned}
\end{equation}
where $\b{P} \succ \b{0}$ (which is the second derivative of objective function) is a symmetric positive definite matrix, the objective function is quadratic, and equality constraints are affine functions. 
The feasible set of a quadratic programming problem is a a polyhedron set while the cost is curvy (quadratic).

\subsubsection{Quadratically Constrained Quadratic Programming (QCQP)}

A QCQP problem is of the form:
\begin{equation}
\begin{aligned}
& \underset{\b{x}}{\text{minimize}}~~~~ (1/2)\b{x}^\top \b{P} \b{x} + \b{q}^\top \b{x} + r \\
& \text{subject to} \\
&~~~~~ (1/2)\b{x}^\top \b{M}_i \b{x} + \b{s}_i^\top \b{x} + z_i \leq 0, \; i \in \{1, \ldots, m_1\},  \\
&~~~~~ \b{Ax} = \b{b},
\end{aligned}
\end{equation}
where $\b{P}, \b{M}_i \succ \b{0}, \forall i$, the objective function and the inequality constraints are quadratic, and equality constraints are affine functions. 
The feasible set of a QCQP problem is intersection of $m_1$ ellipsoids and an affine set, while the cost is curvy (quadratic).

\subsubsection{Second-Order Cone Programming (SOCP)}

A SOCP problem is of the form:
\begin{equation}
\begin{aligned}
& \underset{\b{x}}{\text{minimize}}
& & \b{f}^\top \b{x} \\
& \text{subject to}
& & \|\b{A}_i \b{x} + \b{b}_i\|_2 \leq \b{c}_i^\top \b{x} + d_i, \; i \in \{1, \ldots, m_1\}, \\
& & & \b{Fx} = \b{g},
\end{aligned}
\end{equation}
where the inequality constraints are norm of an affine function being less than an affine function. The constraint $\|\b{A}_i \b{x} + \b{b}_i\|_2 - \b{c}_i^\top \b{x} - d_i \leq 0$ is called the \textit{second-order cone} whose shape is like an ice-cream cone. 

\subsubsection{Semidefinite Programming (SDP)}

A SDP problem is of the form:
\begin{equation}
\begin{aligned}
& \underset{\b{X}}{\text{minimize}}
& & \textbf{tr}(\b{C} \b{X}) \\
& \text{subject to}
& & \b{X} \succeq \b{0}, \\
& & & \textbf{tr}(\b{D}_i \b{X}) \leq \b{e}_i, \quad i \in \{1, \dots, m_1\}, \\
& & & \textbf{tr}(\b{A}_i \b{X}) = \b{b}_i, \quad i \in \{1, \dots, m_2\}, 
\end{aligned}
\end{equation}
where the optimization variable $\b{X}$ belongs to the positive semidefinite cone $\mathbb{S}_+^d$, $\textbf{tr}(.)$ denotes the trace of matrix, $\b{C}, \b{D}_i, \b{A}_i \in \mathbb{S}^d, \forall i$, and $\mathbb{S}^d$ denotes the cone of $(d \times d)$ symmetric matrices. 
The trace terms may be written in summation forms.
Note that $\textbf{tr}(\b{C}^\top \b{X})$ is the inner product of two matrices $\b{C}$ and $\b{X}$ and if the matrix $\b{C}$ is symmetric, this inner product is equal to $\textbf{tr}(\b{C} \b{X})$.  

Another form for SDP is: 
\begin{equation}
\begin{aligned}
& \underset{\b{x}}{\text{minimize}}
& & \b{c}^\top \b{x} \\
& \text{subject to}
& & \Big(\sum_{i=1}^d x_i \b{F}_i\Big) + \b{G} \preceq \b{0}, \\
& & & \b{Ax} = \b{b},
\end{aligned}
\end{equation}
where $\b{x} = [x_1, \dots, x_d]^\top$, $\b{G}, \b{F}_i \in \mathbb{S}^d, \forall i$, and $\b{A}$, $\b{b}$, and $\b{c}$ are constant matrices/vectors.

\subsubsection{Optimization Toolboxes}

All the standard optimization forms can be restated as SDP because their constraints can be written as belonging to some cones (see Definitions \ref{definition_generalized_inequality} and \ref{definition_generalized_inequality_examples}); hence, they are special cases of SDP.
The interior-point method, or the barrier method, introduced in Section \ref{section_interior_point_method}, can be used for solving various optimization problems including SDP \cite{nesterov1994interior,boyd2004convex}. Optimization toolboxes such as CVX \cite{grant2009cvx} often use interior-point method (see Section \ref{section_interior_point_method}) for solving optimization problems such as SDP. Note that the interior-point method is iterative and solving SDP usually is time consuming especially for large matrices. 
If the optimization problem is a convex optimization problem (e.g. SDP is a convex problem), it has only one local optima which is the global optima (see Corollary \ref{corollary_only_one_min_convex_function}).

\subsection{Eliminating Constraints and Equivalent Problems}

Here, we review some of the useful techniques in converting optimization problems to their equivalent forms. 

\subsubsection{Eliminating Inequality Constraints}

As was discussed in Section \ref{section_interior_point_method}, we can eliminate the inequality constraints by embedding the inequality constraints into the objective function using the indicator or barrier functions.

\subsubsection{Eliminating Equality Constraints}

Consider the optimization problem (\ref{equation_optimization_problem_inequalityConstraints_Ax_equals_b_constraint}). We can eliminate the equality constraints, $\b{Ax} = \b{b}$, as explained in the following. 
Let $\b{A} \in \mathbb{R}^{m_2 \times d}$, $m_2 < d$, $\mathcal{N}(\b{A}) := \{\b{x} \in \mathbb{R}^d\,|\,\b{Ax} = \b{0}\}$ denote the null-space of matrix $\b{A}$. We have:
\begin{align*}
& \forall \b{z} \in \mathcal{N}(\b{A}), \exists \b{u} \in \mathbb{R}^{d - m_2}, \b{C} \in \mathbb{R}^{m_2 \times (d-m_2)}: \\
& \mathbb{C}\text{ol}(\b{C}) = \mathcal{N}(\b{A}), \quad \b{z} = \b{C} \b{u},
\end{align*}
where $\mathbb{C}\text{ol}(.)$ is the column-space or range of matrix. 
Therefore, we can say:
\begin{align}
& \forall \b{z} \in \mathcal{N}(\b{A}): \b{A}(\b{x} - \b{z}) = \b{A}\b{x} - \b{A}\b{z} = \b{A}\b{x} - \b{0} = \b{Ax} \nonumber\\
&\implies \b{A}(\b{x} - \b{z}) = \b{Ax} = \b{b} \nonumber\\
&\implies \b{x} = \b{A}^\dagger \b{b} + \b{z} = \b{A}^\dagger \b{b} + \b{C} \b{u}, \label{equation_x_eliminateEqualityConstraint}
\end{align}
where $\b{A}^\dagger := \b{A}^\top (\b{A}\b{A}^\top)^{-1}$ is the pseudo-inverse of matrix $\b{A}$. 
Putting Eq. (\ref{equation_x_eliminateEqualityConstraint}) in problem (\ref{equation_optimization_problem_inequalityConstraints_Ax_equals_b_constraint}) changes the optimization variable and eliminates the equality constraint: 
\begin{equation}
\begin{aligned}
& \underset{\b{u}}{\text{minimize}}
& & f(\b{C} \b{u} + \b{A}^\dagger \b{b}) \\
& \text{subject to}
& & y_i(\b{C} \b{u} + \b{A}^\dagger \b{b}) \leq 0, \; i \in \{1, \ldots, m_1\}.
\end{aligned}
\end{equation}
If $\b{u}^*$ is the solution to this problem, the solution to problem (\ref{equation_optimization_problem_inequalityConstraints_Ax_equals_b_constraint}) is $\b{x}^* = \b{C} \b{u}^* + \b{A}^\dagger \b{b}$.

\subsubsection{Adding Equality Constraints}

Conversely, we can convert the problem:
\begin{equation}
\begin{aligned}
& \underset{\{\b{x}_i\}_{i=0}^{m_1}}{\text{minimize}}
& & f(\b{A}\b{x}_0 + \b{b}_0) \\
& \text{subject to}
& & y_i(\b{A}\b{x}_i + \b{b}_i) \leq 0, \; i \in \{1, \ldots, m_1\},
\end{aligned}
\end{equation}
to:
\begin{equation}
\begin{aligned}
& \underset{\{\b{u}_i, \b{x}_i\}_{i=0}^{m_1}}{\text{minimize}}
& & f(\b{u}_0) \\
& \text{subject to}
& & y_i(\b{u}_i) \leq 0, \; i \in \{1, \ldots, m_1\}, \\
& & & \b{u}_i = \b{A}\b{x}_i + \b{b}_i,\; i \in \{0, 1, \ldots, m_1\},
\end{aligned}
\end{equation}
by change of variables. 

\subsubsection{Eliminating Set Constraints}

As was discussed in Section \ref{section_projected_gradient_method}, we can convert problem (\ref{equation_constraint_set_optimization_problem}) to problem (\ref{equation_constraint_set_optimization_problem_withIndicator}) by using the indicator function. That problem can be solved iteratively where at every iteration, the solution is updated (by first- or second-order methods) without the set constraint and then the updated solution of iteration is projected onto the set. This procedure is repeated until convergence. 

\subsubsection{Adding Slack Variables}

Consider the following problem with inequality constraints:
\begin{equation}
\begin{aligned}
& \underset{\b{x}}{\text{minimize}}
& & f(\b{x}) \\
& \text{subject to}
& & y_i(\b{x}) \leq 0, \; i \in \{1, \ldots, m_1\}.
\end{aligned}
\end{equation}
Using the so-called slack variables, denoted by $\{\xi_i \in \mathbb{R}_+\}_{i=1}^{m_1}$, we can convert this problem to the following problem:
\begin{equation}
\begin{aligned}
& \underset{\b{x}, \{\xi_i\}_{i=1}^{m_1}}{\text{minimize}}
& & f(\b{x}) \\
& \text{subject to}
& & y_i(\b{x}) + \xi_i = 0, \; i \in \{1, \ldots, m_1\}, \\
& & & \xi_i \geq 0, \quad\qquad\,\,\, \; i \in \{1, \ldots, m_1\}.
\end{aligned}
\end{equation}
The slack variables should be non-negative because the inequality constraints are less than or equal to zero. 

\subsubsection{Epigraph Form}

We can convert the optimization problem (\ref{equation_optimization_problem}) to its \textit{epigraph form}:
\begin{equation}
\begin{aligned}
& \underset{\b{x}, t}{\text{minimize}}
& & t \\
& \text{subject to}
& & f(\b{x}) - t \leq 0, \\
& & & y_i(\b{x}) \leq 0, \; i \in \{1, \ldots, m_1\}, \\
& & & h_i(\b{x}) = 0, \; i \in \{1, \ldots, m_2\},
\end{aligned}
\end{equation}
because we can minimize an upper-bound $t$ on the objective function rather than minimizing the objective function.
Likewise, for a maximization problem, we can maximize a lower-bound of the objective function rather than maximizing the objective function.
The upper-/lower-bound does not necessarily need to be $t$; it can be any upper-/lower-bound function for the objective function.
This is a good technique because sometimes optimizing an upper-/lower-bound function is simpler than the objective function itself.

\section{Karush-Kuhn-Tucker (KKT) Conditions}

Many of the optimization algorithms are reduced to and can be explained by the Karush-Kuhn-Tucker (KKT) conditions. Therefore, KKT conditions are fundamental requirements for optimization.
In this section, we explain these conditions. 

\subsection{The Lagrangian Function}

\subsubsection{Lagrangian and Dual Variables}

\begin{definition}[Lagrangian and dual variables]
The Lagrangian function for the optimization problem (\ref{equation_optimization_problem}) is $\mathcal{L}: \mathbb{R}^d \times \mathbb{R}^{m_1} \times \mathbb{R}^{m_2} \rightarrow \mathbb{R}$, with domain $\mathcal{D} \times \mathbb{R}^{m_1} \times \mathbb{R}^{m_2}$, defined as:
\begin{equation}\label{equation_Lagrangian}
\begin{aligned}
\mathcal{L}(\b{x},\b{\lambda},\b{\nu}) &:= f(\b{x}) + \sum_{i=1}^{m_1} \lambda_i y_i(\b{x}) + \sum_{i=1}^{m_2} \nu_i h_i(\b{x}) \\
&= f(\b{x}) + \b{\lambda}^\top \b{y}(\b{x}) + \b{\nu}^\top \b{h}(\b{x}),
\end{aligned}
\end{equation}
where $\{\lambda_i\}_{i=1}^{m_1}$ and $\{\nu_i\}_{i=1}^{m_2}$ are the Lagrange multipliers, also called the dual variables, corresponding to inequality and equality constraints, respectively. Note that $\b{\lambda} := [\lambda_1, \dots, \lambda_{m_1}]^\top \in \mathbb{R}^{m_1}$, $\b{\nu} := [\nu_1, \dots, \nu_{m_2}]^\top \in \mathbb{R}^{m_2}$, $\b{y}(\b{x}) := [y_1(\b{x}), \dots, y_{m_1}(\b{x})]^\top \in \mathbb{R}^{m_1}$, and $\b{h}(\b{x}) := [h_1(\b{x}), \dots, h_{m_2}(\b{x})]^\top \in \mathbb{R}^{m_2}$. 
Eq. (\ref{equation_Lagrangian}) is also called the Lagrange relaxation of the optimization problem (\ref{equation_optimization_problem}). 
\end{definition}

\subsubsection{Sign of Terms in Lagrangian}

In some papers, the plus sign behind $\sum_{i=1}^{m_2} \nu_i h_i(\b{x})$ is replaced with the negative sign. As $h_i(\b{x})$ is for equality constraint, its sign is not important in the Lagrangian function. However, the sign of the term $\sum_{i=1}^{m_1} \lambda_i y_i(\b{x})$ is important because the sign of inequality constraint is important. We will discuss the sign of $\{\lambda_i\}_{i=1}^{m_1}$ later.
Moreover, according to Eq. (\ref{equation_convert_maximize_to_minimize}), if the problem (\ref{equation_dual_optimization_problem}) is a maximization problem rather than minimization, the Lagrangian function is $\mathcal{L}(\b{x},\b{\lambda},\b{\nu}) = -f(\b{x}) + \sum_{i=1}^{m_1} \lambda_i y_i(\b{x}) + \sum_{i=1}^{m_2} \nu_i h_i(\b{x})$ instead of Eq. (\ref{equation_Lagrangian}).

\subsubsection{Interpretation of Lagrangian}\label{section_Lagrangian_interpretation}

We can interpret Lagrangian using penalty. As Eq. (\ref{equation_optimization_problem}) states, we want to minimize the objective function $f(\b{x})$. 
We create a cost function consisting of the objective function. 
The optimization problem has constraints so its constraints should also be satisfied while minimizing the objective function. Therefore, we penalize the cost function if the constraints are not satisfied. For this, we can add the constraints to the objective function as the regularization (or penalty) terms and we minimize the regularized cost. The dual variables $\b{\lambda}$ and $\b{\nu}$ can be seen as the regularization parameters which weight the penalties compared to the objective function $f(\b{x})$.
This regularized cost function is the Lagrangian function or the Lagrangian relaxation of the problem (\ref{equation_optimization_problem}).
Minimization of the regularized cost function minimizes the function $f(\b{x})$ while trying to satisfy the constraints. 

\subsubsection{Lagrange Dual Function}

\begin{definition}[Lagrange dual function]
The Lagrange dual function (also called the dual function) $g: \mathbb{R}^{m_1} \times \mathbb{R}^{m_2} \rightarrow \mathbb{R}$ is defined as:
\begin{equation}\label{equation_dual_function}
\begin{aligned}
g(\b{\lambda}, \b{\nu}) &:= \inf_{\b{x} \in \mathcal{D}} \mathcal{L}(\b{x},\b{\lambda},\b{\nu}) \\
&= \inf_{\b{x} \in \mathcal{D}} \Big(f(\b{x}) + \sum_{i=1}^{m_1} \lambda_i y_i(\b{x}) + \sum_{i=1}^{m_2} \nu_i h_i(\b{x})\Big).
\end{aligned}
\end{equation}
\end{definition}
Note that the dual function $g$ is a concave function. We will see later, in Section \ref{section_dual_problem}, that we maximize this concave function in a so-called dual problem.

\subsection{Primal Feasibility}

\begin{definition}[The optimal point and the optimum]
The solution of optimization problem (\ref{equation_optimization_problem}) is the optimal point denoted by $\b{x}^*$.
The minimum function from this solution, i.e., $f^* := f(\b{x}^*)$, is called the optimum function of problem (\ref{equation_optimization_problem}).
\end{definition}

The optimal point $\b{x}^*$ is one of the feasible points which minimizes function $f(.)$ with constraints in problem (\ref{equation_optimization_problem}). Hence, the optimal point is a feasible point and according to Eq. (\ref{equation_feasible_point}), we have:
\begin{align}
& y_i(\b{x}^*) \leq 0, \quad \forall i \in \{1, \ldots, m_1\}, \\
& h_i(\b{x}^*) = 0, \quad \forall i \in \{1, \ldots, m_2\}.
\end{align}
These are called the \textit{primal feasibility}. 

The optimal point $\b{x}^*$ minimizes the Lagrangian function because Lagrangian is the relaxation of optimization problem to an unconstrained problem (see Section \ref{section_Lagrangian_interpretation}). On the other hand, according to Eq. (\ref{equation_dual_function}), the dual function is the minimum of Lagrangian w.r.t. $\b{x}$. Hence, we can write the dual function as:
\begin{equation}\label{equation_dual_function_2}
\begin{aligned}
g(\b{\lambda}, \b{\nu}) \overset{(\ref{equation_dual_function})}{=} \inf_{\b{x} \in \mathcal{D}} \mathcal{L}(\b{x},\b{\lambda},\b{\nu}) = \mathcal{L}(\b{x}^*,\b{\lambda},\b{\nu}).
\end{aligned}
\end{equation}

\subsection{Dual Feasibility}

\begin{lemma}[Dual function as a lower bound]\label{lemma_dual_function_lower_bound}
If $\b{\lambda} \succeq \b{0}$, then the dual function is a lower bound for $f^*$, i.e., $g(\b{\lambda}, \b{\nu}) \leq f^*$. 
\end{lemma}
\begin{proof}
Let $\lambda \succeq \b{0}$ which means $\lambda_i \geq 0, \forall i$.
Consider a feasible $\widetilde{\b{x}}$ for problem (\ref{equation_optimization_problem}). According to Eq. (\ref{equation_feasible_point}), we have:
\begin{align}
\mathcal{L}(\widetilde{\b{x}},\b{\lambda},\b{\nu}) &\overset{(\ref{equation_Lagrangian})}{=} f(\widetilde{\b{x}}) + \sum_{i=1}^{m_1} \underbrace{\lambda_i}_{\geq 0} \underbrace{y_i(\widetilde{\b{x}})}_{\leq 0} + \sum_{i=1}^{m_2} \nu_i \underbrace{h_i(\widetilde{\b{x}})}_{=0} \nonumber\\
&\leq f(\widetilde{\b{x}}). \label{equation_proofLowerBound_dualFunction_1}
\end{align}
Therefore, we have:
\begin{align*}
f(\widetilde{\b{x}}) \overset{(\ref{equation_proofLowerBound_dualFunction_1})}{\geq} \mathcal{L}(\widetilde{\b{x}}, \b{\lambda}, \b{\nu}) \geq \inf_{\b{x} \in \mathcal{D}} \mathcal{L}(\b{x}, \b{\lambda}, \b{\nu}) \overset{(\ref{equation_dual_function})}{=} g(\b{\lambda}, \b{\nu}).
\end{align*}
Hence, the dual function is a lower bound for the function of all feasible points. As the optimal point $\b{x}^*$ is a feasible point, the dual function is a lower bound for $f^*$. Q.E.D.
\end{proof}

\begin{corollary}[Nonnegativity of dual variables for inequality constraints]\label{corollary_nonnegativity_dual_variables}
From Lemma \ref{lemma_dual_function_lower_bound}, we conclude that for having the dual function as a lower bound for the optimum function, the dual variable $\{\lambda_i\}_{i=1}^{m_1}$ for inequality constraints (less than or equal to zero) should be non-negative, i.e.:
\begin{align}\label{equation_nonnegativity_dual_variables}
\b{\lambda} \succeq \b{0} \quad\text{ or }\quad \lambda_i \geq 0, \,\, \forall i \in \{1, \dots, m_1\}.
\end{align}
Note that if the inequality constraints are greater than or equal to zero, we should have $\lambda_i \leq 0, \quad \forall i$ because $y_i(\b{x}) \geq 0 \implies -y_i(\b{x}) \leq 0$. In this paper, we assume that the inequality constraints are less than or equal to zero. If some of the inequality constraints are greater than or equal to zero, we convert them to less than or equal to zero by multiplying them to $-1$.
\end{corollary}
The inequalities in Eq. (\ref{equation_nonnegativity_dual_variables}) are called the \textit{dual feasibility}. 

\subsection{The Dual Problem, Weak and Strong Duality, and Slater's Condition}\label{section_dual_problem}

According to Eq. (\ref{lemma_dual_function_lower_bound}), the dual function is a lower bound for the optimum function, i.e., $g(\b{\lambda}, \b{\nu}) \leq f^*$. We want to find the best lower bound so we maximize $g(\b{\lambda}, \b{\nu})$ w.r.t. the dual variables $\b{\lambda}, \b{\nu}$. Moreover, Eq. (\ref{equation_nonnegativity_dual_variables}) says that the dual variables for inequalities must be nonnegative. Hence, we have the following optimization:
\begin{equation}\label{equation_dual_optimization_problem}
\begin{aligned}
& \underset{\b{\lambda}, \b{\nu}}{\text{maximize}}
& & g(\b{\lambda}, \b{\nu}) \\
& \text{subject to}
& & \b{\lambda} \succeq \b{0}. 
\end{aligned}
\end{equation}
The problem (\ref{equation_dual_optimization_problem}) is called the \textit{Lagrange dual optimization problem} for problem (\ref{equation_optimization_problem}). The problem (\ref{equation_optimization_problem}) is also referred to as the \textit{primal optimization problem}.
The variable of problem (\ref{equation_optimization_problem}), i.e. $\b{x}$, is called the \textit{primal variable} while the variables of problem (\ref{equation_dual_optimization_problem}), i.e. $\b{\lambda}$ and $\b{\nu}$, are called the \textit{dual variables}. 
Let the solutions of the dual problem be denoted by $\b{\lambda}^*$ and $\b{\nu}^*$. We denote $g^* := g(\b{\lambda}^*, \b{\nu}^*) = \sup_{\b{\lambda}, \b{\nu}} g$. 

\begin{definition}[Weak and strong duality]
For all convex and nonconvex problems, the optimum dual problem is a lower bound for the optimum function:
\begin{align}\label{equation_weak_duality}
g^* \leq f^* \quad \text{i.e.,} \quad g(\b{\lambda}^*, \b{\nu}^*) \leq f(\b{x}^*).
\end{align}
This is called the weak duality. 
For some optimization problems, we have strong duality which is when the optimum dual problem is equal to the optimum function:
\begin{align}\label{equation_strong_duality}
g^* = f^* \quad \text{i.e.,} \quad g(\b{\lambda}^*, \b{\nu}^*) = f(\b{x}^*).
\end{align}
The strong duality usually holds for convex optimization problems. 
\end{definition}

\begin{corollary}\label{corollary_dual_optimal_lower_bound_primal_optimal}
Eqs. (\ref{equation_weak_duality}) and (\ref{equation_strong_duality}) show that the optimum dual function, $g*$, always provides a lower-bound for the optimum primal function, $f^*$. 
\end{corollary}

The primal optimization problem, i.e. Eq. (\ref{equation_optimization_problem}), is minimization so its cost function is like a bowl as illustrated in Fig. \ref{figure_primal_dual_problems}. 
The dual optimization problem, i.e. Eq. (\ref{equation_dual_optimization_problem}), is maximization so its cost function is like a reversed bowl as shown in Fig. \ref{figure_primal_dual_problems}.
The domains for primal and dual problems are the domain of primal variable $\b{x}$ and the domain of dual variables $\b{\lambda}$ and $\b{\nu}$, respectively. As the figure shows, the optimal $\b{x}^*$ is corresponded to the optimal $\b{\lambda}^*$ and $\b{\nu}^*$. 
As shown in the figure, there is a possible nonnegative gap between the two bowls. In the best case, this gap is zero. If the gap is zero, we have strong duality; otherwise, a weak duality exists.

\begin{figure}[!t]
\centering
\includegraphics[width=3.2in]{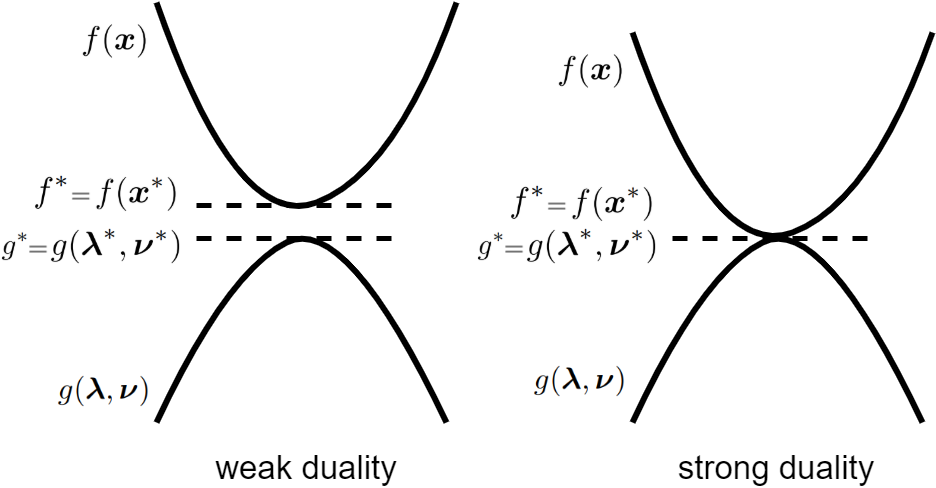}
\caption{Illustration of weak duality and strong duality.}
\label{figure_primal_dual_problems}
\end{figure}

\begin{figure}[!t]
\centering
\includegraphics[width=3.2in]{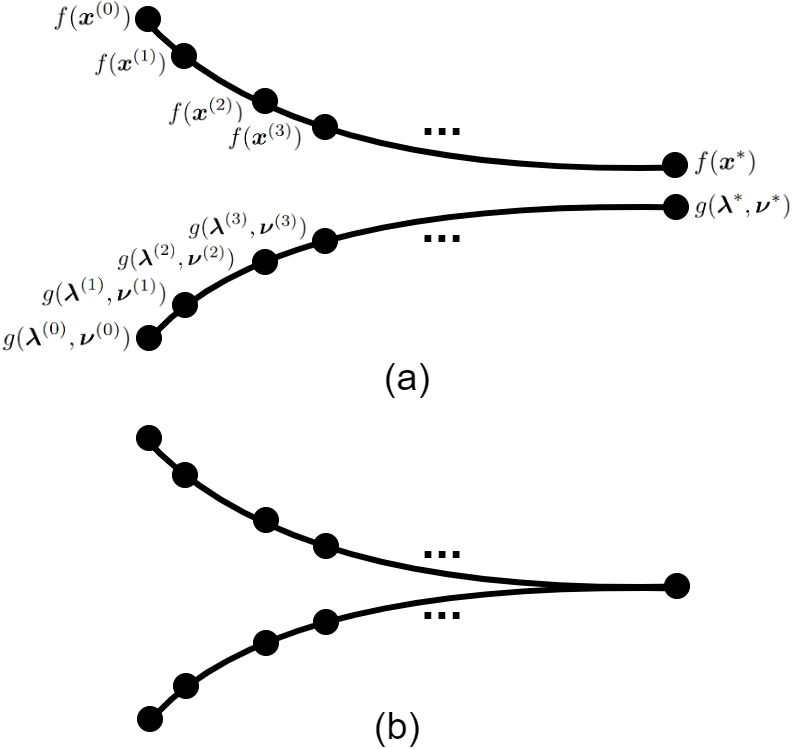}
\caption{Progress of iterative optimization: (a) gradual minimization of the primal function and maximization of dual function and (b) the primal optimal and dual optimal reach each other and become equal if strong duality holds.}
\label{figure_iterative_primal_dual}
\end{figure}

If optimization is iterative, the solution is updated iteratively until convergence. First-order and second-order numerical optimization, which we will introduce later, are iterative. 
In optimization, the series of primal optimal and dual optimal converge to the optimal solution and the dual optimal, respectively. The function values converge to the local minimum and the dual function values converge to the optimal (maximum) dual function. Let the superscript $(k)$ denotes the value of variable at iteration $k$. We have:
\begin{equation}\label{equation_iterative_optimization_series}
\begin{aligned}
& \{\b{x}^{(0)}, \b{x}^{(1)}, \b{x}^{(2)}, \dots\} \rightarrow \b{x}^*, \\
& \{\b{\nu}^{(0)}, \b{\nu}^{(1)}, \b{\nu}^{(2)}, \dots\} \rightarrow \b{\nu}^*, \\
& \{\b{\lambda}^{(0)}, \b{\lambda}^{(1)}, \b{\lambda}^{(2)}, \dots\} \rightarrow \b{\lambda}^*, \\
& f(\b{x}^{(0)}) \geq f(\b{x}^{(1)}) \geq f(\b{x}^{(2)}) \geq \dots \geq f(\b{x}^*), \\
& g(\b{\lambda}^{(0)}, \b{\nu}^{(0)}) \leq g(\b{\lambda}^{(1)}, \b{\nu}^{(1)}) \leq \dots \leq g(\b{\lambda}^*, \b{\nu}^*).
\end{aligned}
\end{equation}
Hence, the value of function goes down but the value of dual function goes up. As Fig. \ref{figure_iterative_primal_dual} depicts, they reach each other if strong duality holds; otherwise, there will be a gap between them after convergence.
Note that if the optimization problem is a convex problem, the eventually found solution is the global solution; otherwise, the solution is local. 

\begin{corollary}\label{corollary_dual_optimal_lower_bound_primal_optimal_iterative}
As every iteration of a numerical optimization must satisfy either the weak or strong duality, the optimum dual function at every iteration always provides a lower-bound for the optimum primal function at that iteration:
\begin{align}
g(\b{\lambda}^{(k)}, \b{\nu}^{(k)}) \leq f(\b{x}^{(k)}), \quad \forall k.
\end{align}
\end{corollary}

\begin{lemma}[Slater's condition \cite{slater1950lagrange}]
For a convex optimization problem in the form: 
\begin{equation}\label{equation_optimization_problem_inequalityConstraints_Ax_equals_b_constraint}
\begin{aligned}
& \underset{\b{x}}{\text{minimize}}
& & f(\b{x}) \\
& \text{subject to}
& & y_i(\b{x}) \leq 0, \; i \in \{1, \ldots, m_1\}, \\
& & & \b{Ax} = \b{b}, 
\end{aligned}
\end{equation}
we have strong duality if it is strictly feasible, i.e.:
\begin{equation}
\begin{aligned}
\exists \b{x} \in \textbf{int}(\mathcal{D}):\, & y_i(\b{x}) < 0, \quad \forall i \in \{1, \dots, m_1\}, \\
& \b{Ax} = \b{b}.
\end{aligned}
\end{equation}
In other words, for at least one point in the interior of domain (not on the boundary of domain), all the inequality constraints hold strictly. This is called the Slater's condition. 
\end{lemma}

\subsection{Complementary Slackness}

Assume that the problem has strong duality, the primal optimal is $\b{x}^*$ and dual optimal variables are $\b{\lambda}^*$ and $\b{\nu}^*$. We have:
\begin{align}
f(\b{x}^*) &\overset{(\ref{equation_strong_duality})}{=} g(\b{\lambda}^*, \b{\nu}^*) \nonumber \\
&\overset{(\ref{equation_dual_function})}{=} \inf_{\b{x} \in \mathcal{D}} \Big(f(\b{x}) + \sum_{i=1}^{m_1} \lambda_i^* y_i(\b{x}) + \sum_{i=1}^{m_2} \nu_i^* h_i(\b{x})\Big) \nonumber \\
&\overset{(a)}{=} f(\b{x}^*) + \sum_{i=1}^{m_1} \lambda_i^* y_i(\b{x}^*) + \sum_{i=1}^{m_2} \nu_i^* h_i(\b{x}^*) \nonumber \\
&\overset{(b)}{=} f(\b{x}^*) + \sum_{i=1}^{m_1} \lambda_i^* y_i(\b{x}^*) \overset{(c)}{\leq} f(\b{x}^*), \label{equation_complementary_slackness_step1}
\end{align}
where $(a)$ is because $\b{x}^*$ is the primal optimal solution for problem (\ref{equation_optimization_problem}) and it minimizes the Lagrangian, $(b)$ is because $\b{x}^*$ is a feasible point and satisfies $h_i(\b{x}^*)=0$ in Eq. (\ref{equation_feasible_point}), and $(c)$ is because $\lambda_i^* \geq 0$ according to Eq. (\ref{equation_nonnegativity_dual_variables}) and the feasible $\b{x}^*$ satisfies $y_i(\b{x}^*) \leq 0$ in Eq. (\ref{equation_feasible_point}) so we have:
\begin{align}\label{equation_complementary_slackness_step2}
\lambda_i^* y_i(\b{x}^*) \leq 0, \quad \forall i \in \{1, \dots, m_1\}.
\end{align}
From Eq. (\ref{equation_complementary_slackness_step1}), we have:
\begin{align*}
&f(\b{x}^*) = f(\b{x}^*) + \sum_{i=1}^{m_1} \lambda_i^* y_i(\b{x}^*) \leq f(\b{x}^*) \\
&\implies \sum_{i=1}^{m_1} \lambda_i^* y_i(\b{x}^*) = 0 \overset{(\ref{equation_complementary_slackness_step2})}{\implies} \lambda_i^* y_i(\b{x}^*) = 0, \forall i. 
\end{align*}
Therefore, the multiplication of every optimal dual variable $\lambda_i^*$ with $y_i(.)$ of optimal primal solution $\b{x}^*$ must be zero. 
This is called the \textit{complementary slackness}:
\begin{align}
& \lambda_i^*\, y_i(\b{x}^*) = 0, \quad \forall i \in \{1, \dots, m_1\}.
\end{align}
These conditions can be restated as:
\begin{align}
&\lambda_i^* > 0 \implies y_i(\b{x}^*) = 0, \\
&y_i(\b{x}^*) < 0 \implies \lambda_i^* = 0, 
\end{align}
which means that, for an inequality constraint, if the dual optimal is nonzero, its inequality function of the primal optimal must be zero. If the inequality function of the primal optimal is nonzero, its dual optimal must be zero.

\subsection{Stationarity Condition}

As was explained before, the Lagrangian function can be interpreted as a regularized cost function to be minimized. Hence, the constrained optimization problem (\ref{equation_optimization_problem}) is converted to minimization of the Lagrangian function, Eq. (\ref{equation_Lagrangian}), which is an unconstrained optimization problem:
\begin{align}
\underset{\b{x}}{\text{minimize}}\,\, \mathcal{L}(\b{x},\b{\lambda},\b{\nu}).
\end{align}
Note that this problem is the dual function according to Eq. (\ref{equation_dual_function}). 
As this is an unconstrained problem, its optimization is easy. We can find its minimum by setting its derivative w.r.t. $\b{x}$, denoted by $\nabla_{\b{x}} \mathcal{L}$, to zero:
\begin{equation}
\begin{aligned}
&\nabla_{\b{x}} \mathcal{L}(\b{x},\b{\lambda},\b{\nu}) = 0 \overset{(\ref{equation_Lagrangian})}{\implies} \\
& \nabla_{\b{x}} f(\b{x}) + \sum_{i=1}^{m_1} \lambda_i \nabla_{\b{x}} y_i(\b{x}) + \sum_{i=1}^{m_2} \nu_i \nabla_{\b{x}} h_i(\b{x}) = 0.
\end{aligned}
\end{equation}
This equation is called the \textit{stationarity condition} because this shows that the gradient of Lagrangian w.r.t. $\b{x}$ should vanish to zero (n.b. a stationary point of a function is a point where the derivative of function is zero). This derivative holds for all dual variables and not just for the optimal dual variables. 
We can claim that the gradient of Lagrangian w.r.t. $\b{x}$ should vanish to zero because the dual function, defined in Eq. (\ref{equation_dual_function}), should exist.  

\subsection{KKT Conditions}

We derived the primal feasibility, dual feasibility, complementary slackness, and stationarity condition. These four conditions are called the \textit{Karush-Kuhn-Tucker (KKT) conditions} \cite{karush1939minima,kuhn1951nonlinear}. 
The primal optimal variable $\b{x}^*$ and the dual optimal variables $\b{\lambda}^* = [\lambda_1^*, \dots, \lambda_{m_1}^*]^\top$, $\b{\nu}^* = [\nu_1^*, \dots, \nu_{m_2}^*]^\top$ must satisfy the KKT conditions. 
We summarize the KKT conditions in the following: 
\begin{enumerate}
\item Stationarity condition:
\begin{equation}\label{equation_stationarity_condition}
\begin{aligned}
\nabla_{\b{x}} \mathcal{L}(\b{x},\b{\lambda},\b{\nu}) = &\nabla_{\b{x}} f(\b{x}) + \sum_{i=1}^{m_1} \lambda_i \nabla_{\b{x}} y_i(\b{x}) \\
&+ \sum_{i=1}^{m_2} \nu_i \nabla_{\b{x}} h_i(\b{x}) = 0.
\end{aligned}
\end{equation}
\item Primal feasibility:
\begin{align}
& y_i(\b{x}^*) \leq 0, \quad \forall i \in \{1, \ldots, m_1\}, \\
& h_i(\b{x}^*) = 0, \quad \forall i \in \{1, \ldots, m_2\}.
\end{align}
\item Dual feasibility:
\begin{align}\label{equation_dual_constraints}
\b{\lambda} \succeq \b{0} \quad\text{ or }\quad \lambda_i \geq 0, \,\, \forall i \in \{1, \dots, m_1\}.
\end{align}
\item Complementary slackness:
\begin{align}
\lambda_i^*\, y_i(\b{x}^*) = 0, \quad \forall i \in \{1, \dots, m_1\}.
\end{align}
\end{enumerate}
As listed above, KKT conditions impose constraints on the optimal dual variables of inequality constraints because the sign of inequalities are important. 

Recall the dual problem (\ref{equation_dual_optimization_problem}). The constraint in this problem is already satisfied by the dual feasibility in the KKT conditions. Hence, we can ignore the constraint of the dual problem (as it is automatically satisfied by dual feasibility):
\begin{equation}\label{equation_dual_optimization_problem_unconstrained}
\begin{aligned}
& \underset{\b{\lambda}, \b{\nu}}{\text{maximize}}
& & g(\b{\lambda}, \b{\nu}), 
\end{aligned}
\end{equation}
which should give us $\b{\lambda}^*$, $\b{\nu}^*$, and $g^* = g(\b{\lambda}^*, \b{\nu}^*)$.
This is an unconstrained optimization problem and for solving it, we should set the derivative of $g(\b{\lambda}, \b{\nu})$ w.r.t. $\b{\lambda}$ and $\b{\nu}$ to zero:
\begin{align}
& \nabla_{\b{\lambda}} g(\b{\lambda}, \b{\nu}) = 0 \overset{(\ref{equation_dual_function_2})}{\implies} \nabla_{\b{\lambda}} \mathcal{L}(\b{x}^*,\b{\lambda},\b{\nu}) = 0. \label{equation_derivative_Lagrangian_wrt_lambda_zero} \\
& \nabla_{\b{\nu}} g(\b{\lambda}, \b{\nu}) = 0 \overset{(\ref{equation_dual_function_2})}{\implies} \nabla_{\b{\nu}} \mathcal{L}(\b{x}^*,\b{\lambda},\b{\nu}) = 0. \label{equation_derivative_Lagrangian_wrt_nu_zero}
\end{align}
Note that setting the derivatives of Lagrangian w.r.t. dual variables always gives back the corresponding constraints in the primal optimization problem. 
Eqs. (\ref{equation_stationarity_condition}), (\ref{equation_derivative_Lagrangian_wrt_lambda_zero}), and (\ref{equation_derivative_Lagrangian_wrt_nu_zero}) state that the primal and dual residuals must be zero. 

Finally, Eqs. (\ref{equation_dual_function}) and (\ref{equation_dual_optimization_problem_unconstrained}) can be summarized into the following max-min optimization problem:
\begin{align}
\sup_{\b{\lambda}, \b{\nu}}\, g(\b{\lambda}, \b{\nu}) \overset{(\ref{equation_dual_function})}{=} \sup_{\b{\lambda}, \b{\nu}}\, \inf_{\b{x}}\, \mathcal{L}(\b{x},\b{\lambda}, \b{\nu}) = \mathcal{L}(\b{x}^*,\b{\lambda}^*, \b{\nu}^*).
\end{align}

The reason for the name KKT is as follows \cite{kjeldsen2000contextualized}. In 1952, Kuhn and Tucker published an important paper proposing the conditions \cite{kuhn1951nonlinear}. However, later it was found out that there is a master's these by Karush, in 1939, at the University of Chicago, Illinois \cite{karush1939minima}. That thesis had also proposed the conditions; however, researchers including Kuhn and Tucker were not aware of that thesis. Therefore, these conditions were named after all three of them. 

\subsection{Solving Optimization by Method of Lagrange Multipliers}\label{section_method_of_multipliers}

We can solve the optimization problem (\ref{equation_optimization_problem}) using duality and KKT conditions. 
This technique is also called the \textit{method of Lagrange multipliers}. 
For this, we should do the following steps:
\begin{enumerate}
\item We write the Lagrangian as Eq. (\ref{equation_Lagrangian}).
\item We consider the dual function defined in Eq. (\ref{equation_dual_function}) and we solve it: 
\begin{align}\label{equation_KKT_x_dagger}
\b{x}^\dagger := \arg \min_{\b{x}}\, \mathcal{L}(\b{x},\b{\lambda},\b{\nu}). 
\end{align}
It is an unconstrained problem and according to Eqs. (\ref{equation_dual_function}) and (\ref{equation_stationarity_condition}), we solve this problem by taking the derivative of Lagrangian w.r.t. $\b{x}$ and setting it to zero, i.e., $\nabla_{\b{x}}\mathcal{L}(\b{x},\b{\lambda},\b{\nu}) \overset{\text{set}}{=} 0$.
This gives us the dual function, according to Eq. (\ref{equation_Lagrangian}):
\begin{align}
g(\b{\lambda}, \b{\nu}) = \mathcal{L}(\b{x}^\dagger,\b{\lambda},\b{\nu}). 
\end{align}

\item We consider the dual problem, defined in Eq. (\ref{equation_dual_optimization_problem}) which is simplified to Eq. (\ref{equation_dual_optimization_problem_unconstrained}) because of Eq. (\ref{equation_dual_constraints}). This gives us the optimal dual variables $\b{\lambda}^*$ and $\b{\nu}^*$:
\begin{align}
\b{\lambda}^*, \b{\nu}^* := \arg \max_{\b{\lambda}, \b{\nu}}\, g(\b{\lambda},\b{\nu}). 
\end{align}
It is an unconstrained problem and according to Eqs. (\ref{equation_derivative_Lagrangian_wrt_lambda_zero}) and (\ref{equation_derivative_Lagrangian_wrt_nu_zero}), we solve this problem by taking the derivative of dual function w.r.t. $\b{\lambda}$ and $\b{\nu}$ and setting them to zero, i.e., $\nabla_{\b{\lambda}}g(\b{\lambda},\b{\nu}) \overset{\text{set}}{=} 0$ and $\nabla_{\b{\nu}}g(\b{\lambda},\b{\nu}) \overset{\text{set}}{=} 0$.
The optimum dual value is obtained as:
\begin{align}
g^* = \max_{\b{\lambda}, \b{\nu}}\, g(\b{\lambda},\b{\nu}) = g(\b{\lambda}^*,\b{\nu}^*).
\end{align}
\item We put the optimal dual variables $\b{\lambda}^*$ and $\b{\nu}^*$ in Eq. (\ref{equation_stationarity_condition}) to find the optimal primal variable:
\begin{align}
\b{x}^* := \arg \min_{\b{x}}\, \mathcal{L}(\b{x},\b{\lambda}^*,\b{\nu}^*). 
\end{align}
It is an unconstrained problem and we solve this problem by taking the derivative of Lagrangian at optimal dual variables w.r.t. $\b{x}$ and setting it to zero, i.e., $\nabla_{\b{x}}\mathcal{L}(\b{x},\b{\lambda}^*,\b{\nu}^*) \overset{\text{set}}{=} 0$.
The optimum primal value is obtained as:
\begin{align}
f^* = \min_{\b{x}}\, \mathcal{L}(\b{x},\b{\lambda}^*,\b{\nu}^*) = \mathcal{L}(\b{x}^*,\b{\lambda}^*,\b{\nu}^*).
\end{align}
\end{enumerate}

\section{First-Order Optimization: Gradient Methods}\label{section_first_order_methods}

\subsection{Gradient Descent}\label{section_gradient_descent}

\textit{Gradient descent} is one of the fundamental first-order methods. It was first suggested by Cauchy in 1874 \cite{lemarechal2012cauchy} and Hadamard in 1908 \cite{hadamard1908memoire} and its convergence was later analyzed in \cite{curry1944method}. In the following, we introduce this method. 

\subsubsection{Step of Update}\label{section_GD_step_update}

Consider the unconstrained optimization problem (\ref{equation_optimization_problem_unconstrained}).
Here, we denote $\b{x}^* := \arg \min_{\b{x}} f(\b{x})$ and $f^* := \min_{\b{x}} f(\b{x}) = f(\b{x}^*)$.
In numerical optimization for unconstrained optimization, we start with a random feasible initial point and iteratively update it by step $\Delta \b{x}$: 
\begin{align}\label{equation_update_point_numerical_optimization}
\b{x}^{(k+1)} := \b{x}^{(k)} + \Delta \b{x},
\end{align}
until we converge to (or get sufficiently close to) the desired optimal point $\b{x}^*$. 
Note that the step $\Delta\b{x}$ is also denoted by $\b{p}$ in the literature, i.e., $\b{p} := \Delta\b{x}$.
Let the function $f(.)$ be differentiable and its gradient is $L$-smooth. 
If we set $\b{x} = \b{x}^{(k)}$ and $\b{y} = \b{x}^{(k+1)} = \b{x}^{(k)} + \Delta \b{x}$ in Eq. (\ref{equation_fundamental_theorem_calculus_Lipschitz}), we have:
\begin{align}
&f(\b{x}^{(k)} + \Delta \b{x}) \leq f(\b{x}^{(k)}) + \nabla f(\b{x}^{(k)})^\top \Delta \b{x} + \frac{L}{2} \|\Delta \b{x}\|_2^2 \nonumber\\
&\implies f(\b{x}^{(k)} + \Delta \b{x}) - f(\b{x}^{(k)}) \nonumber\\
&~~~~~~~~~~~~~~~~~~~~~~~~~~~\leq \nabla f(\b{x}^{(k)})^\top \Delta \b{x} + \frac{L}{2} \|\Delta \b{x}\|_2^2. \label{equation_fundamental_theorem_calculus_Lipschitz_GD}
\end{align}
Until reaching the minimum, we want to decrease the cost function $f(.)$ in every iteration; hence, we desire:
\begin{align}\label{equation_GD_decrease_cost_function}
f(\b{x}^{(k)} + \Delta \b{x}) - f(\b{x}^{(k)}) < 0.
\end{align}
According to Eq. (\ref{equation_fundamental_theorem_calculus_Lipschitz_GD}), one way to achieve Eq. (\ref{equation_GD_decrease_cost_function}) is:
\begin{align*}
\nabla f(\b{x}^{(k)})^\top \Delta \b{x} + \frac{L}{2} \|\Delta \b{x}\|_2^2 < 0.
\end{align*}
Hence, we should minimize $\nabla f(\b{x}^{(k)})^\top \Delta \b{x} + \frac{L}{2} \|\Delta \b{x}\|_2^2$ w.r.t. $\Delta \b{x}$:
\begin{align}\label{equation_GD_min_RHS_of_corollary_fundamental}
\underset{\Delta\b{x}}{\text{minimize}}\,\, \nabla f(\b{x}^{(k)})^\top \Delta \b{x} + \frac{L}{2} \|\Delta \b{x}\|_2^2.
\end{align}
This function is convex w.r.t. $\Delta \b{x}$ and we can optimize it by setting its derivative to zero:
\begin{align}
&\frac{\partial }{\partial \Delta \b{x}} (\nabla f(\b{x}^{(k)})^\top \Delta \b{x} + \frac{L}{2} \|\Delta \b{x}\|_2^2) = \nabla f(\b{x}^{(k)}) + L \Delta \b{x} \nonumber\\
&\overset{\text{set}}{=} \b{0} \implies \Delta \b{x} = -\frac{1}{L} \nabla f(\b{x}^{(k)}). \label{equation_GD_step_by_L}
\end{align}
Using Eq. (\ref{equation_GD_step_by_L}) in Eq. (\ref{equation_fundamental_theorem_calculus_Lipschitz_GD}) gives:
\begin{align*}
f(\b{x}^{(k)} + \Delta \b{x}) - f(\b{x}^{(k)}) \leq -\frac{1}{2L} \|\nabla f(\b{x}^{(k)})\|_2^2 \leq 0,
\end{align*}
which satisfies Eq. (\ref{equation_GD_decrease_cost_function}).
Eq. (\ref{equation_GD_step_by_L}) means that it is better to move toward a scale of minus gradient for updating the solution. This inspires the name of algorithm which is \textit{gradient descent}. 

The problem is that often we either do not know the Lipschitz constant $L$ or it is hard to compute. Hence, rather than Eq. (\ref{equation_GD_step_by_L}), we use:
\begin{align}\label{equation_GD_step_by_eta}
\Delta \b{x} = -\eta \nabla f(\b{x}^{(k)}), \text{ i.e., } \b{x}^{(k+1)} := \b{x}^{(k)} - \eta \nabla f(\b{x}^{(k)}),
\end{align}
where $\eta > 0$ is the step size, also called the learning rate in data science literature. 
Note that if the optimization problem is maximization rather than minimization, the step should be $\Delta \b{x} = \eta \nabla f(\b{x}^{(k)})$ rather than Eq. (\ref{equation_GD_step_by_eta}). In that case, the name of method is \textit{gradient ascent}. 

Using Eq. (\ref{equation_GD_step_by_eta}) in Eq. (\ref{equation_fundamental_theorem_calculus_Lipschitz_GD}) gives:
\begin{align}
f(\b{x}^{(k)} + \Delta \b{x}) &- f(\b{x}^{(k)}) \nonumber\\
&\leq -\eta \|\nabla f(\b{x}^{(k)})\|_2^2 + \frac{L}{2} \eta^2 \|\nabla f(\b{x}^{(k)})\|_2^2 \label{equation_GD_step_size_eta_L_mid1}\\
&= \eta (\frac{L}{2} \eta - 1) \|\nabla f(\b{x}^{(k)})\|_2^2 \nonumber
\end{align}
If $\b{x}^{(k)}$ is not a stationary point, we have $\|\nabla f(\b{x}^{(k)})\|_2^2 > 0$. 
Noticing $\eta > 0$, for satisfying Eq. (\ref{equation_GD_decrease_cost_function}), we must set:
\begin{align}\label{equation_GD_step_size_eta_L}
\frac{L}{2} \eta - 1 < 0 \implies \eta < \frac{2}{L}.
\end{align}
On the other hand, we can minimize Eq. (\ref{equation_GD_step_size_eta_L_mid1}) by setting its derivative w.r.t. $\eta$ to zero:
\begin{align*}
&\frac{\partial }{\partial \eta} (-\eta \|\nabla f(\b{x}^{(k)})\|_2^2 + \frac{L}{2} \eta^2 \|\nabla f(\b{x}^{(k)})\|_2^2) \\
&= -\|\nabla f(\b{x}^{(k)})\|_2^2 + L \eta \|\nabla f(\b{x}^{(k)})\|_2^2 \\
&= (-1+L \eta) \|\nabla f(\b{x}^{(k)})\|_2^2 \overset{\text{set}}{=} 0 \implies \eta = \frac{1}{L}.
\end{align*}
If we set:
\begin{align}\label{equation_GD_step_size_eta_L_2}
\eta < \frac{1}{L},
\end{align}
then Eq. (\ref{equation_GD_step_size_eta_L_mid1}) becomes:
\begin{align}
&f(\b{x}^{(k)} + \Delta \b{x}) - f(\b{x}^{(k)}) \nonumber \\
&\leq -\frac{1}{L} \|\nabla f(\b{x}^{(k)})\|_2^2 + \frac{1}{2L} \|\nabla f(\b{x}^{(k)})\|_2^2 \nonumber \\
&= -\frac{1}{2L}\|\nabla f(\b{x}^{(k)})\|_2^2 < 0 \nonumber \\
&\implies f(\b{x}^{(k+1)}) \leq f(\b{x}^{(k)}) -\frac{1}{2L}\|\nabla f(\b{x}^{(k)})\|_2^2.
\label{equation_GD_decrease_cost}
\end{align}
Eq. (\ref{equation_GD_step_size_eta_L_2}) means that there should be an upper-bound, dependent on the Lipschitz constant, on the step size. Hence, $L$ is still required. 
Eq. (\ref{equation_GD_decrease_cost}) shows that every iteration of gradient descent decreases the cost function:
\begin{align}
f(\b{x}^{(k+1)}) \leq f(\b{x}^{(k)}),
\end{align}
and the amount of this decrease depends on the norm of gradient at that iteration. 
In conclusion, the series of solutions converges to the optimal solution while the function value decreases iteratively until the local minimum:
\begin{align*}
& \{\b{x}^{(0)}, \b{x}^{(1)}, \b{x}^{(2)}, \dots\} \rightarrow \b{x}^*, \\
& f(\b{x}^{(0)}) \geq f(\b{x}^{(1)}) \geq f(\b{x}^{(2)}) \geq \dots \geq f(\b{x}^*).
\end{align*}
If the optimization problem is a convex problem, the solution is the global solution; otherwise, the solution is local. 

\subsubsection{Line-Search}\label{section_GD_line_search}

As was shown in Section \ref{section_GD_step_update}, the step size of gradient descent requires knowledge of the Lipschitz constant for the smoothness of gradient. Hence, we can find the suitable step size $\eta$ by a search which is named the \textit{line-search}. In line-search of every optimization iteration, we start with $\eta=1$ and halve it, $\eta \gets \eta/2$, if it does not satisfy Eq. (\ref{equation_GD_decrease_cost_function}) with step $\Delta \b{x} = -\eta \nabla f(\b{x}^{(k)})$:
\begin{align}\label{equation_GD_line_search_condition}
f(\b{x}^{(k)} -\eta \nabla f(\b{x}^{(k)})) < f(\b{x}^{(k)}).
\end{align}
This halving step size is repeated until this equation is satisfied, i.e., until we have a decrease in the objective function. Note that this decrease will happen when the step size becomes small enough to satisfy Eq. (\ref{equation_GD_step_size_eta_L_2}). 
The algorithm of gradient descent with line-search is shown in Algorithm \ref{algorithm_line_search}.
As this algorithm shows, line-search has its own internal iterations inside every iteration of gradient descent. 

\begin{lemma}[Time complexity of line-search]\label{lemma_time_complexity_line_search}
In the worst-case, line-search takes $(\log L / \log 2)$ iterations until Eq. (\ref{equation_GD_line_search_condition}) is satisfied. 
\end{lemma}
\begin{proof}
Proof is available in Appendix \ref{app_time_complexity_line_search}.
\end{proof}

\subsubsection{Backtracking Line-Search}\label{section_GD_Armijo_line_search}


A more sophisticated line-search method is the \textit{Armijo line-search} \cite{armijo1966minimization}, also called the \textit{backtracking line-search}. Rather than Eq. (\ref{equation_GD_line_search_condition}), it checks if the cost function is \textit{sufficiently} decreased:
\begin{align}\label{equation_general_Armijo_line_search_condition}
& f(\b{x}^{(k)} + \b{p}) \leq f(\b{x}^{(k)}) +  c\, \b{p}^\top f(\b{x}^{(k)}),
\end{align}
where $c \in (0. 0.5]$ is the parameter of Armijo line-search and $\b{p} = \Delta\b{x}$ is the search direction for update.
The value of $c$ should be small, e.g., $c = 10^{-4}$ \cite{nocedal2006numerical}.
This condition is called the \textit{Armijo condition} or the \textit{Armijo-Goldstein condition}.
In gradient descent, the search direction is $\b{p} = \Delta\b{x} = -\eta \nabla f(\b{x}^{(k)})$ according to Eq. (\ref{equation_GD_step_by_eta}). Hence, for gradient descent, it checks:
\begin{align}\label{equation_GD_Armijo_line_search_condition}
& f(\b{x}^{(k)} -\eta \nabla f(\b{x}^{(k)})) \leq f(\b{x}^{(k)}) - \eta\, \gamma \|\nabla f(\b{x}^{(k)})\|_2^2.
\end{align}
The algorithm of gradient descent with Armijo line-search is shown in Algorithm \ref{algorithm_line_search}.
Note that we can have more sophisticated line-search with \textit{Wolfe conditions} \cite{wolfe1969convergence}. This will be introduced in Section \ref{section_Wolfe_conditions}.

\SetAlCapSkip{0.5em}
\IncMargin{0.8em}
\begin{algorithm2e}[!t]
\DontPrintSemicolon
    Initialize $\b{x}^{(0)}$\;
    \For{iteration $k = 0, 1, \dots$}{
        Initialize $\eta := 1$\;
        \For{iteration $\tau = 1, 2, \dots$}{
            Check Eq. (\ref{equation_GD_line_search_condition}) or (\ref{equation_GD_Armijo_line_search_condition})\; 
            \uIf{not satisfied}{
                $\eta \gets \frac{1}{2} \times \eta$\;
            }
            \Else{
                $\b{x}^{(k+1)} := \b{x}^{(k)} -\eta \nabla f(\b{x}^{(k)})$\;
                \textbf{break} the loop\;
            }
        }
        Check the convergence criterion\;
        \If{converged}{
            \textbf{return} $\b{x}^{(k+1)}$\;
        }
    }
\caption{Gradient descent with line search}\label{algorithm_line_search}
\end{algorithm2e}
\DecMargin{0.8em}

\subsubsection{Convergence Criterion}\label{section_convergence_criterion}

For all numerical optimization methods including gradient descent, there exist several methods for convergence criterion to stop updating the solution and terminate optimization. Some of them are:
\begin{itemize}
\item Small norm of gradient: $\|\nabla f(\b{x}^{(k+1)})\|_2 \leq \epsilon$ where $\epsilon$ is a small positive number. The reason for this criterion is the first-order optimality condition (see Lemma \ref{lemma_first_order_optimality_condition}).
\item Small change of cost function: $|f(\b{x}^{(k+1)}) - f(\b{x}^{(k)})| \leq \epsilon$. 
\item Small change of gradient of function: $|\nabla f(\b{x}^{(k+1)}) - \nabla f(\b{x}^{(k)})| \leq \epsilon$. 
\item Reaching maximum desired number of iterations, denoted by $\max_k$: $k+1 < \max_k$. 
\end{itemize}

\subsubsection{Convergence Analysis for Gradient Descent}

We showed in Eq. (\ref{equation_GD_decrease_cost}) that the cost function value is decreased by gradient descent iterations. The following theorem provides the convergence rate of gradient descent. 
\begin{theorem}[Convergence rate and iteration complexity of gradient descent]\label{theorem_GD_convergence_rate}
Consider a differentiable function $f(.)$, with domain $\mathcal{D}$, whose gradient is $L$-smooth (see Eq. (\ref{equation_gradient_L_smooth})).
Starting from the initial point $\b{x}^{(0)}$, after $t$ iterations of gradient descent, we have:
\begin{align}\label{equation_GD_upperbound_min_norm_gradient}
\min_{0\leq k \leq t} \|\nabla f(\b{x}^{(k)})\|_2^2 \leq \frac{2 L (f(\b{x}^{(0)}) - f^*)}{t+1},
\end{align}
where $f^*$ is the minimum of cost function. In other words, after $t$ iterations, we have:
\begin{align}\label{equation_GD_norm_gradient_sublinear_rate}
\exists\, \b{x}^{(k)}: \|\nabla f(\b{x}^{(k)})\|_2^2 = \mathcal{O}(\frac{1}{t}),
\end{align}
which means the squared norm of gradient has sublinear convergence (see Definition \ref{definition_convergence_rate}) in gradient descent. 
Moreover, after:
\begin{align}\label{equation_GD_t_lowerbound_for_convergence}
t \geq \frac{2 L (f(\b{x}^{(0)}) - f^*)}{\epsilon} - 1,
\end{align}
iterations, gradient descent is guaranteed to satisfy $\|\nabla f(\b{x}^{(k)})\|_2^2 \leq \epsilon$. 
Hence, the iteration complexity of gradient descent is $\mathcal{O}(1/\epsilon)$. 
\end{theorem}
\begin{proof}
Proof is available in Appendix \ref{app_GD_convergence_rate}. 
\end{proof}

The above theorem provides the convergence rate of gradient descent for a general function. If the function is convex, we can simplify this convergence rate further, as stated in the following. 

\begin{theorem}[Convergence rate of gradient descent for convex functions]\label{theorem_GD_convergence_rate_convexFunction}
Consider a convex and differentiable function $f(.)$, with domain $\mathcal{D}$, whose gradient is $L$-smooth (see Eq. (\ref{equation_gradient_L_smooth})).
Starting from the initial point $\b{x}^{(0)}$, after $t$ iterations of gradient descent, we have:
\begin{align}\label{equation_GD_upperbound_min_norm_gradient_convexFunction}
f(\b{x}^{(t+1)}) - f^* \leq \frac{2 L \|\b{x}^{(0)} - \b{x}^*\|_2^2}{t+1},
\end{align}
where $f^*$ is the minimum of cost function and $\b{x}^*$ is the minimizer. In other words, after $t$ iterations, we have:
\begin{align}\label{equation_GD_norm_gradient_sublinear_rate_convexFunction}
f(\b{x}^{(t)}) - f^* = \mathcal{O}(\frac{1}{t}),
\end{align}
which means the distance of convex function value to its optimum has sublinear convergence (see Definition \ref{definition_convergence_rate}) in gradient descent. 
The iteration complexity is the same as Eq. (\ref{equation_GD_t_lowerbound_for_convergence}).
\end{theorem}
\begin{proof}
Proof is available in Appendix \ref{app_GD_convergence_rate_convexFunction}. 
\end{proof}

\begin{theorem}[Convergence rate of gradient descent for strongly convex functions]\label{theorem_GD_convergence_rate_stronglyConvexFunction}
Consider a $\mu$-strongly convex and differentiable function $f(.)$, with domain $\mathcal{D}$, whose gradient is $L$-smooth (see Eq. (\ref{equation_gradient_L_smooth})).
Starting from the initial point $\b{x}^{(0)}$, after $t$ iterations, the convergence rate and iteration complexity of gradient descent are:
\begin{align}
&f(\b{x}^{(t)}) - f^* \leq (1 - \frac{\mu}{L})^t \big(f(\b{x}^{(0)}) - f^*\big) \nonumber \\
&\implies f(\b{x}^{(t)}) - f^* = \mathcal{O}\big((1 - \frac{\mu}{L})^t\big), \label{equation_GD_norm_gradient_sublinear_rate_stronglyConvexFunction} \\
&t = \mathcal{O}(\log \frac{1}{\epsilon}), \label{equation_GD_stronglyConvexFunction_iterationComplexity}
\end{align}
respectively, where $f^*$ is the minimum of cost function. 
It means that gradient descent has linear convergence rate (see Definition \ref{definition_convergence_rate}) for strongly convex functions. 
\end{theorem}

Note that some convergence proofs and analyses for gradient descent can be found in \cite{gower2018convergence}.

\subsubsection{Gradient Descent with Momentum}\label{section_momentum}


Gradient descent and other first-order methods can have a momentum term. 
Momentum, proposed in \cite{rumelhart1986learning}, makes the change of solution $\Delta \b{x}$ a little similar to the previous change of solution. Hence, the change adds a history of previous change to Eq. (\ref{equation_GD_step_by_eta}):
\begin{align}\label{equation_GD_momentum_update}
&(\Delta \b{x})^{(k)} := \alpha (\Delta \b{x})^{(k-1)} - \eta^{(k)} \nabla f(\b{x}^{(k)}), 
\end{align}
where $\alpha>0$ is the momentum parameter which weights the importance of history compared to the descent direction. We use this $(\Delta \b{x})^{(k)}$ in Eq. (\ref{equation_update_point_numerical_optimization}) for updating the solution.
Because of faithfulness to the track of previous updates, momentum reduces the amount of oscillation of updates in gradient descent optimization. 

\subsubsection{Steepest Descent}


\textit{Steepest descent} is similar to gradient descent but there is a difference between them. In steepest descent, we move toward the negative gradient as much as possible to reach the smallest function value which can be achieved at every iteration. Hence, the step size at iteration $k$ of steepest descent is calculated as \cite{chong2004introduction}:
\begin{align}\label{equation_steepest_descent_update}
\eta^{(k)} := \arg\min_{\eta}\, f\big(\b{x}^{(k)} - \eta \nabla f(\b{x}^{(k)})\big),
\end{align}
and then, the solution is updated using Eq. (\ref{equation_GD_step_by_eta}) as in gradient descent. 

Another interpretation of steepest descent is as follows, according to {\citep[Chapter 9.4]{boyd2004convex}}. The first-order Taylor expansion of function is $f(\b{x} + \b{v}) \approx f(\b{x}) + \nabla f(\b{x})^\top \b{v}$. Hence, the step size in the normalized steepest descent, at iteration $k$, is obtained as:
\begin{align}
\Delta \b{x} = \arg\min_{\b{v}}\{\nabla f(\b{x}^{(k)})^\top \b{v}\,|\, \|\b{v}\|_2 \leq 1\},
\end{align}
which is used in Eq. (\ref{equation_update_point_numerical_optimization}) for updating the solution.

\subsubsection{Backpropagation}\label{section_backpropagation}

Backpropagation \cite{rumelhart1986learning} is the most well-known optimization method used in neural networks. It is actually gradient descent with chain rule in derivatives because of having layers of parameters. 
Consider Fig. \ref{figure_network_layers} which shows three neurons in three layers of a network. 
Let $x_{ji}$ denote the weight connecting neuron $i$ to neuron $j$. Let $a_i$ and $z_i$ be the output of neuron $i$ before and after applying its activation function $\sigma_i(.): \mathbb{R} \rightarrow \mathbb{R}$, respectively. In other words, $z_i := \sigma_i(a_i)$.

\begin{figure}[!t]
\centering
\includegraphics[width=3in]{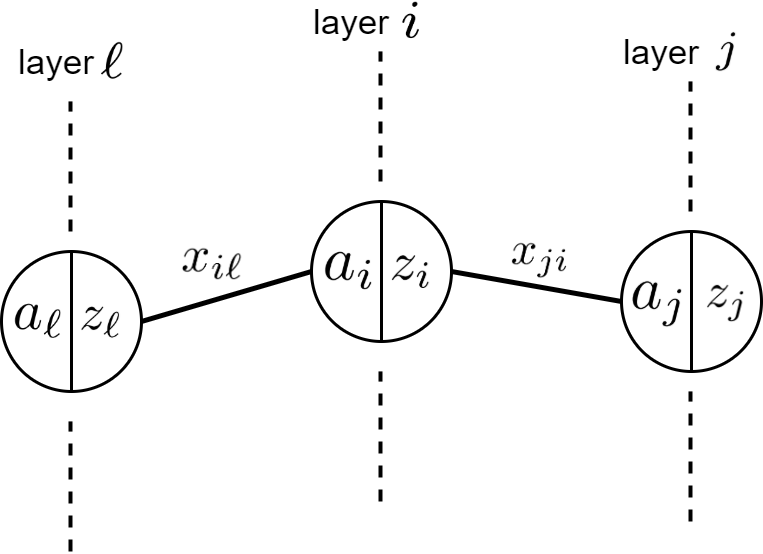}
\caption{Neurons in three layers of a neural network.}
\label{figure_network_layers}
\end{figure}

According to neural network, we have $a_i = \sum_{\ell} x_{i\ell} z_\ell$ which sums over the neurons in layer $\ell$.
By chain rule, the gradient of error $e$ w.r.t. to the weight between neurons $\ell$ and $i$ is:
\begin{align}\label{equation_backpropagation_partial_e_partial_x}
& \frac{\partial e}{\partial x_{i\ell}} = \frac{\partial e}{\partial a_i} \times \frac{\partial a_i}{\partial x_{i\ell}} \overset{(a)}{=} \delta_i \times z_\ell,
\end{align}
where $(a)$ is because $a_i = \sum_{\ell} x_{i\ell} z_\ell$ and we define $\delta_i := \partial e / \partial a_i$. 
If layer $i$ is the last layer, $\delta_i$ can be computed by derivative of error (loss function) w.r.t. the output. However, if $i$ is one of the hidden layers, $\delta_i$ is computed by chain rule as:
\begin{align}\label{equation_backpropagation_delta_i}
\delta_i = \frac{\partial e}{\partial a_i} = \sum_j \Big(\frac{\partial e}{\partial a_j} \times \frac{\partial a_j}{\partial a_i}\Big) = \sum_j \Big(\delta_j \times \frac{\partial a_j}{\partial a_i}\Big).
\end{align}
The term $\partial a_j / \partial a_i$ is calculated by chain rule as:
\begin{align}\label{equation_backpropagation_partial_a_partial_a}
& \frac{\partial a_j}{\partial a_i} = \frac{\partial a_j}{\partial z_i} \times \frac{\partial z_i}{\partial a_i} \overset{(a)}{=} x_{ji}\, \sigma'(a_i),
\end{align}
where $(a)$ is because $a_j = \sum_i x_{ji} z_i$ and $z_i = \sigma(a_i)$ and $\sigma'(.)$ denotes the derivative of activation function. 
Putting Eq. (\ref{equation_backpropagation_partial_a_partial_a}) in Eq. (\ref{equation_backpropagation_delta_i}) gives:
\begin{align*}
& \delta_i = \sigma'(a_i) \sum_j (\delta_j\, x_{ji}). 
\end{align*}
Putting this equation in Eq. (\ref{equation_backpropagation_partial_e_partial_x}) gives:
\begin{align}\label{equation_backpropagation_gradient}
& \frac{\partial e}{\partial x_{i\ell}} = z_\ell\, \sigma'(a_i) \sum_j (\delta_j\, x_{ji}).
\end{align}
Backpropagation uses the gradient in Eq. (\ref{equation_backpropagation_gradient}) for updating the weight $x_{i\ell}, \forall i, \ell$ by gradient descent:
\begin{align*}
x_{i\ell}^{(k+1)} := x_{i\ell}^{(k)} - \eta^{(k)} \frac{\partial e}{\partial x_{i\ell}}. 
\end{align*}
This tunes the weights from last layer to the first layer for every iteration of optimization. 

\subsection{Accelerated Gradient Method}



It was shown in the literature that gradient descent is not optimal in convergence rate and can be improved. 
It was at that time that Nesterov proposed \textit{Accelerated Gradient Method (AGM)} \cite{nesterov1983method} to make the convergence rate of gradient descent optimal {\citep[Chapter 2.2]{nesterov2003introductory}}. AGM is also called the \textit{Nesterov's accelerated gradient method} or \textit{Fast Gradient Method (FGM)}. A series of Nesterov's papers improved AGM \cite{nesterov1983method,nesterov1988approach,nesterov2005smooth,nesterov2013gradient}. 

Consider a sequence $\{\gamma^{(k)}\}$ which satisfies:
\begin{align}
& \prod_{i=0}^k (1 - \gamma^{(i)}) \geq (\gamma^{(k)})^2, \quad \forall k \geq 0, \gamma^{(k)} \in [0,1].
\end{align}
An example sequence, satisfying this condition, is $\gamma^{(0)} = \gamma^{(1)} = \gamma^{(2)} = \gamma^{(3)} = 0, \gamma^{(k)} = 2/k, \forall k \geq 4$.
The AGM updates the solution iteratively as \cite{nesterov1983method}:
\begin{align}
& \b{x}^{(k+1)} := \b{y}^{(k)} - \eta^{(k)} \nabla f(\b{y}^{(k)}), \\
& \b{y}^{(k+1)} := (1 - \gamma^{(k)}) \b{x}^{(k+1)} + \gamma^{(k)} \b{x}^{(k)},
\end{align}
until convergence.

\begin{theorem}[Convergence rate of AGM for convex functions {\citep[under Eq. 7]{nesterov1983method}}, {\citep[Theorem 3.19]{bubeck2014convex}}]\label{theorem_AGM_convergence_rate_convexFunction}
Consider a convex and differentiable function $f(.)$, with domain $\mathcal{D}$, whose gradient is $L$-smooth (see Eq. (\ref{equation_gradient_L_smooth})).
Starting from the initial point $\b{x}^{(0)}$, after $t$ iterations of AGM, we have:
\begin{align}\label{equation_AGM_upperbound_min_norm_gradient_convexFunction}
f(\b{x}^{(t+1)}) - f^* \leq \frac{2 L \|\b{x}^{(0)} - \b{x}^*\|_2^2}{(t+1)^2} = \mathcal{O}(\frac{1}{t^2}),
\end{align}
where $f^*$ is the minimum of cost function and $\b{x}^*$ is the minimizer. 
It means the distance of convex function value to its optimum has sublinear convergence (see Definition \ref{definition_convergence_rate}) in AGM.
\end{theorem}
Comparing Eqs. (\ref{equation_GD_upperbound_min_norm_gradient_convexFunction}) and (\ref{equation_AGM_upperbound_min_norm_gradient_convexFunction}) shows that AGM converges much faster than gradient descent. 
A book chapter on AGM is {\citep[Section 3.7]{bubeck2014convex}}.
Various versions of AGM have been unified in \cite{tseng2008accelerated}. 
Moreover, the connection of AGM with ordinary differential equations has been investigated in \cite{su2016differential}.

\subsection{Stochastic Gradient Methods}

\subsubsection{Stochastic Gradient Descent}



Assume we have a dataset of $n$ data points, $\{\b{a}_i \in \mathbb{R}^d\}_{i=1}^n$ and their labels $\{l_i \in \mathbb{R}\}_{i=1}^n$.
Let the cost function $f(.)$ be decomposed into summation of $n$ terms $\{f_i(\b{x})\}_{i=1}^n$. Some well-known examples for the cost function terms are:
\begin{itemize}
\item Least squares error: $f_i(\b{x}) = 0.5 (\b{a}_i^\top \b{x} - l_i)^2$,
\item Absolute error: $f_i(\b{x}) = \b{a}_i^\top \b{x} - l_i$,
\item Hinge loss (for $l_i \in \{-1,1\}$): $f_i(\b{x}) = \max(0, 1 - l_i \b{a}_i^\top \b{x})$.
\item Logistic loss (for $l_i \in \{-1,1\}$): $\log(\frac{1}{1 + \exp(-l_i \b{a}_i^\top \b{x})})$.
\end{itemize}
The optimization problem (\ref{equation_optimization_problem_unconstrained}) becomes:
\begin{equation}\label{equation_optimization_problem_unconstrained_SGD}
\begin{aligned}
& \underset{\b{x}}{\text{minimize}}
& & \frac{1}{n} \sum_{i=1}^n f_i(\b{x}).
\end{aligned}
\end{equation}
In this case, the full gradient is the average gradient, i.e:
\begin{align}\label{equation_gradient_average_of_gradients}
\nabla f(\b{x}) = \frac{1}{n} \sum_{i=1}^n \nabla f_i(\b{x}),
\end{align}
so Eq. (\ref{equation_GD_step_by_L}) becomes $\Delta \b{x} = -(1/(Ln)) \sum_{i=1}^n \nabla f_i(\b{x}^{(k)})$.
This is what gradient descent uses in Eq. (\ref{equation_update_point_numerical_optimization}) for updating the solution at every iteration. However, calculation of this full gradient is time-consuming and inefficient for large values of $n$, especially as it needs to be recalculated at every iteration. 
\textit{Stochastic Gradient Descent (SGD)}, also called \textit{stochastic gradient method}, approximates gradient descent stochastically and samples (i.e. bootstraps) one of the points at every iteration for updating the solution. Hence, it uses:
\begin{align}\label{equation_SGD_step_by_eta}
\b{x}^{(k+1)} := \b{x}^{(k)} - \eta^{(k)} \nabla f_i(\b{x}^{(k)}),
\end{align}
ragther than Eq. (\ref{equation_GD_step_by_eta}). 
The idea of stochastic approximation was first proposed in \cite{robbins1951stochastic}. It was first used for machine learning in \cite{bottou1998online}. 

As Eq. (\ref{equation_SGD_step_by_eta}) states, SGD often uses an adaptive step size which changes in every iteration. The step size can be decreasing because in initial iterations, where we are far away from the optimal solution, the step size can be large; however, it should be small in the last iterations which is supposed to be close to the optimal solution. Some well-known adaptations for the step size are:
\begin{align}
& \eta^{(k)} := \frac{1}{k}, \quad \eta^{(k)} := \frac{1}{\sqrt{k}}, \quad \eta^{(k)} := \eta. 
\end{align}

\begin{theorem}[Convergence rates for SGD]
Consider a function $f(\b{x}) = \sum_{i=1}^n f_i(\b{x})$ and which is bounded below and each $f_i$ is differentiable. Let the domain of function $f(.)$ be $\mathcal{D}$ and its gradient be $L$-smooth (see Eq. (\ref{equation_gradient_L_smooth})). Assume $\mathbb{E}[\|\nabla f_i(\b{x}_k)\|_2^2\,|\,\b{x}_k] \leq \beta^2$ where $\beta$ is a constant. Depending on the step size, the convergence rate of SGD is:
\begin{align}
& \mathcal{O}(\frac{1}{\log t}) \quad \text{if} \quad \eta^{(k)} = \frac{1}{k}, \\
& \mathcal{O}(\frac{\log t}{\sqrt{t}}) \quad \text{if} \quad \eta^{(k)} = \frac{1}{\sqrt{k}}, \\
& \mathcal{O}(\frac{1}{t} + \eta) \quad \text{if} \quad \eta^{(k)} = \eta, \label{equation_convergence_rate_SGD_3}
\end{align}
where $t$ denotes the iteration index. 
If the functions $f_i$'s are $\mu$-strongly convex, then the convergence rate of SGD is:
\begin{align}
& \mathcal{O}(\frac{1}{t}) \quad \text{if} \quad \eta^{(k)} = \frac{1}{\mu k}, \\
& \mathcal{O}\big((1 - \frac{\mu}{L})^t + \eta\big) \quad \text{if} \quad \eta^{(k)} = \eta. \label{equation_convergence_rate_SGD_3_stronglyConvex}
\end{align}
\end{theorem}

Eqs. (\ref{equation_convergence_rate_SGD_3}) and (\ref{equation_convergence_rate_SGD_3_stronglyConvex}) shows that with a fixed step size $\eta$, SGD converges sublinearly for a non-convex function and linearly for a strongly convex function (see Definition \ref{definition_convergence_rate}) in the initial iterations. However, in the late iterations, it stagnates to a neighborhood around the optimal point and never reaches it. Hence, SGD has less accuracy than gradient descent. 
The advantage of SGD over gradient descent is that its every iteration is much faster than every iteration of gradient descent because of less computations for gradient. This faster pacing of every iteration shows off more when $n$ is huge. In summary, SGD has fast convergence to low accurate optimal point. 

It is noteworthy that the full gradient is not available in SGD to use for checking convergence, as discussed in Section \ref{section_convergence_criterion}. One can use other criteria in that section or merely check the norm of gradient for the sampled point. 
Moreover, note that SGD can be used with the line-search methods, too. 
SGD can also use a momentum term (see Section \ref{section_momentum}).

\subsubsection{Mini-batch Stochastic Gradient Descent}


Gradient descent uses the entire $n$ data points and SGD uses one randomly sampled point at every iteration. 
For large datasets, gradient descent is very slow and intractable in every iteration while SGD will need a significant number of iterations to roughly cover all data. Besides, SGD has low accuracy in convergence to the optimal point. 
We can have a middle case where we use a batch of $b$ randomly sampled points at every iteration. 
This method is named the \textit{mini-batch SGD} or the \textit{hybrid deterministic-stochastic gradient} method. 
This batch-wise approach is wise for large datasets because of the mentioned problems gradient descent and SGD face in big data optimization \cite{bottou2018optimization}. 

Usually, before start of optimization, the $n$ data points are randomly divided into $\lfloor n / b \rfloor$ batches of size $b$. This is equivalent to simple random sampling for sampling points into batches without replacement. 
We denote the dataset by $\mathcal{D}$ (where $|\mathcal{D}|=n$) and the $i$-th batch by $\mathcal{B}_i$ (where $|\mathcal{B}_i|=b$). The batches are disjoint:
\begin{align}
&\bigcup_{i=1}^{\lfloor n / b \rfloor} \mathcal{B}_i = \mathcal{D}, \label{equation_minibatch_SGD_batches1} \\
& \mathcal{B}_i \cap \mathcal{B}_j = \varnothing, ~~~ \forall i,j \in \{1, \dots, \lfloor n / b \rfloor\}, ~ i \neq j. \label{equation_minibatch_SGD_batches2}
\end{align}
Another less-used approach for making batches is to sample points for a batch during optimization. This is equivalent to bootstrapping for sampling points into batches with replacement. In this case, the batches are not disjoint anymore and Eqs. (\ref{equation_minibatch_SGD_batches1}) and (\ref{equation_minibatch_SGD_batches2}) do not hold. 

\begin{definition}[Epoch]
In mini-batch SGD, when all $\lfloor n / b \rfloor$ batches of data are used for optimization once, an \textit{epoch} is completed. After completion of an epoch, the next epoch is started and epochs are repeated until convergence of optimization. 
\end{definition}

In mini-batch SGD, if the $k$-th iteration of optimization is using the $k'$-th batch, the update of solution is done as:
\begin{align}\label{equation_minibatch_SGD_step_by_eta}
\b{x}^{(k+1)} := \b{x}^{(k)} - \eta^{(k)} \frac{1}{b} \sum_{i \in \mathcal{B}_{k'}} \nabla f_i(\b{x}^{(k)}).
\end{align}
The scale factor $1/b$ is sometimes dropped for simplicity. 
Mini-batch SGD is used significantly in machine learning, especially in neural networks \cite{bottou1998online,goodfellow2016deep}. Because of dividing data into batches, mini-batch SGD can be solved on parallel servers as a distributed optimization method. 

\begin{theorem}[Convergence rates for mini-batch SGD]
Consider a function $f(\b{x}) = \sum_{i=1}^n f_i(\b{x})$ which is bounded below and each $f_i$ is differentiable. Let the domain of function $f(.)$ be $\mathcal{D}$ and its gradient be $L$-smooth (see Eq. (\ref{equation_gradient_L_smooth})) and assume $\eta^{(k)} = \eta$ is fixed.
The batch-wise gradient is an approximation to the full gradient with some error $e_t$ for the $t$-th iteration:
\begin{align}\label{equation_minibatch_SGD_error}
\frac{1}{b} \sum_{i \in \mathcal{B}_{t'}} \nabla f_i(\b{x}^{(t)}) = \nabla f(\b{x}^{(t)}) + e_t.
\end{align}
The convergence rate of mini-batch SGD for non-convex and convex functions are:
\begin{align}\label{equation_minibatch_SGD_convergence_rate_nonconvex}
& \mathcal{O}\big(\frac{1}{t} + \|e_t\|_2^2\big),
\end{align}
where $t$ denotes the iteration index. 
If the functions $f_i$'s are $\mu$-strongly convex, then the convergence rate of mini-batch SGD is:
\begin{align}\label{equation_minibatch_SGD_convergence_rate_strongly_convex}
& \mathcal{O}\big((1 - \frac{\mu}{L})^t + \|e_t\|_2^2\big).
\end{align}
\end{theorem}

According to Eq. (\ref{equation_minibatch_SGD_error}), the expected error of mini-batch SGD at the $k$-th iteration is:
\begin{align}\label{equation_minibatch_SGD_error_expectation}
\mathbb{E}[\|e_t\|_2^2] = \mathbb{E}\Big[\big\|\nabla f(\b{x}^{(t)}) - \frac{1}{b} \sum_{i \in \mathcal{B}_{t'}} \nabla f_i(\b{x}^{(t)})\big\|_2^2\Big],
\end{align}
which is variance of estimation. 
If we sample the batches without replacement (i.e., sampling batches by simple random sampling before start of optimization) or with replacement (i.e., bootstrapping during optimization), the expected error is {\citep[Proposition 3]{ghojogh2020sampling}}:
\begin{align}
&\mathbb{E}[\|e_t\|_2^2] = (1 - \frac{b}{n}) \frac{\sigma^2}{b}, \label{equation_minibatch_SGD_without_replacement} \\
&\mathbb{E}[\|e_t\|_2^2] = \frac{\sigma^2}{b}, \label{equation_minibatch_SGD_with_replacement}
\end{align}
respectively, where $\sigma^2$ is the variance of whole dataset. 
According to Eqs. (\ref{equation_minibatch_SGD_without_replacement}) and (\ref{equation_minibatch_SGD_with_replacement}), the accuracy of SGD by sampling without and with replacement increases by $b \rightarrow n$ and $b \rightarrow \infty$, respectively. However, this increase makes every iteration slower so there is trade-off between accuracy and speed. 
Also, comparing Eqs. (\ref{equation_minibatch_SGD_convergence_rate_nonconvex}) and (\ref{equation_minibatch_SGD_convergence_rate_strongly_convex}) with Eqs. (\ref{equation_GD_norm_gradient_sublinear_rate}) and (\ref{equation_GD_norm_gradient_sublinear_rate_stronglyConvexFunction}), while noticing Eqs. (\ref{equation_minibatch_SGD_without_replacement}) and (\ref{equation_minibatch_SGD_with_replacement}), shows that the convergence rate of mini-batch gets closer to that of gradient descent if the batch size increases. 

\subsection{Stochastic Average Gradient Methods}

\subsubsection{Stochastic Average Gradient}


SGD is faster than gradient descent but its problem is its lower accuracy compared to gradient descent. \textit{Stochastic Average Gradient} (SAG) \cite{roux2012stochastic} keeps a trade-off between accuracy and speed. 
Consider the optimization problem (\ref{equation_optimization_problem_unconstrained_SGD}). 
Let $\nabla f_{i}(\b{x}^{(k)})$ be the gradient of $f_i(.)$, evaluated in point $\b{x}^{(k)}$, at iteration $k$. 
According to Eqs. (\ref{equation_GD_step_by_eta}) and (\ref{equation_gradient_average_of_gradients}), gradient descent updates the solution as:
\begin{align*}
\b{x}^{(k+1)} := \b{x}^{(k)} - \frac{\eta^{(k)}}{n} \sum_{i=1}^n \nabla f_{i}(\b{x}^{(k)}).
\end{align*}
SAG randomly samples one of the points and updates its gradient among the gradient terms. If the sampled point is the $j$-th one, we have:
\begin{equation}\label{equation_update_SAG}
\begin{aligned}
\b{x}^{(k+1)} := \b{x}^{(k)} - \frac{\eta^{(k)}}{n} &\Big( \nabla f_{j}(\b{x}^{(k)}) - \nabla f_{j}(\b{x}^{(k-1)}) \\
&+ \sum_{i=1}^n \nabla f_{i}(\b{x}^{(k-1)}) \Big).
\end{aligned}
\end{equation}
In other words, we subtract the $j$-th gradient from the summation of all $n$ gradients in previous iteration $(k-1)$ by $\sum_{i=1}^n \nabla f_{i}(\b{x}^{(k-1)}) - \nabla f_{j}(\b{x}^{(k-1)})$; then, we add back the new $j$-th gradient in this iteration by adding $\nabla f_{j}(\b{x}^{(k)})$. 

\begin{theorem}[Convergence rates for SAG {\citep[Proposition 1]{roux2012stochastic}}]
Consider a function $f(\b{x}) = \sum_{i=1}^n f_i(\b{x})$ which is bounded below and each $f_i$ is differentiable. Let the domain of function $f(.)$ be $\mathcal{D}$ and its gradient be $L$-smooth (see Eq. (\ref{equation_gradient_L_smooth})). 
The convergence rate of SAG is $\mathcal{O}(1/t)$ where $t$ denotes the iteration index. 
\end{theorem}
Comparing the convergence rates of SAG, gradient descent, and SGD shows that SAG has the same rate order as gradient descent; although, it usually needs some more iterations to converge. 
Practical experiments have shown that SAG requires many parameter fine-tuning to perform perfectly. 
Some other variants of SAG are optimization of a finite sum of smooth convex functions \cite{schmidt2017minimizing} and its second-order version named Stochastic Average Newton (SAN) \cite{chen2021san}.

\subsubsection{Stochastic Variance Reduced Gradient}


Another effective first-order method is the Stochastic Variance Reduced Gradient (SVRG) \cite{johnson2013accelerating} which updates the solution according to Algorithm \ref{algorithm_SVRG}.
As this algorithm shows, the update of solution is similar to SAG (see Eq. (\ref{equation_update_SAG})) but for every iteration, it updates the solution for $m$ times. SVRG is an efficient method and its convergence rate is similar to that of SAG. It is shown in \cite{johnson2013accelerating} that both SAG and SVRG reduce the variance of solution to optimization. 
Recently, SVRG has been used for semidefinite programming optimization \cite{zeng2021stochastic}. 

\SetAlCapSkip{0.5em}
\IncMargin{0.8em}
\begin{algorithm2e}[!t]
\DontPrintSemicolon
    Initialize $\widetilde{\b{x}}^{(0)}$\;
    \For{iteration $k = 1, 2, \dots$}{
        $\widetilde{\b{x}} := \widetilde{\b{x}}^{(k-1)}$\;
        $\nabla f(\widetilde{\b{x}}) \overset{(\ref{equation_gradient_average_of_gradients})}{:=} \frac{1}{n} \sum_{i=1}^n \nabla f_i(\widetilde{\b{x}})$\;
        $\b{x}^{(0)} := \widetilde{\b{x}}$\;
        \For{iteration $\tau = 0, 1, \dots, m-1$}{
            Randomly sample $j$ from $\{1, \dots, n\}$.\;
            $\b{x}^{(\tau+1)} := \b{x}^{(\tau)} - \eta^{(\tau)} \big( \nabla f_{j}(\b{x}^{(\tau)}) - \nabla f_{j}(\widetilde{\b{x}}) + \nabla f(\widetilde{\b{x}}) \big)$.\;
        }
        $\widetilde{\b{x}}^{(k)} := \b{x}^{(m)}$\;
    }
\caption{The SVRG algorithm}\label{algorithm_SVRG}
\end{algorithm2e}
\DecMargin{0.8em}

\subsubsection{Adapting Learning Rate with AdaGrad, RMSProp, and Adam}


Consider the optimization problem (\ref{equation_optimization_problem_unconstrained_SGD}).
We can adapt the learning rate in stochastic gradient methods. In the following, we introduce the three most well-known methods for adapting the learning rate, which are AdaGrad, RMSProp, and Adam.

\hfill\break
\textbf{-- Adaptive Gradient (AdaGrad):}
\textit{AdaGrad} method \cite{duchi2011adaptive} updates the solution iteratively as:
\begin{align}\label{equation_AdaGrad_update}
\b{x}^{(k+1)} := \b{x}^{(k)} - \eta^{(k)} \b{G}^{-1} \nabla f_{i}(\b{x}^{(k)}),
\end{align}
where $\b{G}$ is a $(d \times d)$ diagonal matrix whose $(j,j)$-th element is:
\begin{align}\label{equation_AdaGrad_G}
\b{G}(j,j) := \sqrt{\varepsilon + \sum_{\tau=0}^k \big(\nabla_j f_{i_\tau}(\b{x}^{(\tau)})\big)^2},
\end{align}
where $\varepsilon \geq 0$ is for stability (making $\b{G}$ full rank), $i_\tau$ is the randomly sampled point (from $\{1, \dots, n\}$) at iteration $\tau$, and $\nabla_j f_{i_\tau}(.)$ is the partial derivative of $f_{i_\tau}(.)$ w.r.t. its $j$-th element (n.b. $f_{i_\tau}(.)$ is $d$-dimensional). 
Putting Eq. (\ref{equation_AdaGrad_G}) in Eq. (\ref{equation_AdaGrad_update}) can simplify AdaGrad to:
\begin{align}\label{equation_AdaGrad_update_2}
\b{x}_j^{(k+1)} := \b{x}_j^{(k)} - \frac{\eta^{(k)}}{\sqrt{\varepsilon + \sum_{\tau=0}^k \big(\nabla_j f_{i_\tau}(\b{x}^{(\tau)})\big)^2}} \nabla f_{j}(\b{x}_j^{(k)}).
\end{align}
AdaGrad keeps a history of the sampled points and it takes derivative for them to use. During the iterations so far, if a dimension has changed significantly, it dampens the learning rate for that dimension (see the inverse in Eq. (\ref{equation_AdaGrad_update})); hence, it gives more weight for changing the dimensions which have not changed noticeably. 



\hfill\break
\textbf{-- Root Mean Square Propagation (RMSProp):}
RMSProp was first proposed in \cite{tieleman2012lecture} which is unpublished. 
It is an improved version of Rprop (resilient backpropagation) \cite{riedmiller1992rprop} which uses the sign of gradient in optimization.  
Inspired by momentum in Eq. (\ref{equation_GD_momentum_update}), it updates a scalar variable $v$ as \cite{hinton2012neural}:
\begin{align}
v^{(k+1)} := \gamma v^{(k)} + (1 - \gamma) \|\nabla f_i(\b{x}^{(k)})\|_2^2,
\end{align}
where $\gamma \in [0,1]$ is the forgetting factor (e.g. $\gamma = 0.9$).
Then, it uses this $v$ to weight the learning rate:
\begin{align}\label{equation_RMSProp_update}
\b{x}^{(k+1)} := \b{x}^{(k)} - \frac{\eta^{(k)}}{\sqrt{\varepsilon + v^{(k+1)}}} \nabla f_{j}(\b{x}_j^{(k)}),
\end{align}
where $\epsilon \geq 0$ is for stability to not have division by zero.
Comparing Eqs. (\ref{equation_AdaGrad_update_2}) and (\ref{equation_RMSProp_update}) shows that RMSProp has a similar form to AdaGrad.

\hfill\break
\textbf{-- Adaptive Moment Estimation (Adam):}
Adam optimizer \cite{kingma2014adam} improves over RMSProp by adding a momentum term (see Section \ref{section_momentum}). 
It updates the scalar $v$ and the vector $\b{m} \in \mathbb{R}^d$ as:
\begin{align}
& \b{m}^{(k+1)} := \gamma_1 \b{m}^{(k)} + (1 - \gamma_1) \nabla f_i(\b{x}^{(k)}), \\
& v^{(k+1)} := \gamma_2 v^{(k)} + (1 - \gamma_2) \|\nabla f_i(\b{x}^{(k)})\|_2^2,
\end{align}
where $\gamma_1, \gamma_2 \in [0,1]$. It normalizes these variables as:
\begin{align*}
& \widehat{\b{m}}^{(k+1)} := \frac{1}{1-\gamma_1^k} \b{m}^{(k+1)}, \quad \widehat{v}^{(k+1)} := \frac{1}{1-\gamma_2^k} v^{(k+1)}. 
\end{align*}
Then, it updates the solution as:
\begin{align}\label{equation_Adam_update}
\b{x}^{(k+1)} := \b{x}^{(k)} - \frac{\eta^{(k)}}{\sqrt{\varepsilon + \widehat{v}^{(k+1)}}} \widehat{\b{m}}^{(k+1)},
\end{align}
which is stochastic gradient descent with momentum while using RMSProp. 
Convergences of RMSProp and Adam methods have been discussed in \cite{zou2019sufficient}. 
The Adam optimizer is one of the mostly used optimizers in neural networks. 

\subsection{Proximal Methods}\label{section_proximal_methods}


\subsubsection{Proximal Mapping and Projection}

\begin{definition}[Proximal mapping/operator \cite{parikh2014proximal}]
The proximal mapping or proximal operator of a convex function $g(.)$ is:
\begin{align}\label{equation_proximal_mapping}
\textbf{prox}_g(\b{x}) := \arg\min_{\b{u}} \Big( g(\b{u}) + \frac{1}{2} \|\b{u} - \b{x}\|_2^2 \Big).
\end{align}
In case the function $g(.)$ is scaled by a scalar $\lambda$ (e.g., this often holds in Eq. (\ref{equation_composite_optimization}) where $\lambda$ can scale $g(.)$ as the regularization parameter), the proximal mapping is defined as:
\begin{align}\label{equation_proximal_mapping_scaled}
\textbf{prox}_{\lambda g}(\b{x}) := \arg\min_{\b{u}} \Big( g(\b{u}) + \frac{1}{2\lambda} \|\b{u} - \b{x}\|_2^2 \Big).
\end{align}
\end{definition}

The proximal mapping is related to the Moreau-Yosida regularization defined below. 
\begin{definition}[Moreau-Yosida regularization or Moreau envelope \cite{moreau1965proximite,yosida1965functional}]\label{definition_Moreau_envelope}
The Moreau-Yosida regularization or the Moreau envelope of function $g(.)$ is:
\begin{align}\label{equation_Moreau_envelope}
M_{\lambda g}(\b{x}) := \inf_{\b{u}} \Big( g(\b{u}) + \frac{1}{2} \|\b{u} - \b{x}\|_2^2 \Big).
\end{align}
This Moreau-Yosida regularized function has the same minimizer as the function $g(.)$ \cite{lemarechal1997practical}. 
\end{definition}

\begin{lemma}[Moreau decomposition \cite{moreau1962decomposition}]
We always have the following decomposition, named the Moreau decomposition:
\begin{align}
& \b{x} = \textbf{prox}_{g}(\b{x}) + \textbf{prox}_{g^*}(\b{x}), \label{equation_Moreau_decomposition} \\
& \b{x} = \textbf{prox}_{\lambda g}(\b{x}) + \lambda\, \textbf{prox}_{\frac{1}{\lambda} g^*}(\frac{\b{x}}{\lambda}), \label{equation_Moreau_decomposition_with_lambda}
\end{align}
where $g(.)$ is a function in a space and $g^*(.)$ is its corresponding function in the dual space (e.g., if $g(.)$ is a norm, $g^*(.)$ is its dual norm or if $g(.)$ is projection onto a cone, $g^*(.)$ is projection onto the dual cone). 
\end{lemma}

\begin{lemma}[Projection onto set]\label{lemma_projection_onto_set}
Consider an indicator function $\mathbb{I}(.)$ which is zero if its condition is satisfied and is infinite otherwise. 
The proximal mapping of the indicator function to a convex set $\mathcal{S}$, i.e. $\mathbb{I}(\b{x} \in \mathcal{S})$, is projection of the point $\b{x}$ onto the set $\mathcal{S}$. Hence, projection of $\b{x}$ onto set $\mathcal{S}$, denoted by $\Pi_{\mathcal{S}}(\b{x})$, is defined as:
\begin{align}\label{equation_projection_onto_set}
\Pi_{\mathcal{S}}(\b{x}) := \textbf{prox}_{\mathbb{I}(. \in \mathcal{S})}(\b{x}) = \arg\min_{\b{u} \in \mathcal{S}} \Big( \frac{1}{2} \|\b{u} - \b{x}\|_2^2 \Big).
\end{align}
As Fig. \ref{figure_projection_onto_set} shows, this projection simply means projecting the point $\b{x}$ onto the closest point of set from the point $\b{x}$; hence, the vector connecting the points $\b{x}$ and $\Pi_{\mathcal{S}}(\b{x})$ is orthogonal to the set $\mathcal{S}$.
\end{lemma}
\begin{proof}
\begin{align*}
\textbf{prox}_{\mathbb{I}(. \in \mathcal{S})}(\b{x}) &= \arg\min_{\b{u}} \Big( \mathbb{I}(\b{x} \in \mathcal{S}) + \frac{1}{2} \|\b{u} - \b{x}\|_2^2 \Big) \\
&\overset{(a)}{=} \arg\min_{\b{u} \in \mathcal{S}} \Big( \frac{1}{2} \|\b{u} - \b{x}\|_2^2 \Big),
\end{align*}
where $(a)$ is because $\mathbb{I}(\b{x} \in \mathcal{S})$ becomes infinity if $\b{x} \not\in \mathcal{S}$. Q.E.D.
\end{proof}

\begin{figure}[!t]
\centering
\includegraphics[width=1.7in]{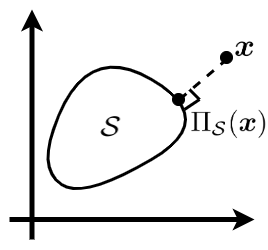}
\caption{Projection of point $\b{x}$ onto a set $\mathcal{S}$.}
\label{figure_projection_onto_set}
\end{figure}

\begin{corollary}[Moreau decomposition for norm]
If the function is a scaled norm, $g(.) = \lambda \|.\|$, we have from re-arranging Eq. (\ref{equation_Moreau_decomposition_with_lambda}) that:
\begin{align}\label{equation_Moreau_decomposition_norm}
& \textbf{prox}_{\lambda \|.\|}(\b{x}) = \b{x} - \lambda\, \Pi_\mathcal{B}(\frac{\b{x}}{\lambda}), 
\end{align}
where $\mathcal{B}$ is the unit ball of dual norm (see Definition \ref{definition_unit_ball}). 
\end{corollary}

Derivation of proximal operator for various $g(.)$ functions are available in {\citep[Chapter 6]{beck2017first}}.
Here, we review the proximal mapping of some mostly used functions.
If $g(\b{x}) = 0$, proximal mapping becomes an identity mapping:
\begin{align*}
\textbf{prox}_{\lambda 0}(\b{x}) \overset{(\ref{equation_proximal_mapping_scaled})}{=} \arg\min_{\b{u}} \Big( \frac{1}{2\lambda} \|\b{u} - \b{x}\|_2^2 \Big) = \b{x}.
\end{align*}


\begin{lemma}[Proximal mapping of $\ell_2$ norm {\citep[Example 6.19]{beck2017first}}]\label{lemma_prox_l2_norm}
The proximal mapping of the $\ell_2$ norm is:
\begin{equation}\label{equation_proximal_mapping_scaled_l2_norm}
\begin{aligned}
\textbf{prox}_{\lambda \|.\|_2}(\b{x}) &= \Big( 1 - \frac{\lambda}{\max(\|\b{x}\|_2, \lambda)} \Big) \b{x} \\
&= 
\left\{
    \begin{array}{ll}
        \big( 1 - \frac{\lambda}{\|\b{x}\|_2} \big) \b{x} & \mbox{if } \|\b{x}\|_2 \geq \lambda \\
        0 & \mbox{if } \|\b{x}\|_2 < \lambda.
    \end{array}
\right.
\end{aligned}
\end{equation}
\end{lemma}
\begin{proof}
Let $g(.) = \|.\|_2$ and $\mathcal{B}$ be the unit $\ell_2$ ball (see Definition \ref{definition_unit_ball}) because $\ell_2$ is the dual norm of $\ell_2$ according to Eq. (\ref{equation_dual_norm_calculation}). We have:
\begin{align*}
&\Pi_\mathcal{B}(\b{x}) = 
\left\{
    \begin{array}{ll}
        \b{x} / \|\b{x}\|_2 & \mbox{if } \|\b{x}\|_2 \geq 1 \\
        \b{x} & \mbox{if } \|\b{x}\|_2 < 1.
    \end{array}
\right. \\
& \implies \textbf{prox}_{\lambda \|.\|_2}(\b{x}) \overset{(\ref{equation_Moreau_decomposition_norm})}{=} \b{x} - \lambda\, \Pi_\mathcal{B}(\frac{\b{x}}{\lambda}) \\
&~~~~~~~~~~~= 
\left\{
    \begin{array}{ll}
        \big( 1 - \frac{\lambda}{\|\b{x}\|_2} \big) \b{x} & \mbox{if } \|\b{x}\|_2 \geq \lambda \\
        \b{x} - \lambda (\b{x} / \lambda) = 0 & \mbox{if } \|\b{x}\|_2 < \lambda.
    \end{array}
\right.
\end{align*}
Q.E.D.
\end{proof}

\begin{lemma}[Proximal mapping of $\ell_1$ norm {\citep[Example 6.8]{beck2017first}}]\label{lemma_prox_l1_norm}
Let $x_j$ denote the $j$-th element of $\b{x} = [x_1, \dots, x_d]^\top \in \mathbb{R}^d$ and let $[\,\textbf{prox}_{\lambda \|.\|_1}(\b{x})]_j$ denote the $j$-th element of the $d$-dimensional $\textbf{prox}_{\lambda \|.\|_1}(\b{x})$ mapping. 
The $j$-th element of proximal mapping of the $\ell_1$ norm is:
\begin{equation}\label{equation_proximal_mapping_scaled_l1_norm}
\begin{aligned}
&[\,\textbf{prox}_{\lambda \|.\|_1}(\b{x})]_j = \max(0, |x_j|-\lambda)\, \textbf{sign}(x_j) \\
&~~~~~~~~~ = s_\lambda(x_j) := 
\left\{
    \begin{array}{ll}
        x_j - \lambda & \mbox{if } x_j \geq \lambda \\
        0 & \mbox{if } |x_j| < \lambda \\
        x_j + \lambda & \mbox{if } x_j \leq -\lambda,
    \end{array}
\right.
\end{aligned}
\end{equation}
for all $j \in \{1, \dots, d\}$.
Eq. (\ref{equation_proximal_mapping_scaled_l1_norm}) is called the \textit{soft-thresholding} function, denoted here by $s_\lambda(.)$. It is depicted in Fig. \ref{figure_soft_thresholding}. 
\end{lemma}
\begin{proof}
Let $g(.) = \|.\|_1$ and $\mathcal{B}$ be the unit $\ell_\infty$ ball (see Definition \ref{definition_unit_ball}) because $\ell_\infty$ is the dual norm of $\ell_1$ according to Eq. (\ref{equation_dual_norm_calculation}). The $j$-th element of projection is:
\begin{align*}
&[\Pi_\mathcal{B}(\b{x})]_j = 
\left\{
    \begin{array}{ll}
        1 & \mbox{if } x_j \geq 1 \\
        x_j & \mbox{if } |x_j| < 1 \\
        -1 & \mbox{if } x_j \leq -1 \\
    \end{array}
\right. \\
& \implies [\textbf{prox}_{\lambda \|.\|_1}(\b{x})]_j \overset{(\ref{equation_Moreau_decomposition_norm})}{=} x_j - \lambda\, [\Pi_\mathcal{B}(\frac{\b{x}}{\lambda})]_j \\
&~~~~~~~~~~~= 
\left\{
    \begin{array}{ll}
        x_j - \lambda & \mbox{if } x_j \geq \lambda \\
        0 & \mbox{if } |x_j| < \lambda \\
        x_j + \lambda & \mbox{if } x_j \leq -\lambda.
    \end{array}
\right.
\end{align*}
Q.E.D.
\end{proof}

\begin{figure}[!t]
\centering
\includegraphics[width=2.5in]{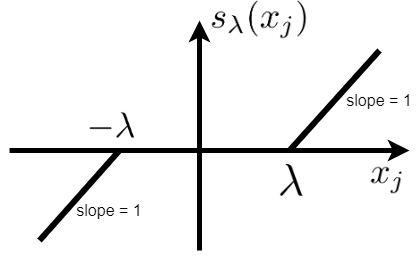}
\caption{Soft-thresholding function.}
\label{figure_soft_thresholding}
\end{figure}

\subsubsection{Proximal Point Algorithm}

The term $g(\b{u}) + 1 / (2\lambda) \|\b{u} - \b{x}\|_2^2$ in Eq. (\ref{equation_proximal_mapping_scaled}) is strongly convex; hence, the proximal point, $\textbf{prox}_{\lambda g}(\b{x})$, is unique. 

\begin{lemma}
The point $\b{x}^*$ minimizes the function $f(.)$ if and only if $\b{x}^* = \textbf{prox}_{\lambda f}(\b{x}^*)$.
In other words, the optimal point $\b{x}^*$ is the fixed point of the $\textbf{prox}_{\lambda f}(.)$ operator (see Definition \ref{definition_fixed_point}). 
\end{lemma}

Consider the optimization problem (\ref{equation_optimization_problem_unconstrained}). 
\textit{Proximal point algorithm}, also called \textit{proximal minimization}, was proposed in \cite{rockafellar1976monotone}. It finds the optimal point of this problem by iteratively updating the solution as:
\begin{align}
\b{x}^{(k+1)} &:= \textbf{prox}_{\lambda f}(\b{x}^{(k)}) \nonumber\\
&\overset{(\ref{equation_proximal_mapping_scaled})}{=} \arg\min_{\b{u}} \Big( f(\b{u}) + \frac{1}{2\lambda} \|\b{u} - \b{x}^{(k)}\|_2^2 \Big),
\end{align}
until convergence, where $\lambda$ can be seen as a parameter related to the step size. 
In other words, proximal gradient method applies gradient descent on the Moreau envelope $M_{\lambda f}(\b{x})$ (recall Eq. (\ref{equation_Moreau_envelope})) rather than on the function $f(.)$ itself.

\subsubsection{Proximal Gradient Method}




\textbf{-- Composite Problems:}
Consider the following optimization problem:
\begin{align}\label{equation_composite_optimization}
\underset{\b{x}}{\text{minimize}}\quad f(\b{x}) + g(\b{x}),
\end{align}
where $f(\b{x})$ is a smooth function and $g(\b{x})$ is a convex function which is not smooth necessarily. According to the following definition, this is a composite optimization problem. 

\begin{definition}[Composite objective function \cite{nesterov2013gradient}]
In optimization, if a function is stated as a summation of two terms, $f(\b{x}) + g(\b{x})$, it is called a composite function and its optimization is a composite optimization problem. 
\end{definition}

Composite problems are widely used in machine learning and regularized problems because $f(\b{x})$ can be the cost function to be minimized while $g(\b{x})$ is the penalty or regularization term \cite{ghojogh2019theory}. 

\hfill\break
\textbf{-- Proximal Gradient Method for Composite Optimization:}

For solving problem (\ref{equation_composite_optimization}), we can approximate the function $f(.)$ by its quadratic approximation around point $\b{x}$ because it is smooth (differentiable):
\begin{align*}
f(\b{u}) \approx f(\b{x}) + \nabla f(\b{x})^\top (\b{u} - \b{x}) + \frac{1}{2 \eta} \|\b{u} - \b{x}\|_2^2,
\end{align*}
where we have replaced $\nabla^2 f(\b{x})$ with scaled identity matrix, $(1/\eta) \b{I}$. Hence, the solution of problem (\ref{equation_composite_optimization}) can be approximated as:
\begin{align}
&\b{x} = \arg \min_{\b{u}} \big(f(\b{u}) + g(\b{u})\big) \nonumber\\
&\approx \arg \min_{\b{u}} \big(f(\b{x}) + \nabla f(\b{x})^\top (\b{u} - \b{x}) + \frac{1}{2 \eta} \|\b{u} - \b{x}\|_2^2 \nonumber\\
&+ g(\b{u})\big) = \arg \min_{\b{u}} \big(\frac{1}{2\eta} \|\b{u} - (\b{x} - \eta \nabla f(\b{x}))\|_2^2 + g(\b{u})\big). \label{equation_proximal_gradient_method_proof}
\end{align}
The first term in Eq. (\ref{equation_proximal_gradient_method_proof}) keeps the solution close to the solution of gradient descent for minimizing the function $f(.)$ (see Eq. (\ref{equation_GD_step_by_eta})) and the second term in Eq. (\ref{equation_proximal_gradient_method_proof}) makes the function $g(.)$ small. 

\textit{Proximal gradient method}, also called \textit{proximal gradient descent}, uses Eq. (\ref{equation_proximal_gradient_method_proof}) for solving the composite problem (\ref{equation_composite_optimization}). 
It was first proposed in \cite{nesterov2013gradient} and also in \cite{beck2009fast} for $g = \|.\|_1$.
It finds the optimal point by iteratively updating the solution as:
\begin{align}
&\b{x}^{(k+1)}\! \overset{(\ref{equation_proximal_gradient_method_proof})}{:=}\! \nonumber\\
&\arg \min_{\b{u}} \big(\frac{1}{2\eta^{(k)}} \|\b{u} - (\b{x}^{(k)} - \eta^{(k)} \nabla f(\b{x}^{(k)}))\|_2^2 + g(\b{u})\big) \nonumber\\
&\overset{(\ref{equation_proximal_mapping_scaled})}{=} \textbf{prox}_{\eta^{(k)} g}\big(\b{x}^{(k)} - \eta^{(k)} \nabla f(\b{x}^{(k)})\big), \label{equation_proximal_gradient_method_update}
\end{align}
until convergence, where $\eta^{(k)}$ is the step size which can be fixed or found by line-search. 
In Eq. (\ref{equation_composite_optimization}), the function $g(.)$ can be a regularization term such as $\ell_2$ or $\ell_1$ norm. In these cases, we use Lemmas \ref{lemma_prox_l2_norm} and \ref{lemma_prox_l1_norm} for calculating Eq. (\ref{equation_proximal_gradient_method_update}). 
The convergence rates of proximal gradient method is discussed in \cite{schmidt2011convergence}.
A distributed version of this method is also proposed in \cite{chen2012fast}.

\subsection{Gradient Methods for Constrained Problems}

\subsubsection{Projected Gradient Method}\label{section_projected_gradient_method}




\textit{Projected gradient method} \cite{iusem2003convergence}, also called \textit{gradient projection method} and \textit{projected gradient descent}, considers $g(\b{x})$ to be the indicator function $\mathbb{I}(\b{x} \in \mathcal{S})$ in problem (\ref{equation_composite_optimization}). In other words, the optimization problem is a constrained problem as problem (\ref{equation_constraint_set_optimization_problem}), which can be restated to:
\begin{equation}\label{equation_constraint_set_optimization_problem_withIndicator}
\begin{aligned}
& \underset{\b{x}}{\text{minimize}}
& & f(\b{x}) + \mathbb{I}(\b{x} \in \mathcal{S}),
\end{aligned}
\end{equation}
because the indicator function becomes infinity if its condition is not satisfied. 
According to Eq. (\ref{equation_proximal_gradient_method_update}), the solution is updated as:
\begin{align}
\b{x}^{(k+1)} & \overset{(\ref{equation_proximal_gradient_method_update})}{:=} \textbf{prox}_{\eta^{(k)} \mathbb{I}(. \in \mathcal{S})}\big(\b{x}^{(k)} - \eta^{(k)} \nabla f(\b{x}^{(k)})\big) \nonumber\\
&\overset{(\ref{equation_projection_onto_set})}{=} \Pi_{\mathcal{S}}\big(\b{x}^{(k)} - \eta^{(k)} \nabla f(\b{x}^{(k)})\big). \label{equation_projected_gradient_method_update}
\end{align}
In other words, projected gradient method performs a step of gradient descent and then projects the solution onto the set of constraint. This procedure is repeated until convergence of solution. 

\begin{lemma}[Projection onto the cone of orthogonal matrices {\citep[Section 6.7.2]{parikh2014proximal}}]\label{lemma_projection_onto_orthogonal_cone}
A function $g: \mathbb{R}^{d_1 \times d_2} \rightarrow \mathbb{R}$ is orthogonally invariant if $g(\b{U}\b{X}\b{V}^\top) = g(\b{X})$, for all $\b{U} \in \mathbb{R}^{d_1 \times d_1}$, $\b{X} \in \mathbb{R}^{d_1 \times d_2}$, and $\b{V} \in \mathbb{R}^{d_2 \times d_2}$ where $\b{U}$ and $\b{V}$ are orthogonal matrices. 

Let $g$ be a convex and orthogonally invariant function, and it works on the singular values of a matrix variable $\b{X} \in \mathbb{R}^{d_1 \times d_2}$, i.e., $g = \widehat{g} \circ \sigma$ where the function $\sigma(\b{X})$ gives the vector of singular values of $\b{X}$. In this case, we have:
\begin{align}\label{equation_prox_projection_orthogonal_matrices}
\textbf{prox}_{\lambda, g}(\b{X}) = \b{U}\,\, \textbf{diag}\Big(\textbf{prox}_{\lambda, g}\big(\sigma(\b{X})\big)\Big)\,\, \b{V}^\top,
\end{align}
where $\textbf{diag}(.)$ makes a diagonal matrix with its input as the diagonal, and $\b{U} \in \mathbb{R}^{d_1 \times d_1}$ and $\b{V} \in \mathbb{R}^{d_2 \times d_2}$ are the matrices of left and right singular vectors of $\b{X}$, respectively.

Consider the constraint for projection onto the cone of orthogonal matrices, i.e., $\b{X}^\top \b{X} = \b{I}$.
In this constraint, the function $g$ deals with the singular values of $\b{X}$. The reason is that, from the Singular Value Decomposition (SVD) of $\b{X}$, we have: $\b{X} \overset{\text{SVD}}{=} \b{U} \b{\Sigma} \b{V}^\top \implies \b{X}^\top \b{X} = \b{U} \b{\Sigma} \b{V}^\top \b{V} \b{\Sigma} \b{U}^\top \overset{(a)}{=} \b{U} \b{\Sigma}^2 \b{U}^\top \overset{\text{set}}{=} \b{I} \implies \b{U} \b{\Sigma}^2 \b{U}^\top \b{U} = \b{U} \overset{(b)}{\implies} \b{U} \b{\Sigma}^2 = \b{U} \implies \b{\Sigma} = \b{I}$, where $(a)$ and $(b)$ are because $\b{U}$ and $\b{V}$ are orthogonal matrices. 
Therefore, the constraint $\b{X}^\top \b{X} = \b{I}$ (i.e., projecting onto the cone of orthogonal matrices) can be modeled by Eq. (\ref{equation_prox_projection_orthogonal_matrices}) which is simplified to setting all singular values of $\b{X}$ to one:
\begin{align}\label{equation_prox_projection_orthogonal_matrices2}
\textbf{prox}_{\lambda, g}(\b{X}) = \Pi_{\mathcal{O}} = \b{U} \b{I} \b{V}^\top,
\end{align}
where $\b{I} \in \mathbb{R}^{d_1 \times d_2}$ is a rectangular identity matrix and $\mathcal{O}$ denotes the cone of orthogonal matrices.
If the constraint is scaled orthogonality, i.e. $\b{X}^\top \b{X} = \lambda \b{I}$ with $\lambda$ as the scale, the projection is setting all singular values to $\lambda$ by $\b{U} (\lambda \b{I}) \b{V}^\top = \lambda \b{U} \b{I} \b{V}^\top$.
\end{lemma}

Although most often projected gradient method is used for Eq. (\ref{equation_projected_gradient_method_update}), there are few other variants of projected gradient methods such as \cite{drummond2004projected}: 
\begin{align}
& \b{y}^{(k)} := \Pi_{\b{S}} \big(\b{x}^{(k)} - \eta^{(k)} \nabla f(\b{x}^{(k)})\big), \\
& \b{x}^{(k+1)} := \b{x}^{(k)} + \gamma^{(k)} (\b{y}^{(k)} - \b{x}^{(k)}),
\end{align}
where $\eta^{(k)}$ and $\gamma^{(k)}$ are positive step sizes at iteration $k$. In this alternating approach, we find an additional variable $\b{y}$ by gradient descent followed by projection and then update $\b{x}$ to get close to the found $\b{y}$ while staying close to the previous solution by line-search.

\subsubsection{Projection Onto Convex Sets (POCS) and Averaged Projections}


Assume we want to project a point onto the intersection of $c$ closed convex sets, i.e., $\bigcap_{j=1}^c \mathcal{S}_j$. 
We can model this by an optimization problem with a fake objective function:
\begin{equation}\label{equation_optimization_problem_projection_onto_sets}
\begin{aligned}
& \underset{\b{x}}{\text{minimize}}
& & \b{x} \in \mathbb{R}^d \\
& \text{subject to}
& & \b{x} \in \mathcal{S}_1, \dots, \b{x} \in \mathcal{S}_c.
\end{aligned}
\end{equation}
\textit{Projection Onto Convex Sets (POCS)} solves this problem, similar to projected gradient method, by projecting onto the sets one-by-one \cite{bauschke1996projection}:
\begin{align}
& \b{x}^{(k+1)} := \Pi_{\b{S}_1}(\Pi_{\b{S}_2}(\dots \Pi_{\b{S}_c}(\b{x}^{(k)})\dots)),
\end{align}
and repeating it until convergence.
Another similar method for solving problem (\ref{equation_optimization_problem_projection_onto_sets}) is the \textit{averaged projections} which updates the solution as:
\begin{align}
& \b{x}^{(k+1)} :=
\frac{1}{c}
\Big(\Pi_{\b{S}_1}(\b{x}^{(k)}) + \dots + \Pi_{\b{S}_c}(\b{x}^{(k)})\Big).
\end{align}

\subsubsection{Frank-Wolfe Method}




\textit{Frank-Wolfe method}, also called \textit{conditional gradient method} and \textit{reduced gradient algorithm}, was first proposed in \cite{frank1956algorithm} and it can be used for solving the constrained problem (\ref{equation_constraint_set_optimization_problem}) using gradient of objective function \cite{levitin1966constrained}. 
It updates the solution as:
\begin{align}
& \b{y}^{(k)} := \arg \min_{\b{y} \in \mathcal{S}} \nabla f(\b{x}^{(k)})^\top \b{y}, \label{equation_Frank_Wolfe_update_y} \\
& \b{x}^{(k+1)} := (1 - \gamma^{(k)})\, \b{x}^{(k)} + \gamma^{(k)} \b{y}^{(k)}, \label{equation_Frank_Wolfe_update_x}
\end{align}
until convergence, where $\gamma^{(k)}$ is the step size at iteration $k$. 
Eq. (\ref{equation_Frank_Wolfe_update_y}) finds the direction to move toward at the iteration and Eq. (\ref{equation_Frank_Wolfe_update_x}) updates the solution while staying close to the previous solution by line-search. 
A stochastic version of Frank-Wolfe method is proposed in \cite{reddi2016stochastic}. 

\section{Non-smooth and L1 Norm Optimization Methods}

\subsection{Lasso Regularization}\label{section_lasso_regularization}

The $\ell_1$ norm can be used for sparsity \cite{ghojogh2019theory}. We explain the reason in the following. 
Sparsity is very useful and effective because of betting on sparsity principal \cite{friedman2001elements} and the  Occam's razor \cite{domingos1999role}. If $\b{x} = [x_1, \dots, x_d]^\top$, for having sparsity, we should use \textit{subset selection} for the regularization of a cost function $\Omega_0(\b{x})$:
\begin{align}\label{equation_regularization_optimization_l0}
\underset{\b{x}}{\text{minimize}} ~~~ \Omega(\b{x}) := \Omega_0(\b{x}) + \lambda\, ||\b{x}||_0,
\end{align}
where:
\begin{align}
||\b{x}||_0 := \sum_{j=1}^d \mathbb{I}(x_j \neq 0) = 
\left\{
    \begin{array}{ll}
      0 & \text{if } x_j = 0, \\
      1 & \text{if } x_j \neq 0,
    \end{array}
\right.
\end{align}
is the ``$\ell_0$'' norm, which is not a norm (so we use ``.'' for it) because it does not satisfy the norm properties \cite{boyd2004convex}. The ``$\ell_0$'' norm counts the number of non-zero elements so when we penalize it, it means that we want to have sparser solutions with many zero entries.
According to \cite{donoho2006most}, the convex relaxation of ``$\ell_0$'' norm (subset selection) is $\ell_1$ norm. Therefore, we write the regularized optimization as:
\begin{align}\label{equation_regularization_optimization_l1}
\underset{\b{x}}{\text{minimize}} ~~~ \Omega(\b{x}) := \Omega_0(\b{x}) + \lambda\, ||\b{x}||_1.
\end{align}
Note that the $\ell_1$ regularization is also referred to as \textit{lasso} (least absolute shrinkage and selection operator) regularization \cite{tibshirani1996regression,hastie2019statistical}.
Different methods exist for solving optimization having $\ell_1$ norm, such as its approximation by Huber function \cite{huber1992robust}, proximal algorithm and soft thresholding \cite{parikh2014proximal}, coordinate descent \cite{wright2015coordinate,wu2008coordinate}, and subgradients. 
In the following, we explain these methods.

\subsection{Convex Conjugate}


\subsubsection{Convex Conjugate}

Consider Fig. \ref{figure_convex_conjugate} showing a line which supports the function $f$ meaning that it is tangent to the function and the function upper-bounds it. In other words, if the line goes above where it is, it will intersect the function in more than a point. 
Now let the support line be multi-dimensional to be a support hyperplane. 
For having this tangent support hyperplane with slope $\b{y} \in \mathbb{R}^d$ and intercept $\beta \in \mathbb{R}$, we should have:
\begin{align*}
& \b{y}^\top \b{x} + \beta = f(\b{x}) \implies \beta = f(\b{x}) - \b{y}^\top \b{x}.
\end{align*}
We want the smallest intercept for the support hyperplane:
\begin{align*}
& \beta^* = \min_{\b{x} \in \mathbb{R}^d} \big( f(\b{x}) - \b{y}^\top \b{x} \big) \overset{(\ref{equation_max_min_conversion})}{=} - \max_{\b{x} \in \mathbb{R}^d} \big( \b{y}^\top \b{x} - f(\b{x}) \big).
\end{align*}
We define $f^*(\b{y}) := -\beta^*$ to have the convex conjugate defined as below.

\begin{definition}[Convex conjugate of function]
The conjugate gradient of function $f(.)$ is defined as:
\begin{align}\label{equation_convex_conjugate}
f^*(\b{y}) := \sup_{\b{x} \in \mathbb{R}^d} \big( \b{y}^\top \b{x} - f(\b{x}) \big). 
\end{align}
\end{definition}

\begin{figure}[!t]
\centering
\includegraphics[width=2.6in]{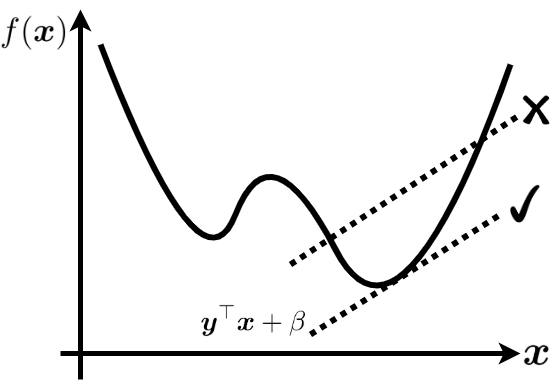}
\caption{Supporting line (or hyper-plane) to the function.}
\label{figure_convex_conjugate}
\end{figure}

The convex conjugate of a function is always convex, even if the function itself is not convex, because it is point-wise maximum of affine functions. 

\begin{lemma}[Conjugate of convex conjugate]
the conjugate of convex conjugate of a function is:
\begin{align}
f^{**}(\b{x}) = \sup_{\b{y} \in \textbf{dom}(f^*)} \big( \b{x}^\top \b{y} - f^*(\b{y}) \big). 
\end{align}
It is always a lower-bound for the function, i.e., $f^{**}(\b{x}) \leq f(\b{x})$. If the function $f(.)$ is convex, we have $f^{**}(\b{x}) = f(\b{x})$; hence, for a convex function, we have:
\begin{align}\label{equation_f_equals_conjigate_of_convex_conjugate}
f(\b{x}) = \sup_{\b{y} \in \textbf{dom}(f^*)} \big( \b{x}^\top \b{y} - f^*(\b{y}) \big). 
\end{align}
\end{lemma}

\begin{lemma}[Gradient in terms of convex conjugate]
For any function $f(.)$, we have:
\begin{align}\label{equation_gradient_intermsof_convex_conjugate}
\nabla f(\b{x}) = \arg\max_{\b{y} \in \textbf{dom}(f^*)} \big( \b{x}^\top \b{y} - f^*(\b{y}) \big).
\end{align}
\end{lemma}

\subsubsection{Huber Function: Smoothing L1 Norm by Convex Conjugate}




\begin{lemma}[The convex conjugate of $\ell_1$ norm]
The convex conjugate of $f(.) = \|.\|_1$ is:
\begin{align}\label{equation_convex_conjugate_l1_norm}
f^*(\b{y}) = 
\left\{
    \begin{array}{ll}
        0 & \mbox{if } \|\b{y}\|_{\infty} \leq 1 \\
        \infty & \mbox{Otherwise.}
    \end{array}
\right.
\end{align}
\end{lemma}
\begin{proof}
We can write $\ell_1$ norm as:
\begin{align*}
f(\b{x}) = \|\b{x}\|_1 = \max_{\|\b{z}\|_\infty \leq 1} \b{x}^\top \b{z}.
\end{align*}
Using this in Eq. (\ref{equation_convex_conjugate}) results in Eq. (\ref{equation_convex_conjugate_l1_norm}). Q.E.D.
\end{proof}

According to Eq. (\ref{equation_gradient_intermsof_convex_conjugate}), we have $\nabla f(\b{x}) = \arg\max_{\|\b{y}\|_\infty \leq 1} \b{x}^\top \b{y}$. 
For $\b{x} = \b{0}$, we have $\nabla f(\b{x}) = \arg\max_{\|\b{y}\|_\infty \leq 1} \b{0}$ which has many solutions. Therefore, at $\b{x} = 0$, the function $\|.\|_1$ norm is not differentiable and not smooth because the gradient at that point is not unique. 

We can smooth the $\ell_1$ norm at $\b{x}=\b{0}$ using convex conjugate. Let $\b{x} = [x_1, \dots, x_d]^\top$. As we have $f(\b{x}) = \|\b{x}\|_1 = \sum_{j=1}^d |x_j|$, we can use the convex conjugate for every dimension $f(x_j) = |x_j|$:
\begin{align}\label{equation_convex_conjugate_of_absolute_value}
f^*(y_j) = 
\left\{
    \begin{array}{ll}
        0 & \mbox{if } |y_j| \leq 1 \\
        \infty & \mbox{Otherwise.}
    \end{array}
\right.
\end{align}
According to Eq. (\ref{equation_f_equals_conjigate_of_convex_conjugate}), we have:
\begin{align*}
|x_j| = \sup_{y \in \mathbb{R}} \big( x_j y_j - f^*(y_j) \big) \overset{(\ref{equation_convex_conjugate_of_absolute_value})}{=} \max_{|y_j| \leq 1} x_j y_j.
\end{align*}
This is not unique for $x_j = 0$. Hence, we add a $\mu$-strongly convex function to the above equation to make the solution unique at $x_j=0$ also.
This added term is named the \textit{proximity function} defined below.

%
\begin{definition}[Proximity function \cite{banaschewski1961proximity}]
A proximity function $p(\b{y})$ for a closed convex set $\mathcal{S} \in \textbf{dom}(p)$ is a function which is continuous and strongly convex.
We can change Eq. (\ref{equation_f_equals_conjigate_of_convex_conjugate}) to:
\begin{align}\label{equation_f_equals_conjigate_of_convex_conjugate_withProximityFunction}
f(\b{x}) \approx f_{\mu}(\b{x}) := \sup_{\b{y} \in \textbf{dom}(f^*)} \big( \b{x}^\top \b{y} - f^*(\b{y}) - \mu\, p(\b{y}) \big),
\end{align}
where $\mu > 0$.
\end{definition}
Using Eq. (\ref{equation_f_equals_conjigate_of_convex_conjugate_withProximityFunction}), we can have:
\begin{align*}
&|x_j| \approx \sup_{y \in \mathbb{R}} \big( x_j y_j - f^*(y_j) - \frac{\mu}{2} y_j^2 \big) \\
&\overset{(\ref{equation_convex_conjugate_of_absolute_value})}{=} \max_{|y_j| \leq 1} (x_j y_j - \frac{\mu}{2} y_j^2) = 
\left\{
    \begin{array}{ll}
        \frac{x_j^2}{2 \mu} & \mbox{if } |x_j| \leq \mu \\
        |x_j| - \frac{\mu}{2} & \mbox{if } |x_j| > \mu.
    \end{array}
\right.
\end{align*}
This approximation to $\ell_1$ norm, which is differentiable everywhere, including at $x_j = 0$, is named the Huber function defined below. 
Note that the Huber function is the  Moreau envelope of absolute value (see Definition \ref{definition_Moreau_envelope}).

\begin{figure*}[!t]
\centering
\includegraphics[width=7in]{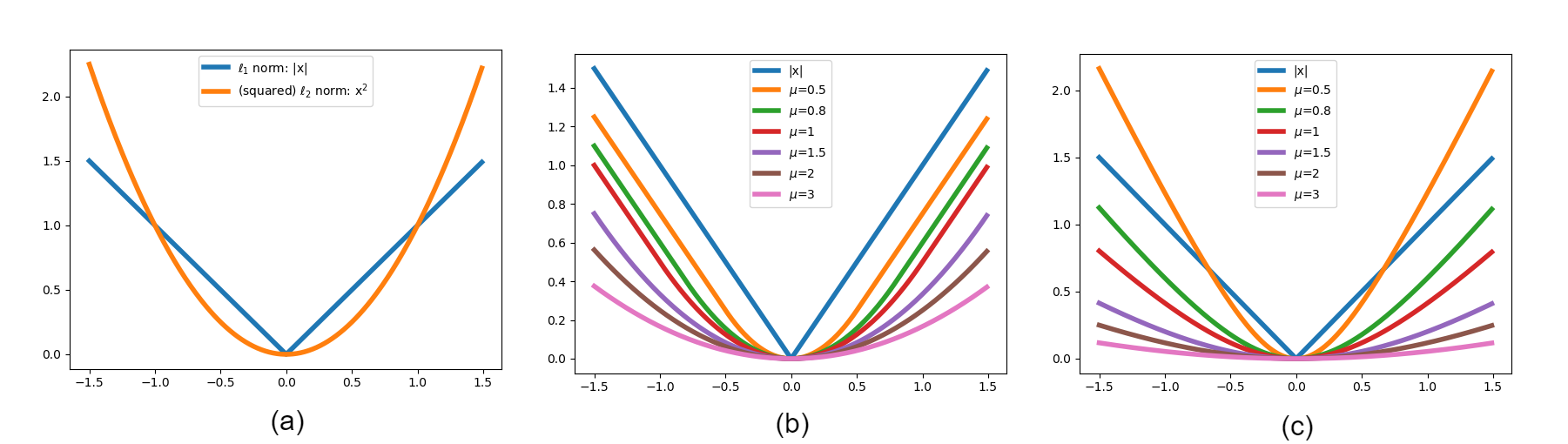}
\caption{(a) Comparison of $\ell_1$ and $\ell_2$ norms in $\mathbb{R}^1$, (b) comparison of $\ell_1$ norm (i.e., absolute value in $\mathbb{R}^1$) and the Huber function, and (c) comparison of $\ell_1$ norm (i.e., absolute value in $\mathbb{R}^1$) and the pseudo-Huber function.}
\label{figure_Huber}
\end{figure*}

\begin{definition}[Huber and pseudo-Huber functions \cite{huber1992robust}]
The Huber function and pseudo-Huber functions are:
\begin{align}
& h_\mu(x) =
\left\{
    \begin{array}{ll}
        \frac{x^2}{2 \mu} & \mbox{if } |x| \leq \mu \\
        |x| - \frac{\mu}{2} & \mbox{if } |x| > \mu,
    \end{array}
\right. 
\label{equation_Huber_function}\\
& \widehat{h}_{\mu}(x) = \sqrt{\Big(\frac{x}{\mu}\Big)^2 + 1} - 1, \label{equation_pseudo_Huber_function}
\end{align}
respectively, where $\mu > 0$.
The derivative of these functions is easily calculated. For example, the derivative of Huber function is:
\begin{align*}
\nabla h_\mu(x) =
\left\{
    \begin{array}{ll}
        \frac{x}{\mu} & \mbox{if } |x| \leq \mu \\
        \textbf{sign}(x) & \mbox{if } |x| > \mu.
    \end{array}
\right. 
\end{align*}
\end{definition}

The Huber and pseudo-Huber functions are shown for different $\mu$ values in Fig. \ref{figure_Huber}.
As this figure shows, in contrast to $\ell_1$ norm or absolute value, these two functions are smooth so they approximate the $\ell_1$ norm smoothly. 
This figure also shows that the Huber function is always upper-bounded by absolute value ($\ell_1$ norm); however, this does not hold for pseudo-Huber function. 
We can also see that the approximation of Huber function is better than the approximation of pseudo-Huber function; however, its calculation is harder than pseudo-Huber function because it is a piece-wise function (compare Eqs. (\ref{equation_Huber_function}) and (\ref{equation_pseudo_Huber_function})). 
Moreover, the figure shows a smaller positive value $\mu$ give better approximations, although it makes calculation of the Huber and pseudo-Huber functions harder. 

\subsection{Soft-thresholding and Proximal Methods}


Proximal mapping was introduced in Section \ref{section_proximal_methods}. We can use proximal mapping of non-smooth functions to solve non-smooth optimization by proximal methods introduced in Section \ref{section_proximal_methods}.
For example, we can solve an optimization problem containing $\ell_1$ norm in its objective function using Eq. (\ref{equation_proximal_mapping_scaled_l1_norm}). That equation, named soft-thresholding, is the proximal mapping of $\ell_1$ norm. Then, we can use any of the proximal methods such as proximal point method and proximal gradient method. 
For solving the regularized problem (\ref{equation_regularization_optimization_l1}), which is optimizing a composite function, we can use the proximal gradient method introduced in Section \ref{section_proximal_methods}. 

\subsection{Coordinate Descent}

\subsubsection{Coordinate Method}




Assume $\b{x} = [x_1, \dots, x_d]^\top$.
For solving Eq. (\ref{equation_optimization_problem_unconstrained}), \textit{coordinate method} \cite{wright2015coordinate} updates the dimensions (coordinates) of solution one-by-one and not all dimensions together at once:
\begin{equation}\label{equation_coordinate_descent}
\begin{aligned}
& x_1^{(k+1)} := \arg\min_{x_1} f(x_1, x_2^{(k)}, x_3^{(k)}, \dots, x_d^{(k)}), \\
& x_2^{(k+1)} := \arg\min_{x_2} f(x_1^{(k+1)}, x_2, x_2^{(k)}, \dots, x_d^{(k)}), \\
&\vdots \\
& x_d^{(k+1)} := \arg\min_{x_d} f(x_1^{(k+1)}, x_2^{(k+1)}, x_3^{(k+1)}, \dots, x_d),
\end{aligned}
\end{equation}
until convergence of all dimensions of solution. 
Note that the update of every dimension uses the latest update of previously updated dimensions. 
The order of updates for the dimensions does not matter. 
The idea of coordinate descent algorithm is similar to the idea of Gibbs sampling \cite{geman1984stochastic,ghojogh2020sampling} where we work on the dimensions of the variable one by one.

If we use a step of gradient descent (i.e. Eq. (\ref{equation_GD_step_by_eta})) for every of the above updates, the method is named \textit{coordinate descent}.
If we use proximal gradient method (i.e., Eq. (\ref{equation_proximal_gradient_method_update})) for every update in coordinate method, the method is named the \textit{proximal coordinate descent}. 
Note that we can group some of the dimensions (features) together and alternate between updating the blocks (groups) of features. That method is named \textit{block coordinate descent}.
The convergence analysis of coordinate descent and block coordinate descent methods can be found in \cite{luo1992convergence,luo1993error} and \cite{tseng2001convergence}, respectively. 
They show that if the function $f(.)$ is continuous, proper, and closed, the coordinate descent method converges to a stationary point. 
There exist some other faster variants of coordinate descent named \textit{accelerated coordinate descent} \cite{lee2013efficient,fercoq2015accelerated}.

similar to SGD, the full gradient is not available in coordinate descent to use for checking convergence, as discussed in Section \ref{section_convergence_criterion}. One can use other criteria in that section. 
Moreover, note that SGD can be used with the line-search methods, too. 
Although coordinate descent methods are very simple and shown to work properly for $\ell_1$ norm optimization \cite{wu2008coordinate}, they have not sufficiently attracted the attention of researchers in the field of optimization \cite{wright2015coordinate}.

\subsubsection{L1 Norm Optimization}


Coordinate descent method can be used for $\ell_1$ norm (lasso) optimization \cite{wu2008coordinate} because every coordinate of the $\ell_1$ norm is an absolute value ($\|\b{x}\|_1 = \sum_{j=1}^d |x_j|$ for $\b{x} = [x_1, \dots, x_d]^\top$) and the derivative of absolute value a simple sign function (note that we have subgradient for absolute value at zero, which will be introduced in Section \ref{section_subgradient}).
One of the well-known $\ell_1$ optimization methods is the lasso regression \cite{tibshirani1996regression,friedman2001elements,hastie2019statistical}:
\begin{align}\label{equation_lasso_regression}
\underset{\b{\beta}}{\text{minimize}} ~~~ \frac{1}{2} \|\b{y} - \b{X} \b{\beta}\|_2^2 + \lambda \|\b{\beta}\|_1,
\end{align}
where $\b{y} \in \mathbb{R}^n$ are the labels, $\b{X} = [\b{x}_1, \dots, \b{x}_d] \in \mathbb{R}^{n \times d}$ are the observations, $\b{\beta} = [\beta_1, \dots, \beta_d]^\top \in \mathbb{R}^d$ are the regression coefficients, and $\lambda$ is the regularization parameter. 
The lasso regression is sparse which is effective because of the reasons explained in Section \ref{section_lasso_regularization}. 

Let $c$ denote the objective function in Eq. (\ref{equation_lasso_regression}). 
The objective function can be simplified as $0.5(\b{y}^\top \b{y} - 2 \b{\beta}^\top \b{X}^\top \b{y} + \b{\beta}^\top \b{X}^\top \b{X} \b{\beta}) + \lambda \|\b{\beta}\|_1$. 
We can write the $j$-th element of this objective, denoted by $c_j$, as:
\begin{align*}
& c_j = \frac{1}{2} (\b{y}^\top \b{y} - 2 \b{x}_j^\top \b{y} \beta_j \\
&~~~~~~~~~~~~~~~~~~~ + \beta_j \b{x}_j^\top \b{x}_j \beta_j + \beta_j \b{x}_j^\top \b{X}_{-j} \b{\beta}_{-j}) + \lambda |\beta_j|,
\end{align*}
where $\b{X}_{-j} := [\b{x}_1, \dots, \b{x}_{j-1}, \b{x}_{j+1}, \dots, \b{x}_d]$ and $\b{\beta}_{-j} := [\beta_1, \dots, \beta_{j-1}, \beta_{j+1}, \dots, \beta_d]^\top$. 
For coordinate descent, we need gradient of objective function w.r.t. every coordinate. 
The derivatives of other coordinates of objective w.r.t. $\beta_j$ are zero so we need $c_j$ for derivative w.r.t. $\beta_j$. 
Taking derivative of $c_j$ w.r.t. $\beta_j$ and setting it to zero gives:
\begin{align*}
&\frac{\partial c}{\partial \beta_j} = \frac{\partial c_j}{\partial \beta_j} \\
&= \b{x}_j^\top \b{x}_j \beta_j + \b{x}_j^\top (\b{X}_{-j} \b{\beta}_{-j} - \b{y}) + \lambda\, \textbf{sign}(\beta_j) \overset{\text{set}}{=} 0 \\
&\implies 
\beta_j = s_{\frac{\lambda}{\|\b{x}_j\|_2^2}}\Big(\frac{\b{x}_j^\top (\b{y} - \b{X}_{-i} \b{\beta}_{-i})}{\b{x}_j^\top \b{x}_j}\Big) \\
&=
\left\{
    \begin{array}{ll}
        \frac{\b{x}_j^\top (\b{y} - \b{X}_{-i} \b{\beta}_{-i})}{\|\b{x}_j\|_2^2} - \frac{\lambda}{\|\b{x}_j\|_2^2} & \mbox{if } \frac{\b{x}_j^\top (\b{y} - \b{X}_{-i} \b{\beta}_{-i})}{\b{x}_j^\top \b{x}_j} \geq \frac{\lambda}{\|\b{x}_j\|_2^2} \\
        0 & \mbox{if } |\frac{\b{x}_j^\top (\b{y} - \b{X}_{-i} \b{\beta}_{-i})}{\b{x}_j^\top \b{x}_j}|\! <\! \frac{\lambda}{\|\b{x}_j\|_2^2} \\
        \frac{\b{x}_j^\top (\b{y} - \b{X}_{-i} \b{\beta}_{-i})}{\|\b{x}_j\|_2^2} + \frac{\lambda}{\|\b{x}_j\|_2^2} & \mbox{if } \frac{\b{x}_j^\top (\b{y} - \b{X}_{-i} \b{\beta}_{-i})}{\b{x}_j^\top \b{x}_j}\! \leq\! -\frac{\lambda}{\|\b{x}_j\|_2^2},
    \end{array}
\right.
\end{align*}
which is a soft-thresholding function (see Eq. (\ref{equation_proximal_mapping_scaled_l1_norm})). Therefore, coordinate descent for $\ell_1$ optimization finds the soft-thresholding solution, the same as the proximal mapping. 
We can use this soft-thresholding in coordinate descent where we use $\beta_j$'s in Eq. (\ref{equation_coordinate_descent}) rather than $x_j$'s.

\subsection{Subgradient Methods}



\subsubsection{Subgradient}\label{section_subgradient}

We know that the convex conjugate $f^*(\b{y})$ is always convex. 
If the convex conjugate $f^*(\b{y})$ is strongly convex, then we have only one gradient according to Eq. (\ref{equation_gradient_intermsof_convex_conjugate}). However, if the convex conjugate is only convex and not strongly convex, Eq. (\ref{equation_gradient_intermsof_convex_conjugate}) might have several solutions so the gradient may not be unique. For the points in which the function does not have a unique gradient, we can have a set of subgradients, defined below. 

\begin{definition}[Subgradient]
Consider a convex function $f(.)$ with domain $\mathcal{D}$.
The vector $\b{g} \in \mathbb{R}^d$ is a subgradient of $f(.)$ at $\b{x} \in \mathcal{D}$ if it satisfies:
\begin{align}\label{equation_subgradient}
f(\b{y}) \geq f(\b{x}) + \b{g}^\top (\b{y} - \b{x}), \quad \forall \b{y} \in \mathcal{D}.
\end{align}
\end{definition}
As Fig. \ref{figure_subgradient} shows, if the function is not smooth at a point, it has multiple subgradients at that point. This is while there is only one subgradient (which is the gradient) for a point at which the function is smooth. 

\begin{definition}[subdifferential]
The subdifferential of a convex function $f(.)$, with domain $\mathcal{D}$, at a point $\b{x} \in \mathcal{D}$ is the set of all subgradients at that point:
\begin{align}
\partial f(\b{x}) := \{\b{g}\, |\, \b{g}^\top (\b{y} - \b{x}) \overset{(\ref{equation_subgradient})}{\leq} f(\b{y}) - f(\b{x}),\, \forall \b{y} \in \mathcal{D}\}.
\end{align}
The subdifferential is a closed convex set. Every subgradient is a member of the subdifferential, i.e., $\b{g} \in \partial f(\b{x})$. An example subdifferential is shown in Fig. \ref{figure_subgradient}.
\end{definition}

\begin{figure}[!t]
\centering
\includegraphics[width=3.2in]{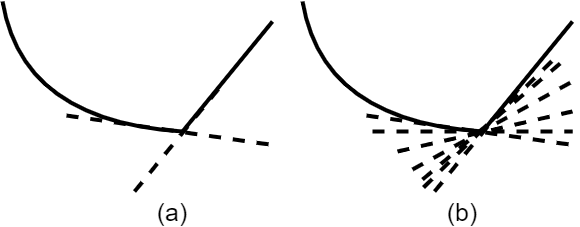}
\caption{Subgradients: (a) the two extreme subgradients at the non-smooth point and (b) some examples of the subgradients in the subdifferential at the non-smooth point.}
\label{figure_subgradient}
\end{figure}

An example of subgradient is the subdifferential of absolute value, $f(.) = |.|$:
\begin{align}\label{equation_subgradient_of_absolute_value}
\partial f(\b{x}) =
\left\{
    \begin{array}{ll}
        1 & \mbox{if } x > 0 \\
        \in [-1,1] & \mbox{if } x = 0 \\
        -1 & \mbox{if } x < 0. 
    \end{array}
\right. 
\end{align}
The subgradient of absolute value is equal to the gradient for $x<0$ and $x>0$ while there exists a set of subgradients at $x=0$ because absolute value is not smooth at that point. 
We can also compute the subgradient of $\ell_1$ norm because we have $f(\b{x}) = \|\b{x}\|_1 = \sum_{i=1}^d |x_i| = \sum_{i=1}^d f_i(\b{x}_i)$ for $\b{x} = [x_1, \dots, x_d]^\top$.
We take Eq. (\ref{equation_subgradient_of_absolute_value}) as the subdifferential of the $i$-th dimension, denoted by $\partial f_i(x_i)$. Hence, for $f(\b{x}) = \|\b{x}\|_1$, we have $\partial f(\b{x}) = \partial f_1(x_1) \times \dots \times \partial f_d(x_d)$ where $\times$ denotes the Cartesian product of sets.

We can have the first-order optimality condition using subgradients by generalizing Lemma \ref{lemma_first_order_optimality_condition} as follows.
\begin{lemma}[First-order optimality condition with subgradient]\label{lemma_first_order_optimality_condition_subgradient}
If $\b{x}^*$ is a local minimizer for a function $f(.)$, then:
\begin{align}\label{equation_first_order_optimality_condition_subgradient}
\b{0} \in \partial f(\b{x}^*).
\end{align}
Note that if $f(.)$ is convex, this equation is a necessary and sufficient condition for a minimizer. 
\end{lemma}
\begin{proof}
According to Eq. (\ref{equation_subgradient}), we have $f(\b{y}) \geq f(\b{x}^*) + \b{g}^\top (\b{y} - \b{x}^*), \forall \b{y}$.
If we have $\b{g} = \b{0} \in \partial f(\b{x}^*)$, we have $f(\b{y}) \geq f(\b{x}^*) + \b{0}^\top (\b{y} - \b{x}^*) = f(\b{x}^*)$ which means that $\b{x}^*$ is a minimizer. Q.E.D.
\end{proof}

The following lemma can be useful for calculation of subdifferential of functions. 
\begin{lemma}\label{lemma_properties_subdifferential}
Some useful properties for calculation of subdifferential of functions:
\begin{itemize}
\item For a smooth function or at points where the function is smooth, subdifferential has only one member which is the gradient: $\partial f(\b{x}) = \{\nabla f(\b{x})\}$.
\item Linear combination: If $f(\b{x}) = \sum_{i=1}^n a_i f_i(\b{x})$ with $a_i \geq 0$, then $\partial f(\b{x}) = \sum_{i=1}^n a_i \partial f_i(\b{x})$.
\item Affine transformation: If $f(\b{x}) = f_0(\b{A}\b{x} + \b{b})$, then $\partial f(\b{x}) = \b{A}^\top \partial f_0(\b{A}\b{x} + \b{b})$.
\item Point-wise maximum: Suppose $f(\b{x}) = \max\{f_1(\b{x}), \dots, f_n(\b{x})\}$ where $f_i$'s are differentiable. Let $I(\b{x}) := \{i | f_i = f(\b{x})\}$ states which function has the maximum value for the point $\b{x}$. At the any point other than the intersection point of functions (which is smooth), The subgradient is $g = \nabla f_i(\b{x})$ for $i \in I(\b{x})$. At the intersection point of two functions (which is not smooth), e.g. $f_i(\b{x}) = f_{i+1}(\b{x})$, we have:
\begin{align*}
\partial f(\b{x}) = \{\b{g}\, |\, t \nabla f_i(\b{x}) + (1-t) \nabla f_{i+1}(\b{x}), \forall t \in [0,1]\}.
\end{align*}
\end{itemize}
\end{lemma}

\subsubsection{Subgradient Method}

\textit{Subgradient method}, first proposed in \cite{shor2012minimization}, is used for solving the unconstrained optimization problem (\ref{equation_optimization_problem_unconstrained}) where the function $f(.)$ is not smooth, i.e., not differentiable, everywhere in its domain. It iteratively updates the solution as:
\begin{align}
\b{x}^{(k+1)} := \b{x}^{(k)} - \eta^{(k)} \b{g}^{(k)},
\end{align}
where $\b{g}^{(k)}$ is \textit{any} subgradient of function $f(.)$ in point $\b{x}$ at iteration $k$, i.e. $\b{g}^{(k)} \in \partial f(\b{x}^{(k)})$, and $\eta^{(k)}$ is the step size at iteration $k$. 
Comparing this update with Eq. (\ref{equation_GD_step_by_eta}) shows that gradient descent is a special case of the subgradient method because for a smooth function, gradient is the only member of the subdifferential set (see Lemma \ref{lemma_properties_subdifferential}); hence, the only subgradient is the gradient. 

\subsubsection{Stochastic Subgradient Method}

Consider the optimization problem (\ref{equation_optimization_problem_unconstrained_SGD}) where at least one of the $f_i(.)$ functions is not smooth.
\textit{Stochastic subgradient method} \cite{shor1998nondifferentiable} randomly samples one of the points to update the solution in every iteration:
\begin{align}\label{equation_SGD_step_by_eta_subgradient}
\b{x}^{(k+1)} := \b{x}^{(k)} - \eta^{(k)} \b{g}_i^{(k)},
\end{align}
where $\b{g}_i^{(k)} \in \partial f_i(\b{x}^{(k)})$.
Comparing this with Eq. (\ref{equation_SGD_step_by_eta}) shows that stochastic gradient descent is a special case of stochastic gradient descent because for a smooth function, gradient is the only member of the subdifferential set (see Lemma \ref{lemma_properties_subdifferential}).
Note that there is another method named stochastic subgradient method which uses a noisy unbiased subgradient for robustness to noisy data \cite{boyd2008stochastic}. Here, our focus was on random sampling of the point and not the noise.  

We can have \textit{mini-batch stochastic subgradient method} which is a generalization of mini-batch SGD for non-smooth functions. In this case, Eq. (\ref{equation_minibatch_SGD_step_by_eta}) is changed to:
\begin{align}\label{equation_minibatch_SGD_step_by_eta_subgradient}
\b{x}^{(k+1)} := \b{x}^{(k)} - \eta^{(k)} \frac{1}{b} \sum_{i \in \mathcal{B}_{k'}} \b{g}_i^{(k)}.
\end{align}

Note that, if the function is not smooth, we can also use subgradient instead of gradient in other stochastic methods such as SAG and SVRG, which were introduced before. For this, we need to use $\b{g}_i^{(k)}$ rather than $\nabla f(\b{x}^{(k)})$ in these methods.

\subsubsection{Projected Subgradient Method}

Consider the optimization problem (\ref{equation_constraint_set_optimization_problem}).
If the function $f(.)$ is not smooth, we can use he \textit{projected subgradient method} \cite{alber1998projected} which generalizes the projected gradient method introduced in Section \ref{section_projected_gradient_method}. Similar to Eq. (\ref{equation_projected_gradient_method_update}), it iteratively updates the solution as:
\begin{align}
\b{x}^{(k+1)} = \Pi_{\mathcal{S}}\big(\b{x}^{(k)} - \eta^{(k)} \b{g}^{(k)}\big),
\end{align}
until convergence of the solution. 

\section{Second-Order Optimization: Newton's Method}\label{section_second_order_methods}

\subsection{Newton's Method from the Newton-Raphson Root Finding Method}

We can find the root of a function $f: \b{x} \mapsto f(\b{x})$ by solving the equation $f(\b{x}) \overset{\text{set}}{=} 0$. The root of function can be found iteratively where we get closer to the root over iterations. One of the iterative root-finding methods is the \textit{Newton-Raphson method} \cite{stoer2013introduction}. 
In every iteration, it finds the next solution as:
\begin{align}\label{equation_Newton_Raphson_update}
\b{x}^{(k+1)} := \b{x}^{(k)} - \frac{f(\b{x}^{(k)})}{\nabla f(\b{x}^{(k)})},
\end{align}
where $\nabla f(\b{x}^{(k)})$ is the derivative of function w.r.t. $\b{x}$.
According to Eq. (\ref{equation_first_order_optimality_condition}), in unconstrained optimization, we can find the extremum (minimum or maximum) of the function by setting its derivative to zero, i.e., $\nabla f(\b{x}) \overset{\text{set}}{=} 0$.
Recall that Eq. (\ref{equation_Newton_Raphson_update}) was used for solving $f(\b{x}) \overset{\text{set}}{=} 0$. Therefore, for solving Eq. (\ref{equation_first_order_optimality_condition}), we can replace $f(\b{x})$ with $\nabla f(\b{x})$ in Eq. (\ref{equation_Newton_Raphson_update}):
\begin{align}\label{equation_Newton_update}
\b{x}^{(k+1)} := \b{x}^{(k)} - \eta^{(k)} \frac{\nabla f(\b{x}^{(k)})}{\nabla^2 f(\b{x}^{(k)})},
\end{align}
where $\nabla^2 f(\b{x}^{(k)})$ is the second derivative of function w.r.t. $\b{x}$ and we have included a step size at iteration $k$ denoted by $\eta^{(k)} > 0$. This step size can be either fixed or adaptive. 
If $\b{x}$ is multivariate, i.e. $\b{x} \in \mathbb{R}^d$, Eq. (\ref{equation_Newton_update}) is written as:
\begin{align}\label{equation_Newton_update_multivariate}
\b{x}^{(k+1)} := \b{x}^{(k)} - \eta^{(k)} \big(\nabla^2 f(\b{x}^{(k)})\big)^{-1} \nabla f(\b{x}^{(k)}),
\end{align}
where $\nabla f(\b{x}^{(k)}) \in \mathbb{R}^d$ is the gradient of function w.r.t. $\b{x}$ and $\nabla^2 f(\b{x}^{(k)}) \in \mathbb{R}^{d \times d}$ is the Hessian matrix w.r.t. $\b{x}$. 
Because of the second derivative or the Hessian, this optimization method is a second-order method. The name of this method is the \textit{Newton's method}. 

\subsection{Newton's Method for Unconstrained Optimization}

Consider the following optimization problem:
\begin{align}
\underset{\b{x}}{\text{minimize}}\quad f(\b{x}).
\end{align}
where $f(.)$ is a convex function.
Iterative optimization can be first-order or second-order. Iterative optimization updates solution iteratively as in Eq. (\ref{equation_update_point_numerical_optimization}).
The update continues until $\Delta\b{x}$ becomes very small which is the convergence of optimization.
In the first-order optimization, the step of updating is $\Delta\b{x} := -\nabla f(\b{x})$.
Near the optimal point $\b{x}^*$, gradient is very small so the second-order Taylor series expansion of function becomes:
\begin{align}
f(\b{x}) \approx &\,f(\b{x}^*) + \underbrace{\nabla f(\b{x}^*)^\top}_{\approx\, 0} (\b{x} - \b{x}^*) \nonumber\\
&+ \frac{1}{2} (\b{x} - \b{x}^*)^\top \nabla^2 f(\b{x}^*) (\b{x} - \b{x}^*) \nonumber\\
&\approx f(\b{x}^*) + \frac{1}{2} (\b{x} - \b{x}^*)^\top \nabla^2 f(\b{x}^*) (\b{x} - \b{x}^*). \label{equation_Newton_method_Taylor_expansion}
\end{align}
This shows that the function is almost quadratic near the optimal point. Following this intuition, Newton's method uses Hessian $\nabla^2 f(\b{x})$ in its updating step:
\begin{align}\label{equation_Newton_method_step}
\Delta\b{x} := - \nabla^2 f(\b{x})^{-1} \nabla f(\b{x}).
\end{align}
In the literature, this equation is sometimes restated to:
\begin{align}\label{equation_Newton_method_step_2}
\nabla^2 f(\b{x})\, \Delta\b{x} := - \nabla f(\b{x}).
\end{align}

\subsection{Newton's Method for Equality Constrained Optimization}

The optimization problem may have equality constraints:
\begin{equation}\label{equation_optimization_problem_equality_constraint}
\begin{aligned}
& \underset{\b{x}}{\text{minimize}}
& & f(\b{x}) \\
& \text{subject to}
& & \b{A} \b{x} = \b{b}.
\end{aligned}
\end{equation}
After a step of update by $\b{p} = \Delta \b{x}$, this optimization becomes:
\begin{equation}
\begin{aligned}
& \underset{\b{x}}{\text{minimize}}
& & f(\b{x} + \b{p}) \\
& \text{subject to}
& & \b{A} (\b{x} + \b{p}) = \b{b}.
\end{aligned}
\end{equation}
The Lagrangian of this optimization problem is:
\begin{align*}
\mathcal{L} = f(\b{x} + \b{p}) + \b{\nu}^\top (\b{A} (\b{x} + \b{p}) - \b{b}),
\end{align*}
where $\b{\nu}$ is the dual variable. 
The second-order Taylor series expansion of function $f(\b{x} + \b{p})$ is:
\begin{align}
f(\b{x} + \b{p}) \approx f(\b{x}) + \nabla f(\b{x})^\top \b{p} + \frac{1}{2} \b{p}^\top \nabla^2 f(\b{x}^*)\, \b{p}.
\end{align}
Substituting this into the Lagrangian gives:
\begin{align*}
\mathcal{L} = f(\b{x}) &+ \nabla f(\b{x})^\top \b{p} + \frac{1}{2} \b{p}^\top \nabla^2 f(\b{x}^*)\, \b{p} \\
&+ \b{\nu}^\top (\b{A} (\b{x} + \b{p}) - \b{b}).
\end{align*}
According to Eqs. (\ref{equation_stationarity_condition}) and (\ref{equation_derivative_Lagrangian_wrt_nu_zero}) in KKT conditions, the primal and dual residuals must be zero:
\begin{align}
& \nabla_{\b{x}} \mathcal{L} = \nabla f(\b{x}) + \nabla^2 f(\b{x})^\top \b{p} + \b{p}^\top \underbrace{\nabla^3 f(\b{x}^*)}_{\approx\, \b{0}}\, \b{p} \nonumber \\
&+ \b{A}^\top \b{\nu} \overset{\text{set}}{=} \b{0} \implies \nabla^2 f(\b{x})^\top \b{p} + \b{A}^\top \b{\nu} = - \nabla f(\b{x}), \label{equation_Newton_method_derivative_x_zero} \\
& \nabla_{\b{\nu}} \mathcal{L} = \b{A} (\b{x} + \b{p}) - \b{b} \overset{(a)}{=} \b{A} \b{p} \overset{\text{set}}{=} \b{0}, \label{equation_Newton_method_derivative_nu_zero}
\end{align}
where we have $\nabla^3 f(\b{x}^*) \approx 0$ because the gradient of function at the optimal point vanishes according to Eq. (\ref{equation_first_order_optimality_condition}) and $(a)$ is because of the constraint $\b{Ax} - \b{b} = \b{0}$ in problem (\ref{equation_optimization_problem_equality_constraint}).
Eqs. (\ref{equation_Newton_method_derivative_x_zero}) and (\ref{equation_Newton_method_derivative_nu_zero}) can be written as a system of equations:
\begin{align}\label{equation_Newton_method_equality_constraints_solution}
\begin{bmatrix}
\nabla^2 f(\b{x})^\top & \b{A}^\top \\
\b{A} & \b{0}
\end{bmatrix}
\begin{bmatrix}
\b{p} \\
\b{\nu}
\end{bmatrix} 
=
\begin{bmatrix}
-\nabla f(\b{x}) \\
\b{0}
\end{bmatrix}.
\end{align}
Solving this system of equations gives the desired step $\b{p}$ (i.e., $\Delta \b{x}$) for updating the solution at the iteration.

\hfill\break
\textbf{-- Starting with Non-feasible Initial Point:}
Newton's method can even start with a non-feasible point which does not satisfy all the constraints.
If the initial point for optimization is not a feasible point, i.e., $\b{Ax} - \b{b} \neq \b{0}$, Eq. (\ref{equation_Newton_method_derivative_nu_zero}) becomes:
\begin{align}
\nabla_{\b{\nu}} \mathcal{L} = \b{A} (\b{x} + \b{p}) - \b{b}  \overset{\text{set}}{=} \b{0} \implies \b{Ap} = -(\b{Ax} - \b{b}).
\end{align}
Hence, for the first iteration, we solve the following system rather than Eq. (\ref{equation_Newton_method_equality_constraints_solution}):
\begin{align}\label{equation_Newton_method_equality_constraints_solution_nonFeasible}
\begin{bmatrix}
\nabla^2 f(\b{x})^\top & \b{A}^\top \\
\b{A} & \b{0}
\end{bmatrix}
\begin{bmatrix}
\b{p} \\
\b{\nu}
\end{bmatrix} 
=
-
\begin{bmatrix}
\nabla f(\b{x}) \\
\b{Ax} - \b{b}
\end{bmatrix},
\end{align}
and we use Eq. (\ref{equation_Newton_method_equality_constraints_solution_nonFeasible}) for the rest of iterations because the next points will be in the feasibility set (because we force the solutions to satisfy $\b{Ax} = \b{b}$). 

\subsection{Interior-Point and Barrier Methods: Newton's Method for Inequality Constrained Optimization}\label{section_interior_point_method}

The optimization problem may have inequality constraints:
\begin{equation}\label{equation_optimization_problem_interior_point_method}
\begin{aligned}
& \underset{\b{x}}{\text{minimize}}
& & f(\b{x}) \\
& \text{subject to}
& & y_i(\b{x}) \leq 0, \quad i \in \{1, \dots, m_1\}, \\
& & & \b{A} \b{x} = \b{b}.
\end{aligned}
\end{equation}
We can solve constrained optimization problems using \textit{Barrier methods}, also known as \textit{interior-point methods} \cite{nesterov1994interior,potra2000interior,boyd2004convex,wright2005interior}. Interior-point methods were first proposed by \cite{dikin1967iterative}. 
The interior-point method is also referred to as the \textit{Unconstrained Minimization Technique (UMT)} or \textit{Sequential UMT (SUMT)} \cite{fiacco1967sequential} because it converts the problem to an unconstrained problem and solves it iteratively. 

The barrier methods or the interior-point methods, convert inequality constrained problems to equality constrained or unconstrained problems. Ideally, we can do this conversion using the indicator function $\mathbb{I}(.)$ which is zero if its input condition is satisfied and is infinity otherwise (n.b. the indicator function in optimization literature is not like the indicator in data science which is one if its input condition is satisfied and is zero otherwise). The problem is converted to:
\begin{equation}\label{equation_log_barrier_optimization}
\begin{aligned}
& \underset{\b{x}}{\text{minimize}}
& & f(\b{x}) + \sum_{i=1}^{m_1} \mathbb{I}(y_i(\b{x}) \leq 0) \\
& \text{subject to}
& & \b{A} \b{x} = \b{b}.
\end{aligned}
\end{equation}
The indicator function is not differentiable because it is not smooth: 
\begin{align}
\mathbb{I}(y_i(\b{x}) \leq 0) := 
\left\{
    \begin{array}{ll}
        0 & \mbox{if } y_i(\b{x}) \leq 0 \\
        \infty & \mbox{if } y_i(\b{x}) > 0.
    \end{array}
\right.
\end{align}
Hence, we can approximate it with differentiable functions called the \textit{barrier functions} \cite{boyd2004convex,nesterov2018lectures}.
A barrier function is logarithm, named the \textit{logarithmic barrier} or \textit{log barrier} in short. It approximates the indicator function by:
\begin{align}\label{equation_log_barrier}
\mathbb{I}(y_i(\b{x}) \leq 0) \approx -\frac{1}{t} \log(-y_i(\b{x})),
\end{align}
where $t > 0$ (usually a large number such as $t = 10^6$) and the approximation becomes more accurate by $t \rightarrow \infty$.
It changes the problem to:
\begin{equation}\label{equation_log_barrier_optimization_withLogBarrier}
\begin{aligned}
& \underset{\b{x}}{\text{minimize}}
& & f(\b{x}) - \frac{1}{t} \sum_{i=1}^{m_1} \log(-y_i(\b{x})) \\
& \text{subject to}
& & \b{A} \b{x} = \b{b}.
\end{aligned}
\end{equation}
This optimization problem is an equality constrained optimization problem which we already explained how to solve. 
Note that there exist many approximations for the barrier. One of mostly used methods is the logarithmic barrier.  

The iterative solutions of the interior-point method satisfy Eq. (\ref{equation_iterative_optimization_series}) and follow Fig. \ref{figure_iterative_primal_dual}. 
If the optimization problem is a convex problem, the solution of interior-point method is the global solution; otherwise, the solution is local. 
The interior-point and barrier methods are used in many optimization toolboxes such as CVX \cite{grant2009cvx}.

\hfill\break
\textbf{-- Accuracy of the log barrier method:}
In the following theorem, we discuss the accuracy of the log barrier method.



\begin{theorem}[On the sub-optimality of log-barrier method]\label{theorem_sub_optimality_logBarrier}
Let the optimum of problems (\ref{equation_optimization_problem_interior_point_method}) and (\ref{equation_log_barrier_optimization_withLogBarrier}) be denoted by $f^*$ and $f_r^*$, respectively. 
We have:
\begin{align}\label{equation_sub_optimality_logBarrier}
f^* - \frac{m_1}{t} \leq f_r^* \leq f^*,
\end{align}
meaning that the optimum of problem (\ref{equation_log_barrier_optimization_withLogBarrier}) is no more than $m_1/t$ from the optimum of problem (\ref{equation_optimization_problem_interior_point_method}). 
\end{theorem}
\begin{proof}
Proof is available in Appendix \ref{app_sub_optimality_logBarrier}. 
\end{proof}

Theorem \ref{theorem_sub_optimality_logBarrier} indicates that by $t \rightarrow \infty$, the log-barrier method is more accurate; i.e., the solution of problem (\ref{equation_log_barrier_optimization_withLogBarrier}) is more accurately close to the solution of problem (\ref{equation_optimization_problem_interior_point_method}). 
This is expected because the approximation in Eq. (\ref{equation_log_barrier}) gets more accurate by increasing $t$. 
Note that by increasing $t$, optimization gets more accurate but harder to solve and slower to converge. 

\subsection{Wolfe Conditions and Line-Search in Newton's Method}\label{section_Wolfe_conditions}


In Sections \ref{section_GD_line_search} and \ref{section_GD_Armijo_line_search}, we introduced line-search for gradient descent. We have line-search for second-order optimization, too. 
Line-search for second-order optimization checks two conditions. These conditions are called the \textit{Wolfe conditions} \cite{wolfe1969convergence}. 
For finding the suitable step size $\eta^{(k)}$, at iteration $k$ of optimization, the Wolfe conditions are checked.
Here, we do not include the step size $\eta$ in $\b{p} = \Delta\b{x}$ and the step at iteration $k$ is $\b{p}^{(k)} = - \nabla^2 f(\b{x}^{(k)})^{-1} \nabla f(\b{x}^{(k)})$ according to Eq. (\ref{equation_Newton_method_step}). 
The Wolfe conditions are:
\begin{align}
& f(\b{x}^{(k)} + \eta^{(k)} \b{p}^{(k)}) \leq f(\b{x}^{(k)}) +  c_1 \eta^{(k)} \b{p}^{(k)\top} f(\b{x}^{(k)}), \\
& -\b{p}^{(k)\top} \nabla f(\b{x}^{(k)} + \eta^{(k)} \b{p}^{(k)}) \leq -c_2 \b{p}^{(k)\top} f(\b{x}^{(k)}),
\end{align}
where $0 < c_1 < c_2 < 1$ are the parameters of Wolfe conditions. It is recommended in \cite{nocedal2006numerical} to have $c_1 = 10^{-4}$ and $c_2 = 0.9$. 
The first condition is the \textit{Armijo} condition \cite{armijo1966minimization} which ensures the step size $\eta^{(k)}$ decreases the function value sufficiently. The second condition is the \textit{curvature condition} which ensures the step size $\eta^{(k)}$ decreases the function slope sufficiently. 
In quasi-Newton's method (introduced later in Section \ref{section_quasi_Newton_method}), the curvature condition makes sure the approximation of Hessian matrix remains positive definite. 
The Armijo and curvature conditions give an upper-bound and lower-bound on the step size, respectively.
There also exists a \textit{strong curvature condition}:
\begin{align}
& |\b{p}^{(k)\top} \nabla f(\b{x}^{(k)} + \eta^{(k)} \b{p}^{(k)})| \leq c_2 |\b{p}^{(k)\top} f(\b{x}^{(k)})|,
\end{align}
which can be used instead of the curvature condition.
Note that the Wolfe conditions can also be used for line-search in first-order methods. 

\subsection{Fast Solving System of Equations in Newton's Method}

In unconstrained Newton's method, the update of solution which is Eq. (\ref{equation_Newton_method_step_2}) is in the form of a system of linear equations. 
In constrained Newton's method (and hence, in the interior-point method), the update of solution which is Eq. (\ref{equation_Newton_method_equality_constraints_solution}) is also in the form of a system of linear equations. 
Therefore, every iteration of all optimization methods is reduced to a system of equations such as:
\begin{align}\label{equation_system_of_equation_Newton_method}
\b{M z} = \b{q},
\end{align}
where we need to calculate $\b{z}$.
Therefore, the optimization toolboxes, such as CVX \cite{grant2009cvx}, solve a system of equations to find the solution at every iteration. 
If the dimensionality of data or the number of constraints is large, the number of equations and the size of matrices in the system of equations increase. Solving the large system of equations is very time-consuming. Hence, some methods have been developed to accelerate solving the system of equations. Here, we review some of these methods. 

\subsubsection{Decomposition Methods}


We can use various matrix decomposition/factorization methods \cite{golub2013matrix} for decomposing the coefficient matrix $\b{M}$ and solving the system of equations (\ref{equation_system_of_equation_Newton_method}). We review some of them here. 

\hfill\break
\textbf{-- LU decomposition:}
We use LU decomposition to decompose $\b{M} = \b{PLU}$. We have:
\begin{align*}
\b{M z} = \b{P}\underbrace{\b{L}\underbrace{\b{U}\b{z}}_{=\b{w}_2}}_{=\b{w}_1} = \b{q},
\end{align*}
where we define $\b{w}_2 := \b{U z}$ and $\b{w}_1 := \b{L} \b{w}_2$. 
Hence, we can solve the system of equations as:
\begin{enumerate}
\item LU decomposition: $\b{M} = \b{PLU}$
\item permutation: Solve $\b{P}\b{w}_1 = \b{q}$ to find $\b{w}_1$
\item forward subtraction: Solve $\b{L}\b{w}_2 = \b{w}_1$ to find $\b{w}_2$
\item backward subtraction: Solve $\b{U}\b{z} = \b{w}_2$ to find $\b{z}$
\end{enumerate}

\hfill\break
\textbf{-- Cholesky decomposition:}
In most cases of optimization, the coefficient matrix is positive definite. for example, in Eq. (\ref{equation_Newton_method_step_2}), the coefficient matrix is the Hessian which is positive definite. 
Therefore, we can use Cholesky decomposition to decompose $\b{M} = \b{L}\b{L}^\top$. We have:
\begin{align*}
\b{M z} = \b{L}\underbrace{\b{L}^\top \b{z}}_{=\b{w}_1} = \b{q},
\end{align*}
where we define $\b{w}_1 := \b{L}^\top \b{z}$. 
Hence, we can solve the system of equations as:
\begin{enumerate}
\item Cholesky decomposition: $\b{M} = \b{L}\b{L}^\top$
\item forward subtraction: Solve $\b{L}\b{w}_1 = \b{q}$ to find $\b{w}_1$
\item backward subtraction: Solve $\b{L}^\top \b{z} = \b{w}_1$ to find $\b{z}$
\end{enumerate}

\hfill\break
\textbf{-- Schur complement:}
Assume the system of equations can be divided into blocks:
\begin{align*}
\b{Mz} =
\begin{bmatrix}
\b{M}_{11} & \b{M}_{12} \\
\b{M}_{21} & \b{M}_{22}
\end{bmatrix}
\begin{bmatrix}
\b{z}_1 \\
\b{z}_2
\end{bmatrix} 
=
\begin{bmatrix}
\b{q}_1 \\
\b{q}_2
\end{bmatrix}.
\end{align*}
From this, we have:
\begin{align*}
& \b{M}_{11} \b{z}_1 + \b{M}_{12} \b{z}_2 = \b{q}_1 \implies \b{z}_1 = \b{M}_{11}^{-1} (\b{q}_1 - \b{M}_{12} \b{z}_2). \\
& \b{M}_{21} \b{z}_1 + \b{M}_{22} \b{z}_2 = \b{q}_2 \\
&\implies \b{M}_{21} (\b{M}_{11}^{-1} (\b{q}_1 - \b{M}_{12} \b{z}_2)) + \b{M}_{22} \b{z}_2 = \b{q}_2 \\
& \implies (\b{M}_{22} - \b{M}_{21}\b{M}_{11}^{-1}\b{M}_{12}) \b{z}_2 = \b{q}_2 - \b{M}_{21} \b{M}_{11}^{-1} \b{q}_1.
\end{align*}
The term $(\b{M}_{22} - \b{M}_{21}\b{M}_{11}^{-1}\b{M}_{12})$ is the Schur complement \cite{zhang2006schur} of block matrix $\b{M}_{11}$ in matrix $\b{M}$.
We assume that the block matrix $\b{M}_{11}$ is not singular so its inverse exists. 
We use the Schur complement to solve the system of equations as:
\begin{enumerate}
\item Calculate $\b{M}_{11}^{-1} \b{M}_{12}$ and $\b{M}_{11}^{-1} \b{q}$
\item Calculate $\widetilde{\b{M}} := \b{M}_{22} - \b{M}_{21}(\b{M}_{11}^{-1}\b{M}_{12})$ and $\widetilde{\b{q}} := \b{q}_2 - \b{M}_{21} (\b{M}_{11}^{-1} \b{q}_1)$
\item Solve $\widetilde{\b{M}} \b{z}_2 = \widetilde{\b{q}}$ (as derived above) to find $\b{z}_2$
\item Solve $\b{M}_{11} \b{z}_1 = \b{q}_1 - \b{M}_{12} \b{z}_2$ (as derived above) to find $\b{z}_1$
\end{enumerate}

\subsubsection{Conjugate Gradient Method}


The \textit{conjugate gradient} method, proposed in \cite{hestenes1952methods}, iteratively solves Eq. (\ref{equation_system_of_equation_Newton_method}) faster than the regular solution. Its pacing shows off better when the matrices are very large. 
A good book on conjugate gradient is {\citep[Chapter 2]{kelley1995iterative}}.
It is noteworthy that \textit{truncated Newton's methods} \cite{nash2000survey}, which approximate the Hessian for large-scale optimization, usually use conjugate gradient as their approximation method for calculating the search direction. 

\begin{definition}[Conjugate vectors]
Two non-zero vectors $\b{x}$ and $\b{y}$ are conjugate if $\b{x}^\top \b{M} \b{y} = \b{0}$ where $\b{M} \succ \b{0}$.
\end{definition}

\begin{definition}[Krylov subspace \cite{krylov1931numerical}]
The order-$r$ Krylov subspace, denoted by $\mathcal{K}_r$, is spanned by the following bases:
\begin{align}
\mathcal{K}_r(\b{M},\b{q}) = \textbf{span}\{\b{q}, \b{Mq}, \b{M}^2\b{q}, \dots, \b{M}^{r-1}\b{q}\}.
\end{align}
\end{definition}
The solution to Eq. (\ref{equation_system_of_equation_Newton_method}) should satisfy $\b{z} = \b{M}^{-1}\b{q}$. Therefore, the solution to Eq. (\ref{equation_system_of_equation_Newton_method}) lies in the Krylov subspace. The conjugate gradient method approximates this solution lying in the Krylov subspace. Every iteration of conjugate gradient can be seen as projection onto the Krylov subspace.

According to Eq. (\ref{equation_first_order_optimality_condition}), the solution to Eq. (\ref{equation_system_of_equation_Newton_method}) minimizes the function:
\begin{align*}
f(\b{z}) = \frac{1}{2} \b{z}^\top \b{M} \b{z} - \b{z}^\top \b{q},
\end{align*}
because Eq. (\ref{equation_system_of_equation_Newton_method}) is the gradient of this function, i.e., $\nabla f(\b{z}) = \b{Mz} - \b{q}$.
conjugate gradient iteratively solves Eq. (\ref{equation_system_of_equation_Newton_method}) as:
\begin{align}\label{equation_conjugate_gradient_update}
& \b{z}^{(k+1)} := \b{z}^{(k)} + \eta^{(k)} \b{p}^{(k)}.
\end{align}
It starts by moving toward minus gradient as gradient descent does. Then, it uses the conjugate of gradient. This is the reason for the name of this method. 

At iteration $k$, the residual (error) for fulfilling  Eq. (\ref{equation_system_of_equation_Newton_method}) is:
\begin{align}\label{equation_conjugate_gradient_residual}
\b{r}^{(k)} := \b{q} - \b{M}\b{z}^{(k)} = -\nabla f(\b{z}^{(k)}).
\end{align}
We also have:
\begin{align*}
&\b{r}^{(k+1)} - \b{r}^{(k)} \overset{(\ref{equation_conjugate_gradient_residual})}{=} \b{q} - \b{M}\b{z}^{(k+1)} - \b{q} + \b{M}\b{z}^{(k)} \\
&\overset{(\ref{equation_conjugate_gradient_update})}{=} \b{M}(-\b{z}^{(k)} - \eta^{(k)} \b{p}^{(k)} + \b{z}^{(k)}) = -\eta^{(k)} \b{M} \b{p}^{(k)}.
\end{align*}
Initially, the direction is this residual, $\b{p}^{(0)} = \b{r}^{(0)}$ as in gradient descent (see Eq. (\ref{equation_GD_step_by_eta})).
If we take derivative of $f(\b{z}^{(k+1)}) \overset{(\ref{equation_conjugate_gradient_update})}{=} f(\b{z}^{(k)} + \eta^{(k)} \b{p}^{(k)})$ w.r.t. $\eta^{(k)}$, we have:
\begin{align*}
&\frac{\partial }{\partial \eta^{(k)}} f(\b{z}^{(k)} + \eta^{(k)} \b{p}^{(k)}) \overset{\text{set}}{=} 0 \\
&\implies \eta^{(k)} = \frac{\b{p}^{(k)\top} (\b{q} - \b{M} \b{x}^{(k)})}{\b{p}^{(k)\top} \b{M} \b{p}^{(k)}} \overset{(\ref{equation_conjugate_gradient_residual})}{=} \frac{\b{p}^{(k)\top} \b{r}^{(k)}}{\b{p}^{(k)\top} \b{M} \b{p}^{(k)}}.
\end{align*}
The conjugate gradient method is shown in Algorithm \ref{algorithm_conjugate_gradient}. As this algorithm shows, the direction of update $\b{p}$ is found by a linear combination of the residual (which was initialized by negative gradient as in gradient descent) and the previous direction. The weight of previous direction in this linear combination is $\beta$ which gets smaller if the residual of this step is much smaller than the residual of previous iteration. This formula for $\beta = (\b{r}^{(k+1)\top} \b{r}^{(k+1)}) / (\b{r}^{(k)\top} \b{r}^{(k)})$ is also used in the Fletcher-Reeves nonlinear conjugate gradient method \cite{fletcher1964function}, introduced in the next section.
The conjugate gradient method returns a $\b{z}$ as an approximation to the solution to Eq. (\ref{equation_system_of_equation_Newton_method}). 

\SetAlCapSkip{0.5em}
\IncMargin{0.8em}
\begin{algorithm2e}[!t]
\DontPrintSemicolon
    Initialize: $\b{z}^{(0)}$\;
    $\b{r}^{(0)} := -\nabla f(\b{z}^{(0)}) = \b{q} - \b{M}\b{z}^{(0)}$, $\b{p}^{(0)} := \b{r}^{(0)}$\;
    \For{iteration $k = 0, 1, \dots$}{
        $\eta^{(k)} = \frac{\b{p}^{(k)\top} \b{r}^{(k)}}{\b{p}^{(k)\top} \b{M} \b{p}^{(k)}}$\;
        $\b{z}^{(k+1)} := \b{z}^{(k)} + \eta^{(k)} \b{p}^{(k)}$\;
        $\b{r}^{(k+1)} := \b{r}^{(k)} - \eta^{(k)} \b{M} \b{p}^{(k)}$\;
        \If{$\|\b{r}^{(k+1)}\|_2$ is small}{
            \textbf{Break} the loop.\;
        }
        $\beta^{(k+1)} := \frac{\b{r}^{(k+1)\top} \b{r}^{(k+1)}}{\b{r}^{(k)\top} \b{r}^{(k)}} = \frac{\|\b{r}^{(k+1)}\|_2^2}{\|\b{r}^{(k)}\|_2^2}$\;
        $\b{p}^{(k+1)} := \b{r}^{(k+1)} + \beta^{(k+1)} \b{p}^{(k)}$
    }
    \textbf{Return} $\b{z}^{(k+1)}$
\caption{The conjugate gradient method to solve Eq. (\ref{equation_system_of_equation_Newton_method}).}\label{algorithm_conjugate_gradient}
\end{algorithm2e}
\DecMargin{0.8em}

\subsubsection{Nonlinear Conjugate Gradient Method}


As we saw, conjugate gradient is for solving linear equations. 
Nonlinear Conjugate Gradient (NCG) generalizes conjugate gradient to nonlinear functions. 
Recall that conjugate gradient solves Eq. (\ref{equation_system_of_equation_Newton_method}):
\begin{align*}
\b{M z} = \b{q} &\implies \b{M}^\top\b{M z} = \b{M}^\top\b{q} \\
&\implies 2\b{M}^\top(\b{M z} - \b{q}) = \b{0}.
\end{align*}
The goal of NCG is to find the minimum of a quadratic function:
\begin{align}\label{equation_NCG_function}
f(\b{z}) = \|\b{M z} - \b{q}\|_2^2,
\end{align}
using its gradient. The gradient of this function is $\nabla f(\b{z}) = 2\b{M}^\top(\b{M z} - \b{q})$ which was found above. 
The NCG method is shown in Algorithm \ref{algorithm_nonlinear_conjugate_gradient} which is very similar to Algorithm \ref{algorithm_conjugate_gradient}. It uses steepest descent for updating the solution (see Eq. (\ref{algorithm_conjugate_gradient})). 
The direction for update is found by a linear combination of the residual, initialized by negative gradient as in gradient descent, and the direction of previous iteration.
Several formulas exist for the weight $\beta$ for the previous direction in the linear combination:
\begin{equation}\label{equation_NCG_beta}
\begin{aligned}
& \beta_1^{(k+1)} := \frac{\b{r}^{(k+1)\top} \b{r}^{(k+1)}}{\b{r}^{(k)\top} \b{r}^{(k)}} = \frac{\|\b{r}^{(k+1)}\|_2^2}{\|\b{r}^{(k)}\|_2^2}, \\
& \beta_2^{(k+1)} := \frac{\b{r}^{(k+1)\top} (\b{r}^{(k+1)} - \b{r}^{(k)})}{\b{r}^{(k)\top} \b{r}^{(k)}}, \\
& \beta_3^{(k+1)} := -\frac{\b{r}^{(k+1)\top} (\b{r}^{(k+1)} - \b{r}^{(k)})}{\b{p}^{(k)\top} (\b{r}^{(k+1)} - \b{r}^{(k)})}, \\
& \beta_4^{(k+1)} := -\frac{\b{r}^{(k)\top} \b{r}^{(k)}}{\b{p}^{(k)\top} (\b{r}^{(k+1)} - \b{r}^{(k)})}.
\end{aligned}
\end{equation}
The $\beta_1$, $\beta_2$, $\beta_3$, and $\beta_4$ are the formulas for Fletcher-Reeves \cite{fletcher1964function}, Polak-Ribi{\`e}re \cite{polak1969note}, Hestenes-Stiefel \cite{hestenes1952methods}, and Dai-Yuan \cite{dai1999nonlinear} methods, respectively. 
In all these formulas, $\beta$ gets smaller if the residual of next iteration gets much smaller than the previous residual. 
The NCG method returns a $\b{z}$ as an approximation to the minimizer of the nonlinear function in Eq. (\ref{equation_NCG_function}). 

\SetAlCapSkip{0.5em}
\IncMargin{0.8em}
\begin{algorithm2e}[!t]
\DontPrintSemicolon
    Initialize: $\b{z}^{(0)}$\;
    $\b{r}^{(0)} := -\nabla f(\b{z}^{(0)})$, $\b{p}^{(0)} := \b{r}^{(0)}$\;
    \For{iteration $k = 0, 1, \dots$}{
        $\b{r}^{(k+1)} := -\nabla f(\b{z}^{(k+1)})$\;
        Compute $\beta^{(k+1)}$ by one of Eqs. (\ref{equation_NCG_beta})\;
        $\b{p}^{(k+1)} := \b{r}^{(k+1)} + \beta^{(k+1)} \b{p}^{(k)}$\;
        $\eta^{(k+1)} := \arg\min_{\eta} f(\b{z}^{(k+1)} + \eta\, \b{p}^{(k+1)})$\;
        $\b{z}^{(k+1)} := \b{z}^{(k)} + \eta^{(k+1)} \b{p}^{(k+1)}$\;
    }
    \textbf{Return} $\b{z}^{(k+1)}$
\caption{The NCG method to find the minimum of Eq. (\ref{equation_NCG_function}).}\label{algorithm_nonlinear_conjugate_gradient}
\end{algorithm2e}
\DecMargin{0.8em}

\subsection{Quasi-Newton's Methods}\label{section_quasi_Newton_method}


\subsubsection{Hessian Approximation}

Similar to what we did in Eq. (\ref{equation_Newton_method_Taylor_expansion}), we can approximate the function at the updated solution by its second-order Taylor series:
\begin{align*}
&f(\b{x}^{(k+1)}) = f(\b{x}^{(k)} + \b{p}^{(k)}) \\
&\approx f(\b{x}^{(k)}) + \nabla f(\b{x}^{(k)})^\top \b{p}^{(k)} + \frac{1}{2} \b{p}^{(k)\top} \b{B}^{(k)} \b{p}^{(k)}. 
\end{align*}
where $\b{p} = \Delta\b{x}$ is the step size and $\b{B}^{(k)} = \nabla^2 f(\b{x}^{(k)})$ is the Hessian matrix at iteration $k$. 
Taking derivative from this equation w.r.t. $\b{p}$ gives:
\begin{align}
\nabla f(\b{x}^{(k)} + \b{p}^{(k)}) \approx \nabla f(\b{x}^{(k)}) + \b{B}^{(k)} \b{p}^{(k)}.
\end{align}
This equation is called the \textit{secant equation}. Setting this derivative to zero, for optimization, gives:
\begin{align}\label{equation_quasi_Newton_method_step}
\b{B}^{(k)} \b{p}^{(k)} = -\nabla f(\b{x}^{(k)}) \implies \b{p}^{(k)} = - \b{H}^{(k)} \nabla f(\b{x}^{(k)}), 
\end{align}
where $\b{H}^{(k)} := (\b{B}^{(k)})^{-1}$ is the inverse of Hessian matrix. 
This equation is the previously found Eqs. (\ref{equation_Newton_method_step}) and (\ref{equation_Newton_method_step_2}).
Note that although the letter H seems to be used for Hessian, literature uses $\b{H}$ to denote the (approximation of) inverse of Hessian. 
Considering the step size, we can write the step $\b{s}^{(k)}$ as:
\begin{equation}
\begin{aligned}
\mathbb{R}^d \ni \b{s}^{(k)} := \Delta \b{x} &= \b{x}^{(k+1)} - \b{x}^{(k)} \\
&= - \eta^{(k)} \b{H}^{(k)} \nabla f(\b{x}^{(k)}),
\end{aligned}
\end{equation}
where the step size is found by the Wolfe conditions in line-search (see Section \ref{section_Wolfe_conditions}).

Computation of the Hessian matrix or its inverse is usually expensive in Newton's method. 
One way to approximate the Newton's solution at every iteration is the conjugate gradient method, which was introduced before. Another approach for approximating the solution is the \textit{quasi-Newton's methods} which approximate the (inverse) Hessian matrix.
The quasi-Newton's methods approximate Hessian based on the step $\b{p}^{(k)}$, the difference of gradients between iterations:
\begin{align}
\mathbb{R}^d \ni \b{y}^{(k)} := \nabla f(\b{x}^{(k+1)}) - \nabla f(\b{x}^{(k)}),
\end{align}
and the previous approximated Hessian $\b{B}^{(k)}$ or its inverse $\b{H}^{(k)}$.

Some papers approximate the Hessian by a diagonal matrix {\citep[Appendix C.1]{lee2007nonlinear}}, \cite{andrei2019diagonal}. In this situation, the inverse of approximated Hessian is simply the inverse of its diagonal elements. 
Some methods use a dense approximation for Hessian matrix. Examples for these are BFGS, DFP, Broyden, and SR1 methods, which will be introduced in the following. 
Some other methods, such as LBFGS introduced later, approximate the Hessian matrix only by a scalar. 

\subsubsection{Quasi-Newton's Algorithms}


The most well-known algorithm for quasi-Newton's method is \textit{Broyden-Fletcher-Goldfarb-Shanno (BFGS)} \cite{fletcher1987practical,dennis1996numerical}. \textit{Limited-memory BFGS (LBFGS)} \cite{nocedal1980updating,liu1989limited} is a simplified version of BFGS which utilizes less memory. 
Some other quasi-Newton's methods are \textit{Davidon-Fletcher-Powell (DFP)} \cite{davidon1991variable,fletcher1987practical}, \textit{Broyden method} \cite{broyden1965class}, and \textit{Symmetric Rank-one (SR1)} \cite{conn1991convergence}.
In the following, we review the approximations of Hessian matrix and its inverse by different methods.
More explanation on these methods can be found in {\citep[Chapter 6]{nocedal2006numerical}}.

We define:
\begin{align}
& \mathbb{R} \ni \rho^{(k)} := \frac{1}{\b{y}^{(k)\top} \b{s}^{(k)}}, \\
& \mathbb{R}^{d \times d} \ni \b{V}^{(k)} := \b{I} - \rho^{(k)} \b{y}^{(k)} \b{s}^{(k)\top} = \b{I} - \frac{\b{y}^{(k)} \b{s}^{(k)\top}}{\b{y}^{(k)\top} \b{s}^{(k)}},
\end{align}
where $\b{I}$ is the identity matrix. 

\textbf{-- BFGS: }
The approximations in BFGS method are:
\begin{equation}\label{equation_BFGS_Hessian}
\begin{aligned}
& \b{B}^{(k+1)} := \b{B}^{(k)} + \rho^{(k)} \b{y}^{(k)} \b{y}^{(k)\top} \\
&~~~~~~~~~~~~~~~~~~~~~~~~~~~~ - \frac{\b{B}^{(k)} \b{s}^{(k)} \b{s}^{(k)\top} \b{B}^{(k)\top}}{\b{s}^{(k)\top} \b{B}^{(k)} \b{s}^{(k)}}, \\
& \b{H}^{(k+1)} := \b{V}^{(k)\top} \b{H}^{(k)} \b{V}^{(k)} + \rho^{(k)} \b{s}^{(k)} \b{s}^{(k)\top}. 
\end{aligned}
\end{equation}

\textbf{-- DFP: }
The approximations in DFP method are:
\begin{equation}\label{equation_DFP_Hessian}
\begin{aligned}
& \b{B}^{(k+1)} := \b{V}^{(k)} \b{B}^{(k)} \b{V}^{(k)\top} + \rho^{(k)} \b{y}^{(k)} \b{y}^{(k)\top}, \\
& \b{H}^{(k+1)} := \b{H}^{(k)} + \rho^{(k)} \b{s}^{(k)} \b{s}^{(k)\top} \\
&~~~~~~~~~~~~~~~~~~~~~~~~~~~~ - \frac{\b{H}^{(k)} \b{y}^{(k)} \b{y}^{(k)\top} \b{H}^{(k)\top}}{\b{y}^{(k)\top} \b{H}^{(k)} \b{y}^{(k)}}.
\end{aligned}
\end{equation}

\textbf{-- Broyden: }
The approximations in Broyden method are:
\begin{equation}
\begin{aligned}
& \b{B}^{(k+1)} := \b{B}^{(k)} + \frac{(\b{y}^{(k)} - \b{B}^{(k)} \b{s}^{(k)})\, \b{s}^{(k)\top}}{\b{s}^{(k)\top} \b{s}^{(k)}}, \\
& \b{H}^{(k+1)} := \b{H}^{(k)} + \frac{(\b{s}^{(k)} - \b{H}^{(k)} \b{y}^{(k)}) \b{s}^{(k)\top} \b{H}^{(k)}}{\b{s}^{(k)\top} \b{H}^{(k)} \b{y}^{(k)}}.
\end{aligned}
\end{equation}

\textbf{-- SR1: }
The approximations in SR1 method are:
\begin{equation}
\begin{aligned}
& \b{B}^{(k+1)} := \b{B}^{(k)} + \frac{(\b{y}^{(k)} - \b{B}^{(k)} \b{s}^{(k)}) (\b{y}^{(k)} - \b{B}^{(k)} \b{s}^{(k)})^\top}{(\b{y}^{(k)} - \b{B}^{(k)} \b{s}^{(k)})^\top \b{s}^{(k)}}, \\
& \b{H}^{(k+1)} := \b{H}^{(k)} + \frac{(\b{s}^{(k)} - \b{H}^{(k)} \b{y}^{(k)}) (\b{s}^{(k)} - \b{H}^{(k)} \b{y}^{(k)})^\top}{(\b{s}^{(k)} - \b{H}^{(k)} \b{y}^{(k)})^\top \b{y}^{(k)}}.
\end{aligned}
\end{equation}
Comparing Eqs. (\ref{equation_BFGS_Hessian}) and (\ref{equation_DFP_Hessian}) shows that the BFGS and DFP methods are dual of each other. Experiments have shown that BFGS often outperforms DFP \cite{avriel2003nonlinear}.

\SetAlCapSkip{0.5em}
\IncMargin{0.8em}
\begin{algorithm2e}[!t]
\DontPrintSemicolon
    Initialize the solution $\b{x}^{(0)}$\;
    $\b{H}^{(0)} := \frac{1}{\|\nabla f(\b{x}^{(0)})\|_2} \b{I}$\;
    \For{$k = 0,1,\dots$ (until convergence)}{
        $\b{p}^{(k)} \gets \text{GetDirection}(-\nabla f(\b{x}^{(k)}), k, 1)$\;
        $\eta^{(k)} \gets $ Line-search with Wolfe conditions\;
        $\b{x}^{(k+1)} := \b{x}^{(k)} - \eta^{(k)} \b{p}^{(k)}$\;
        $\b{s}^{(k)} := \b{x}^{(k+1)} - \b{x}^{(k)} = \eta^{(k)} \b{p}^{(k)}$\;
        $\b{y}^{(k)} := \nabla f(\b{x}^{(k+1)}) - \nabla f(\b{x}^{(k)})$\;
        $\gamma^{(k+1)} := \frac{\b{s}^{(k)\top} \b{y}^{(k)}}{\b{y}^{(k)\top} \b{y}^{(k)}}$\;
        $\b{H}^{(k+1)} := \gamma^{(k+1)} \b{I}$\;
        Store $\b{y}^{(k)}$, $\b{s}^{(k)}$, and $\b{H}^{(k+1)}$\;
    }
    \textbf{return} $\b{x}^{(k+1)}$\;
    \;
    // recursive function:\;
    \textbf{Function} $\text{GetDirection}(\b{p}, k, n\_{\text{recursion}})$\;
    \uIf{$k > 0$}{
        // do up to $m$ recursions: \\
        \If{$n\_{\text{recursion}} > m$}{
            \textbf{return} $\b{p}$\;  
        }
        $\rho^{(k-1)} := \frac{1}{\b{y}^{(k-1)\top} \b{s}^{(k-1)}}$\;
        $\b{\tilde{p}} := \b{p} - \rho^{(k-1)} (\b{s}^{(k-1)\top} \b{p}) \b{y}^{(k-1)}$\;
        $\b{\widehat{p}} := \text{GetDirection}(\b{\tilde{p}}, k-1, n\_{\text{recursion}}+1)$\;
        \textbf{return} $\b{\widehat{p}} - \rho^{(k-1)} (\b{y}^{(k-1)\top} \b{\widehat{p}}) \b{s}^{(k-1)} + \rho^{(k-1)} (\b{s}^{(k-1)\top} \b{s}^{(k-1)}) \b{p}$\;
    }
    \Else{
        \textbf{return} $\b{H}^{(0)} \b{p}$\;
    }
\caption{The LBFGS algorithm}\label{algorithm_LBFGS}
\end{algorithm2e}
\DecMargin{0.8em}

\hfill\break
\textbf{-- LBFGS: }
The above methods, including BFGS, approximate the inverse Hessian matrix by a dense $(d \times d)$ matrix. For large $d$, storing this matrix is very memory-consuming. Hence, LBFGS \cite{nocedal1980updating,liu1989limited} was proposed which uses much less memory than BFGS. 
In LBFGS, the inverse of Hessian is a scalar multiplied to identity matrix, i.e., $\b{H}^{(0)} := \gamma^{(k)} \b{I}$; therefore, it approximates the $(d \times d)$ matrix with a scalar. It uses a memory of $m$ previous variables and recursively calculates the updating direction of solution. 
In other words, it has recursive unfoldings which approximate the descent directions in optimization. 
The number of recursions is a small integer, for example $m = 10$; hence, not much memory is used. By $m$ times recursion on Eq. (\ref{equation_BFGS_Hessian}), LBFGS approximates the inverse Hessian as {\citep[Algorithm 2.1]{liu1989limited}}:
\begin{align*}
& \b{H}^{(k+1)} := (\b{V}^{(k)\top} \dots \b{V}^{(k-m)\top}) \b{H}^{(0)} (\b{V}^{(k-m)} \dots \b{V}^{(k)}) \\
& + \rho^{(k-m)} (\b{V}^{(k)\top} \dots \b{V}^{(k-m+1)\top}) \b{s}^{(k-m)} \b{s}^{(k-m)\top} \\ 
&~~~~~~~~~~~~~~~~~~~~~~~~~~~~~~~~~~~~~~~~~~~~~~~~~~ (\b{V}^{(k-m+1)} \dots \b{V}^{(k)}) \\
& + \rho^{(k-m+1)} (\b{V}^{(k)\top} \dots \b{V}^{(k-m+2)\top}) \b{s}^{(k-m+1)} \b{s}^{(k-m+1)\top} \\ 
&~~~~~~~~~~~~~~~~~~~~~~ (\b{V}^{(k-m+2)} \dots \b{V}^{(k)}) + \dots + \rho^{(k)} \b{s}^{(k)} \b{s}^{(k)\top}. 
\end{align*}
The LBFGS algorithm can be implemented as shown in Algorithm \ref{algorithm_LBFGS} \cite{hosseini2020alternative} which is based on {\citep[Chapter 6]{nocedal2006numerical}}. As this algorithm shows, every iteration of optimization calls a recursive function for up to $m$ recursions and uses the stored previous $m$ memory to calculate the direction $\b{p}$ for updating the solution. As Eq. (\ref{equation_quasi_Newton_method_step}) states, the initial updating direction is the Newton's method direction, which is $\b{p} = - \b{H}^{(0)} \nabla f(\b{x}^{0})$.

\section{Non-convex Optimization by Sequential Convex Programming} 

Consider the optimization problem (\ref{equation_optimization_problem}) where the functions $f(.)$, $y_i(.)$, and $h_i(.)$ are not necessarily convex. 
The explained methods can also work for non-convex problems but they do not guarantee to find the global optimum. 
They can find local minimizers which depend on the random initial solution. 
For example, the optimization landscape of neural network is highly nonlinear but backpropagation (see Section \ref{section_backpropagation}) works very well for it. The reason for this is explained in this way: every layer of neural network pulls data to the feature space such as in kernels \cite{ghojogh2021reproducing}. In the high-dimensional feature space, all local minimizers are almost global minimizers because the local minimum values are almost equal in that space \cite{feizi2017porcupine}. Also see \cite{soltanolkotabi2018theoretical,allen2019learning,allen2019convergence} to understand why backpropagation optimization works well even in highly non-convex optimization. 
Note that another approach for highly non-convex optimization is metaheuristic optimization which will be briefly introduced in Section \ref{section_metaheuristic_optimization}.

As was explained, the already introduced first-order and second-order optimization methods can work fairly well for non-convex problems by finding local minimizers depending on the initial solution. However, there exist some specific methods for non-convex programming, divided into two categories. The \textit{local optimization methods} are faster but do not guarantee to find the global minimizer. The \textit{global optimization methods} find the global minimizer but are usually slow to find the answer \cite{duchi2018sequential}. \textit{Sequential Convex Programming (SCP)} \cite{dinh2010local} is an example for local optimization methods. It is based on a sequence of convex approximations of the non-convex problem. It is closely related to \textit{Sequential Quadratic Programming (SQP)} \cite{boggs1995sequential} which is used for constrained nonlinear optimization.
\textit{Branch and bound method}, first proposed in \cite{land1960automatic}, is an example for the global optimization methods. It divides the optimization landscape, i.e. the feasible set, into local parts by a binary tree and solves optimization in every part. It checks whether the solution of a part is the global solution or not. 
In the following, we explain SCP which is a faster but local method. 

\subsection{Convex Approximation}

SCP iteratively solves a convex problem where, at every iteration, it approximates the non-convex problem (\ref{equation_optimization_problem}) with a convex problem, based on the current solution, and restricts the variable to be in a so-called \textit{trust region} \cite{conn2000trust}. The trust region makes sure that the variable stays in a locally convex region of the optimization problem. At the iteration $k$ of SCP, we solve the following convex problem:
\begin{equation}\label{equation_optimization_problem_SCP}
\begin{aligned}
& \underset{\b{x}}{\text{minimize}}
& & \widehat{f}(\b{x}) \\
& \text{subject to}
& & \widehat{y}_i(\b{x}) \leq 0, \; i \in \{1, \ldots, m_1\}, \\
& & & \widehat{h}_i(\b{x}) = 0, \; i \in \{1, \ldots, m_2\}, \\
& & & \b{x} \in \mathcal{T}^{(k)},
\end{aligned}
\end{equation}
where $\widehat{f}(.)$, $\widehat{y}_i(.)$, and $\widehat{h}_i(.)$, are convex approximations of functions $f(.)$, $y_i(.)$, and $h_i(.)$, and $\mathcal{T}^{(k)}$ is the trust region at iteration $k$. 
This approximated convex problem is also solved iteratively itself using one of the previously introduced methods such as the interior-point method. 
There exist several approaches for convex approximation of the functions. In the following, we introduce some of these approaches. 

\subsubsection{Convex Approximation by Taylor Series Expansion}

The non-convex functions $f(.)$, $y_i(.)$, and $h_i(.)$ can be approximated by affine functions (i.e., first-order Taylor series expansion) to become convex. For example, the function $f(.)$ is approximated as:
\begin{align}
\widehat{f}(\b{x}) = {f}(\b{x}^{(k)}) + \nabla f(\b{x}^{(k)})^\top (\b{x} - \b{x}^{(k)}).
\end{align}
The functions can also be approximated by quadratic functions (i.e., second-order Taylor series expansion) to become convex. For example, the function $f(.)$ is approximated as:
\begin{align}
\widehat{f}(\b{x}) = {f}(\b{x}^{(k)}) &+ \nabla f(\b{x}^{(k)})^\top (\b{x} - \b{x}^{(k)}) \nonumber\\
&+ \frac{1}{2} (\b{x} - \b{x}^{(k)})^\top \b{P} (\b{x} - \b{x}^{(k)}),
\end{align}
where $\b{P} = \Pi_{\mathbb{S}_{+}^d}(\nabla^2 f(\b{x}^{(k)}))$ is projection of Hessian onto the symmetric positive semi-definite cone. This projection is performed by setting the negative eigenvalues of Hessian to zero. 
The same approaches can be used for approximation of functions $y_i(.)$ and $h_i(.)$ using first- or second-order Taylor expansion. 

\subsubsection{Convex Approximation by Particle Method}

We can approximate the functions $f(.)$, $y_i(.)$, and $h_i(.)$ in the domain of trust region using regression. This approach is named the particle method \cite{duchi2018sequential}. Let $\{\b{x}_i \in \mathcal{T}^{(k)}\}_{i=1}^m$ be $m$ points which lie in the trust region. 
We can use  least-squares quadratic regression to make the functions convex in the trust region:
\begin{equation}
\begin{aligned}
& \underset{a \in \mathbb{R}, \b{b} \in \mathbb{R}^d, \b{P} \in \mathbb{S}_{++}^d}{\text{minimize}}
& & \sum_{i=1}^m \Big(\frac{1}{2} (\b{x}_i - \b{x}^{(k)})^\top \b{P} (\b{x}_i - \b{x}^{(k)}) \\
& & &+ \b{b}^\top (\b{x}_i - \b{x}^{(k)}) + a - f(\b{x}_i)\Big)^2 \\
& \text{subject to}
& & \b{P} \succeq \b{0}.
\end{aligned}
\end{equation}
Then, the function $f(.)$ is replaced by its convex approximation $\widehat{f}(\b{x}) = (1/2) (\b{x}_i - \b{x}^{(k)})^\top \b{P} (\b{x}_i - \b{x}^{(k)}) + \b{b}^\top (\b{x}_i - \b{x}^{(k)}) + a$. The same approach can be used for approximation of functions $y_i(.)$ and $h_i(.)$.

\subsubsection{Convex Approximation by Quasi-linearization}

Another approach for convex approximation of functions $f(.)$, $y_i(.)$, and $h_i(.)$ is quasi-linearization.
We should state the function $f(.)$ in the form $f(\b{x}) = \b{A}(\b{x})\, \b{x} + c(\b{x})$. 
For example, we can use the second-order Taylor series expansion to do this:
\begin{align*}
f(\b{x}) \approx \frac{1}{2} \b{x}^\top \b{P} \b{x} + \b{b}^\top \b{x} + a = (\frac{1}{2} \b{P} \b{x} + \b{b})^\top \b{x} + a,
\end{align*}
so we use $\b{A}(\b{x}) := ((1/2) \b{P} \b{x} + \b{b})^\top$ and $c(\b{x}) := a$ which depend on the Taylor expansion of $f(\b{x})$. 
Hence, the convex approximation of function $f(.)$ can be:
\begin{align}
\widehat{f}(\b{x}) &= \b{A}(\b{x}^{(k)})\, \b{x} + c(\b{x}^{(k)}) = (\frac{1}{2} \b{P} \b{x}^{(k)} + \b{b})^\top \b{x} + a.
\end{align}
The same approach can be used for approximation of functions $y_i(.)$ and $h_i(.)$.

\subsection{Trust Region}

\subsubsection{Formulation of Trust Region}

The trust region can be a box around the point at that iteration:
\begin{align}
\mathcal{T}^{(k)} := \{\b{x}\, |\, |x_j - x_j^{(k)}| \leq \rho_i, \forall j \in \{1, \dots, d\}\}.
\end{align}
where $x_j$ and $x_j^{(k)}$ are the $j$-th element of $\b{x}$ and $\b{x}^{(k)}$, respectively, and $\rho_i$ is the bound of box for the $j$-th dimension. 
Another option for trust region is an ellipse around the point to have a quadratic trust region:
\begin{align}
\mathcal{T}^{(k)} := \{\b{x}\, |\, (\b{x} - \b{x}^{(k)})^\top \b{P}^{-1} (\b{x} - \b{x}^{(k)}) \leq \rho\},
\end{align}
where $\b{P} \in \mathbb{S}_{++}^d$ (is symmetric positive definite) and $\rho > 0$ is the radius of ellipse.

\subsubsection{Updating Trust Region}


The trust region gets updated in every iteration of SCP. 
In the following, we explain how the trust region can be updated. 
First, we embed the constraints in the objective function of problem (\ref{equation_optimization_problem}):
\begin{equation}\label{equation_optimization_problem_exact_penalty_method}
\begin{aligned}
& \underset{\b{x}}{\text{minimize}}
& & \phi(\b{x}) := f(\b{x}) + \lambda \Big(\sum_{i=1}^{m_1} \big(\max(y_i(\b{x}), 0)\big)^2 \\
& & &~~~~~~~~~~~~~~~~~~~~~~~~~~ + \sum_{i=1}^{m_1} |h_i(\b{x})|^2 \Big),
\end{aligned}
\end{equation}
where $\lambda > 0$ is the regularization parameter. 
This is called the \textit{exact penalty method} \cite{di1994exact} because it penalizes violation from the constraints. 
For large enough regularization parameter (which gives importance to violation of constraints), the solution of problem (\ref{equation_optimization_problem_exact_penalty_method}) is exactly equal to the solution of problem (\ref{equation_optimization_problem}). That is the reason for the term ``exact" in the name ``exact penalty method". 
Similar to Eq. (\ref{equation_optimization_problem_exact_penalty_method}), we define:
\begin{align}
\widehat{\phi}(\b{x}) := \widehat{f}(\b{x}) + \lambda \Big(&\sum_{i=1}^{m_1} \big(\max(\widehat{y}_i(\b{x}), 0)\big)^2 \nonumber \\
&+ \sum_{i=1}^{m_1} |\widehat{h}_i(\b{x})|^2 \Big),
\end{align}
for the problem (\ref{equation_optimization_problem_SCP}). 
At the iteration $k$ of SCP, let $\widehat{x}^{(k)}$ be the solution of the convex approximated problem (\ref{equation_optimization_problem_SCP}) using any method such as the interior-point method. 
We calculate the predicted and exact decreases which are $\widehat{\delta} := \phi(\b{x}^{(k)}) - \widehat{\phi}(\widehat{\b{x}})$ and $\delta := \phi(\b{x}^{(k)}) - \phi(\widehat{\b{x}})$, respectively. 
Two cases may happen:
\begin{itemize}\itemsep0em
\item We have progress in optimization if $\alpha \widehat{\delta} \leq \delta$ where $0 < \alpha < 1$ (e.g., $\alpha = 0.1$).
In this case, we accept the approximate solution, i.e. $\b{x}^{(k+1)} := \widehat{\b{x}}$, and we increase the size of trust region, for the next iteration of SCP, by $\rho^{(k+1)} := \beta \rho^{(k)}$ where $\beta \geq 1$ (e.g., $\beta = 1.1$). 
\item We do not have progress in optimization if $\alpha \widehat{\delta} > \delta$.
In this case, we reject the approximate solution, i.e. $\b{x}^{(k+1)} := \b{x}^{(k)}$, and we decrease the size of trust region, for the next iteration of SCP, by $\rho^{(k+1)} := \gamma \rho^{(k)}$ where $0 < \gamma < 1$ (e.g., $\gamma = 0.5$).
\end{itemize}
In summary, the trust region is expanded if we find a good solution; otherwise, it is made smaller. 

\section{Distributed Optimization}

\subsection{Alternating Optimization}

When we have several optimization variables, we can alternate between optimizing over each of these variables. This technique is called \textit{alternating optimization} in the literature \cite{li2019alternating} (also see {\citep[Chapter 4]{jain2017non}}). 
Consider the following multivariate optimization problem:
\begin{equation}
\begin{aligned}
& \underset{\{\b{x}_i\}_{i=1}^m}{\text{minimize}}
& & f(\b{x}_1, \dots, \b{x}_m),
\end{aligned}
\end{equation}
where the objective function depends on $m$ variables. 
\textit{Alternating optimization} alternates between updating every variable while assuming other variables are constant, set to their last updated value. After random feasible initialization, it updates solutions as \cite{li2019alternating}:
\begin{align*}
& \b{x}_1^{(k+1)} := \arg \min_{\b{x}_1} f(\b{x}_1, \b{x}_2^{(k)}, \dots, \b{x}_{m-1}^{(k)}, \b{x}_m^{(k)}), \\
& \b{x}_2^{(k+1)} := \arg \min_{\b{x}_2} f(\b{x}_1^{(k+1)}, \b{x}_2, \dots, \b{x}_{m-1}^{(k)}, \b{x}_m^{(k)}), \\
& \vdots \\
& \b{x}_m^{(k+1)} := \arg \min_{\b{x}_m} f(\b{x}_1^{(k+1)}, \b{x}_2^{(k+1)}, \dots, \b{x}_{m-1}^{(k+1)}, \b{x}_m), 
\end{align*}
until convergence. 
Any optimization methods, including first-order and second-order methods, can be used for each of the optimization lines above. 
In most cases, alternating optimization is robust to changing the order of updates of variables. 

\begin{remark}
If the function $f(\b{x}_1, \dots, \b{x}_m)$ is decomposable in terms of variables, i.e., if we have $f(\b{x}_1, \dots, \b{x}_m) = \sum_{i=1}^m f_i(\b{x}_i)$, the alternating optimization can be simplified to:
\begin{align*}
& \b{x}_1^{(k+1)} := \arg \min_{\b{x}_1} f_1(\b{x}_1), \\
& \b{x}_2^{(k+1)} := \arg \min_{\b{x}_2} f_2(\b{x}_2), \\
& \vdots \\
& \b{x}_m^{(k+1)} := \arg \min_{\b{x}_m} f_m(\b{x}_m),
\end{align*}
because other terms become constant in optimization. 
The above updates mean that if the function is completely decomposable in terms of variables, the updates of variables are independent and can be done independently. Hence, in that case, alternating optimization is reduced to $m$ independent optimization problems, each of which can be solved by any optimization method such as the first-order and second-order methods.
\end{remark}

\textit{Proximal alternating optimization} uses proximal operator, Eq. (\ref{equation_proximal_mapping_scaled}), for minimization to keep the updated solution close to the solution of previous iteration \cite{li2019alternating}:
\begin{align*}
& \b{x}_1^{(k+1)} := \arg \min_{\b{x}_1} \Big( f(\b{x}_1, \b{x}_2^{(k)}, \dots, \b{x}_{m-1}^{(k)}, \b{x}_m^{(k)}) \\
&~~~~~~~~~~~~~~~~~~~~~~~~~~~~~~ + \frac{1}{2\lambda} \|\b{x}_1 - \b{x}_1^{(k)}\|_2^2 \Big), \\
& \b{x}_2^{(k+1)} := \arg \min_{\b{x}_2} \Big( f(\b{x}_1^{(k+1)}, \b{x}_2, \dots, \b{x}_{m-1}^{(k)}, \b{x}_m^{(k)}) \\
&~~~~~~~~~~~~~~~~~~~~~~~~~~~~~~ + \frac{1}{2\lambda} \|\b{x}_2 - \b{x}_2^{(k)}\|_2^2 \Big), 
\end{align*}
\begin{align*}
& \vdots \\
& \b{x}_m^{(k+1)} := \arg \min_{\b{x}_m} \Big( f(\b{x}_1^{(k+1)}, \b{x}_2^{(k+1)}, \dots, \b{x}_{m-1}^{(k+1)}, \b{x}_m) \\
&~~~~~~~~~~~~~~~~~~~~~~~~~~~~~~ + \frac{1}{2\lambda} \|\b{x}_m - \b{x}_m^{(k)}\|_2^2 \Big).
\end{align*}

The alternating optimization methods can also be used for constrained problems:
\begin{equation}
\begin{aligned}
& \underset{\{\b{x}_i\}_{i=1}^m}{\text{minimize}}
& & f(\b{x}_1, \dots, \b{x}_m) \\
& \text{subject to}
& & \b{x}_i \in \mathcal{S}_i,\quad \forall i \in \{1, \dots, m\}.
\end{aligned}
\end{equation}
In this case, every line of the optimization is a constrained problem:
\begin{align*}
& \b{x}_1^{(k+1)} := \arg \min_{\b{x}_1} \Big( f(\b{x}_1, \b{x}_2^{(k)}, \dots, \b{x}_{m-1}^{(k)}, \b{x}_m^{(k)}), \\
&~~~~~~~~~~~~~~~~~~~~~~~~~~~~~~~~~~ \text{s.t. }\,\, \b{x}_1 \in \mathcal{S}_1 \Big), \\
& \b{x}_2^{(k+1)} := \arg \min_{\b{x}_2} \Big( f(\b{x}_1^{(k+1)}, \b{x}_2, \dots, \b{x}_{m-1}^{(k)}, \b{x}_m^{(k)}), \\
&~~~~~~~~~~~~~~~~~~~~~~~~~~~~~~~~~~ \text{s.t. }\,\, \b{x}_2 \in \mathcal{S}_2 \Big), \\
& \vdots \\
& \b{x}_m^{(k+1)} := \arg \min_{\b{x}_m} \Big( f(\b{x}_1^{(k+1)}, \b{x}_2^{(k+1)}, \dots, \b{x}_{m-1}^{(k+1)}, \b{x}_m), \\
&~~~~~~~~~~~~~~~~~~~~~~~~~~~~~~~~~~ \text{s.t. }\,\, \b{x}_m \in \mathcal{S}_m \Big).
\end{align*}
Any constrained optimization methods can be used for each of the optimization lines above. 
Some examples are projected gradient method, proximal methods, interior-point methods, etc. 

Finally, it is noteworthy that practical experiments have shown there is usually no need to use a complete optimization until convergence for every step in the alternating optimization, either unconstrained or constrained. Often, a single step of updating, such as a step of gradient descent or projected gradient method, is enough for the whole algorithm to work. 

\subsection{Dual Ascent and Dual Decomposition Methods}

Consider the following problem:
\begin{equation}\label{equation_dual_ascent_method_optimization_problem}
\begin{aligned}
& \underset{\b{x}}{\text{minimize}}
& & f(\b{x}) \\
& \text{subject to}
& & \b{A} \b{x} = \b{b}.
\end{aligned}
\end{equation}

We follow the method of multipliers as discussed in Section \ref{section_method_of_multipliers}.
The Lagrangian is:
\begin{align*}
\mathcal{L}(\b{x}, \b{\nu}) = f(\b{x}) + \b{\nu}^\top (\b{A} \b{x} - \b{b}).
\end{align*}
The dual function is:
\begin{align}\label{equation_dual_ascent_method_g}
g(\b{\nu}) = \inf_{\b{x}} \mathcal{L}(\b{x}, \b{\nu}).
\end{align}
The dual problem maximizes $g(\b{\nu})$:
\begin{align}\label{equation_dual_ascent_method_nu}
\b{\nu}^* = \arg \max_{\b{\nu}} g(\b{\nu}),
\end{align}
so the optimal primal variable is:
\begin{align}\label{equation_dual_ascent_method_x_star}
\b{x}^* = \arg \min_{\b{x}} \mathcal{L}(\b{x}, \b{\nu}^*).
\end{align}
For solving Eq. (\ref{equation_dual_ascent_method_nu}), we should take the derivative of the dual function w.r.t. the dual variable:
\begin{align*}
\nabla_{\b{\nu}} g(\b{\nu}) &\overset{(\ref{equation_dual_ascent_method_g})}{=} \nabla_{\b{\nu}}(\inf_{\b{x}} \mathcal{L}(\b{x}, \b{\nu})) \\
&\overset{(\ref{equation_dual_ascent_method_x_star})}{=} \nabla_{\b{\nu}}(f(\b{x}^*) + \b{\nu}^\top (\b{A} \b{x}^* - \b{b})) = \b{A} \b{x}^* - \b{b}.
\end{align*}
The dual problem is a maximization problem so we can use gradient ascent (see Section \ref{section_gradient_descent}) for iteratively updating the dual variable with this gradient. We can alternate between updating the optimal primal and dual variables:
\begin{align}
& \b{x}^{(k+1)} := \arg\min_{\b{x}} \mathcal{L}(\b{x}, \b{\nu}^{(k)}), \label{equation_dual_ascent_method_x_update} \\
& \b{\nu}^{(k+1)} := \b{\nu}^{(k)} + \eta^{(k)} (\b{A} \b{x}^{(k+1)} - \b{b}), \label{equation_dual_ascent_method_nu_update}
\end{align}
where $k$ is the iteration index and $\eta^{(k)}$ is the step size (also called the learning rate) at iteration $k$.
Eq. (\ref{equation_dual_ascent_method_x_update}) can be performed by any optimization method. We compute the gradient of $\mathcal{L}(\b{x}, \b{\nu}^{(k)})$ w.r.t. $\b{x}$. If setting this gradient to zero does not give $\b{x}$ in closed form, we can use gradient descent (see Section \ref{section_gradient_descent}) to perform Eq. (\ref{equation_dual_ascent_method_x_update}). Some papers approximate Eq. (\ref{equation_dual_ascent_method_x_update}) by one step or few steps of gradient descent rather than a complete gradient descent until convergence. If using one step, we can write Eq. (\ref{equation_dual_ascent_method_x_update}) as:
\begin{align}
\b{x}^{(k+1)} := \b{x}^{(k)} - \gamma \nabla_{\b{x}} \mathcal{L}(\b{x}, \b{\nu}^{(k)}),
\end{align}
where $\gamma>0$ is the step size.
It has been shown empirically that even one step of gradient descent for Eq. (\ref{equation_dual_ascent_method_x_update}) works properly for the whole alternating algorithm. 

We continue the iterations until convergence of the primal and dual variables to stable values. When we get closer to convergence, we will have $(\b{A} \b{x}^{k+1} - \b{b}) \rightarrow 0$ so that we will not have update of dual variable according to Eq. (\ref{equation_dual_ascent_method_nu_update}). This means that after convergence, we have $(\b{A} \b{x}^{k+1} - \b{b}) \approx 0$ so that the constraint in Eq. (\ref{equation_dual_ascent_method_optimization_problem}) is getting satisfied. 
In other words, the update of dual variable in Eq. (\ref{equation_dual_ascent_method_nu_update}) is taking care of satisfying the constraint. 
This method is known as the \textit{dual ascent} method because it uses gradient ascent for updating the dual variable. 

If the objective function can be distributed and decomposed on $b$ blocks $\{\b{x}_i\}_{i=1}^b$, i.e.:
\begin{align*}
f(\b{x}) = f_1(\b{x}_1) + \dots + f_1(\b{x}_b),
\end{align*}
we can have $b$ Lagrangian functions where the total Lagrangian is the summation of these functions:
\begin{align*}
&\mathcal{L}_i(\b{x}_i, \b{\nu}) = f(\b{x}_i) + \b{\nu}^\top (\b{A} \b{x}_i - \b{b}), \\
&\mathcal{L}(\b{x}_i, \b{\nu}) = \sum_{i=1}^b \big( f(\b{x}_i) + \b{\nu}^\top (\b{A} \b{x}_i - \b{b}) \big).
\end{align*}
We can divide the Eq. (\ref{equation_dual_ascent_method_x_update}) into $b$ updates, each for one of the blocks. 
\begin{align}
& \b{x}_i^{(k+1)} := \arg\min_{\b{x}_i} \mathcal{L}(\b{x}, \b{\nu}^{(k)}), \quad \forall i \in \{1, \dots, b\}, \\
& \b{\nu}^{(k+1)} := \b{\nu}^{(k)} + \eta^{(k)} (\b{A} \b{x}^{(k+1)} - \b{b}).
\end{align}
This is called \textit{dual decomposition} developed by decomposition techniques such as the Dantzig-Wolfe decomposition \cite{dantzig1960decomposition}, Bender's decomposition \cite{benders1962partitioning}, and Lagrangian decomposition \cite{everett1963generalized}. 
The dual decomposition methods can divide a problem into sub-problems and solve them in parallel. Hence, its can be used for big data but they are usually slow to converge. 

\subsection{Augmented Lagrangian Method (Method of Multipliers)}

Assume we regularize the objective function in Eq. (\ref{equation_dual_ascent_method_optimization_problem}) by a penalty on not satisfying the constraint:
\begin{equation}\label{equation_dual_ascent_method_optimization_problem_regularized}
\begin{aligned}
& \underset{\b{x}}{\text{minimize}}
& & f(\b{x}) + \frac{\rho}{2} \|\b{A} \b{x} - \b{b}\|_2^2 \\
& \text{subject to}
& & \b{A} \b{x} = \b{b},
\end{aligned}
\end{equation}
where $\rho > 0$ is the regularization parameter. 

\begin{definition}[Augmented Lagrangian \cite{hestenes1969multiplier,powell1969method}]
The Lagrangian for problem (\ref{equation_dual_ascent_method_optimization_problem_regularized}) is:
\begin{align}
\mathcal{L}_\rho(\b{x}, \b{\nu}) := f(\b{x}) + \b{\nu}^\top (\b{A} \b{x} - \b{b}) + \frac{\rho}{2} \|\b{A} \b{x} - \b{b}\|_2^2.
\end{align}
This Lagrangian is called the \textit{augmented Lagrangian} for problem (\ref{equation_dual_ascent_method_optimization_problem}). 
\end{definition}

We can use this augmented Lagrangian in Eqs. (\ref{equation_dual_ascent_method_x_update}) and (\ref{equation_dual_ascent_method_nu_update}):
\begin{align}
& \b{x}^{(k+1)} := \arg\min_{\b{x}} \mathcal{L}_\rho(\b{x}, \b{\nu}^{(k)}), \label{equation_methodOfMultipliers_x_update} \\
& \b{\nu}^{(k+1)} := \b{\nu}^{(k)} + \rho (\b{A} \b{x}^{(k+1)} - \b{b}), \label{equation_methodOfMultipliers_nu_update}
\end{align}
where we use $\rho$ for the step size of updating the dual variable. 
This method is called the \textit{augmented Lagrangian method} or the \textit{method of multipliers} \cite{hestenes1969multiplier,powell1969method,bertsekas1982method}. 

\subsection{Alternating Direction Method of Multipliers (ADMM)}

ADMM \cite{gabay1976dual,glowinski1976finite,boyd2011distributed} has been used in many recent machine learning and signal processing papers. The usefulness and goal for using ADMM (and other distributed methods) are two-fold: (1) it makes the problem distributed and parallelizable on several servers and (2) it makes it possible to solve an optimization problem with multiple variables. 

\subsubsection{ADMM Algorithm}

Consider the following problem:
\begin{equation}\label{equation_ADMM_optimization_problem}
\begin{aligned}
& \underset{\b{x}_1, \b{x}_2}{\text{minimize}}
& & f_1(\b{x}_1) + f_2(\b{x}_2) \\
& \text{subject to}
& & \b{A} \b{x}_1 + \b{B} \b{x}_2 = \b{c},
\end{aligned}
\end{equation}
which is an optimization over two variables $\b{x}_1$ and $\b{x}_2$. 
The augmented Lagrangian for this problem is:
\begin{equation}\label{equation_ADMM_augmented_Lagrangian}
\begin{aligned}
&\mathcal{L}_\rho(\b{x}_1, \b{x}_2, \b{\nu}) = f_1(\b{x}_1) + f_2(\b{x}_2) \\
&~~~ + \b{\nu}^\top (\b{A} \b{x}_1 + \b{B} \b{x}_2 - \b{c}) + \frac{\rho}{2} \|\b{A} \b{x}_1 + \b{B} \b{x}_2 - \b{c}\|_2^2.
\end{aligned}
\end{equation}
We can alternate between updating the primal variables $\b{x}_1$ and $\b{x}_2$ and the dual variable $\b{\nu}$ until convergence of these variables:
\begin{align}
& \b{x}_1^{(k+1)} := \arg\min_{\b{x}_1} \mathcal{L}_\rho(\b{x}_1, \b{x}_2^{(k)}, \b{\nu}^{(k)}), \label{equation_ADMM_x1_update} \\
& \b{x}_2^{(k+1)} := \arg\min_{\b{x}_2} \mathcal{L}_\rho(\b{x}_1^{(k+1)}, \b{x}_2, \b{\nu}^{(k)}), \label{equation_ADMM_x2_update} \\
& \b{\nu}^{(k+1)} := \b{\nu}^{(k)} + \rho (\b{A} \b{x}_1^{(k+1)} + \b{B} \b{x}_2^{(k+1)} - \b{c}). \label{equation_ADMM_nu_update}
\end{align}
Note that the order of updating primal and dual variables is important and the dual variable should be updated after the primal variables but the order of updating primal variables is not important. 
This method is called the Alternating Direction Method of Multipliers (ADMM) \cite{gabay1976dual,glowinski1976finite}. 
A good survey on ADMM is \cite{boyd2011distributed}.

As was explained before, Eqs. (\ref{equation_ADMM_x1_update}) and (\ref{equation_ADMM_x2_update}) can be performed by any optimization method such as calculating the gradient of augmented Lagrangian w.r.t. $\b{x}_1$ and $\b{x}_2$, respectively, and using a few (or even one) iterations of gradient descent for each of these equations. 

\subsubsection{Simplifying Equations in ADMM}

The last term in the augmented Lagrangian, Eq. (\ref{equation_ADMM_augmented_Lagrangian}), can be restated as:
\begin{align*}
&\b{\nu}^\top (\b{A} \b{x}_1 + \b{B} \b{x}_2 - \b{c}) + \frac{\rho}{2} \|\b{A} \b{x}_1 + \b{B} \b{x}_2 - \b{c}\|_2^2 \\
&= \b{\nu}^\top (\b{A} \b{x}_1 + \b{B} \b{x}_2 - \b{c}) + \frac{\rho}{2} \|\b{A} \b{x}_1 + \b{B} \b{x}_2 - \b{c}\|_2^2 \\
&~~~~ + \frac{1}{2\rho} \|\b{\nu}\|_2^2 - \frac{1}{2\rho} \|\b{\nu}\|_2^2 
\end{align*}
\begin{align*}
&= \frac{\rho}{2} \Big(\|\b{A} \b{x}_1 + \b{B} \b{x}_2 - \b{c}\|_2^2 + \frac{1}{\rho^2} \|\b{\nu}\|_2^2 \\
&~~~~ + \frac{2}{\rho} \b{\nu}^\top (\b{A} \b{x}_1 + \b{B} \b{x}_2 - \b{c}) \Big) - \frac{1}{2\rho} \|\b{\nu}\|_2^2 \\
&\overset{(a)}{=} \frac{\rho}{2} \big\|\b{A} \b{x}_1 + \b{B} \b{x}_2 - \b{c} + \frac{1}{\rho} \b{\nu}\big\|_2^2 - \frac{1}{2\rho} \|\b{\nu}\|_2^2 \\
&\overset{(b)}{=} \frac{\rho}{2} \big\|\b{A} \b{x}_1 + \b{B} \b{x}_2 - \b{c} + \b{u}\big\|_2^2 - \frac{1}{2\rho} \|\b{\nu}\|_2^2.
\end{align*}
where $(a)$ is because of the square of summation of two terms and $(b)$ is because we define $\b{u} := (1/\rho) \b{\nu}$. 
The last term $- (1/(2\rho)) \|\b{\nu}\|_2^2$ is constant w.r.t. to the primal variables $\b{x}_1$ and $\b{x}_2$ so we can drop that term fro Lagrangian when updating the primal variables. 
Hence, the Lagrangian can be restated as:
\begin{equation}\label{equation_ADMM_augmented_Lagrangian_simplified}
\begin{aligned}
&\mathcal{L}_\rho(\b{x}_1, \b{x}_2, \b{u}) = f_1(\b{x}_1) + f_2(\b{x}_2) \\
&~~~~~~~~~ + \frac{\rho}{2} \big\|\b{A} \b{x}_1 + \b{B} \b{x}_2 - \b{c} + \b{u}\big\|_2^2 + \text{constant}.
\end{aligned}
\end{equation}
For updating $\b{x}_1$ and $\b{x}_2$, the terms $f_2(\b{x}_2)$ and $f(\b{x}_1)$ are constant, respectively, and can be dropped (because here $\arg\min$ is important and not the minimum value). 
Hence, Eqs. (\ref{equation_ADMM_x1_update}), (\ref{equation_ADMM_x2_update}), and (\ref{equation_ADMM_nu_update}) can be restated as:
\begin{align}
& \b{x}_1^{(k+1)} := \arg\min_{\b{x}_1} \Big(f_1(\b{x}_1) \nonumber \\
&~~~~~~~~~~~~~~~~~~ + \frac{\rho}{2} \big\|\b{A} \b{x}_1 + \b{B} \b{x}_2^{(k)} - \b{c} + \b{u}^{(k)}\big\|_2^2 \Big), \label{equation_ADMM_x1_update2} \\
& \b{x}_2^{(k+1)} := \arg\min_{\b{x}_2} \Big(f_2(\b{x}_2) \nonumber \\
&~~~~~~~~~~~~~~~~~~ + \frac{\rho}{2} \big\|\b{A} \b{x}_1^{(k+1)} + \b{B} \b{x}_2 - \b{c} + \b{u}^{(k)}\big\|_2^2 \Big), \label{equation_ADMM_x2_update2} \\
& \b{u}^{(k+1)} := \b{u}^{(k)} + \rho (\b{A} \b{x}_1^{(k+1)} + \b{B} \b{x}_2^{(k+1)} - \b{c}). \label{equation_ADMM_nu_update2}
\end{align}
Again, Eqs. (\ref{equation_ADMM_x1_update2}) and (\ref{equation_ADMM_x2_update2}) can be performed by one or few steps of gradient descent or any other optimization method. The convergence of ADMM for non-convex and non-smooth functions has been analyzed in \cite{wang2019global}.

\subsection{ADMM Algorithm for General Optimization Problems and Any Number of Variables}

\subsubsection{Distributed Optimization}

ADMM can be extended to several equality and inequality constraints for several optimization variables \cite{giesen2016distributed,giesen2019combining}. 
Consider the following optimization problem with $m$ optimization variables and an equality and inequality constraint for every variable:
\begin{equation}\label{equation_ADMM_optimization_problem_general}
\begin{aligned}
& \underset{\{\b{x}_i\}_{i=1}^m}{\text{minimize}}
& & \sum_{i=1}^m f_i(\b{x}_i) \\
& \text{subject to}
& & y_i(\b{x}_i) \leq 0, \; i \in \{1, \ldots, m\}, \\
& & & h_i(\b{x}_i) = 0, \; i \in \{1, \ldots, m\}.
\end{aligned}
\end{equation}
We can convert every inequality constraint to equality constraints by this technique \cite{giesen2016distributed,giesen2019combining}:
\begin{align*}
y_i(\b{x}_i) \leq 0 \quad \equiv \quad y'_i(\b{x}_i) := \big(\!\max(0, y_i(\b{x}_i))\big)^2 = 0.
\end{align*}
Hence, the problem becomes:
\begin{equation*}
\begin{aligned}
& \underset{\{\b{x}_i\}_{i=1}^m}{\text{minimize}}
& & \sum_{i=1}^m f_i(\b{x}_i) \\
& \text{subject to}
& & y'_i(\b{x}_i) = 0, \; i \in \{1, \ldots, m\}, \\
& & & h_i(\b{x}_i) = 0, \; i \in \{1, \ldots, m\}.
\end{aligned}
\end{equation*}
Having dual variables $\b{\lambda} = [\lambda_1, \dots, \lambda_m]^\top$ and $\b{\nu} = [\nu_1, \dots, \nu_m]^\top$ and regularization parameter $\rho>0$, the augmented Lagrangian for this problem is:
\begin{equation}\label{equation_ADMM_augmented_Lagrangian_general}
\begin{aligned}
&\mathcal{L}_\rho(\{\b{x}_i\}_{i=1}^m, \b{\nu}', \b{\nu}) = \sum_{i=1}^m f_i(\b{x}_i) \\
&~~~~~~~~~~~+ \sum_{i=1}^{m} \lambda_i y'_i(\b{x}_i) + \sum_{i=1}^{m} \nu_i h_i(\b{x}_i) \\
&~~~~~~~~~~~+ \frac{\eta}{2} \sum_{i=1}^{m} (y'_i(\b{x}_i))^2 + \frac{\rho}{2} \sum_{i=1}^{m} (h_i(\b{x}_i))^2  \\
&= \sum_{i=1}^m f_i(\b{x}_i) + \b{\lambda}^\top \b{y}'(\b{x}) + \b{\nu}^\top \b{h}(\b{x}) \\
&~~~~+ \frac{\rho}{2} \|\b{y}'(\b{x})\|_2^2 + \frac{\rho}{2} \|\b{h}(\b{x})\|_2^2, 
\end{aligned}
\end{equation}
where $\mathbb{R}^m \ni \b{y}'(\b{x}) := [y'_1(\b{x}_1), \dots, y'_m(\b{x}_m)]^\top$ and $\mathbb{R}^m \ni \b{h}(\b{x}) := [h_1(\b{x}_1), \dots, h_m(\b{x}_m)]^\top$.
Updating the primal and dual variables are performed as \cite{giesen2016distributed,giesen2019combining}::
\begin{align*}
& \b{x}_i^{(k+1)} := \arg\min_{\b{x}_i} \mathcal{L}_\rho(\b{x}_i, \lambda_i^{(k)}, \nu_i^{(k)}),\, \forall i \in \{1, \dots, m\}, \\
& \b{\lambda}^{(k+1)} := \b{\lambda}^{(k)} + \rho\, \b{y}'(\b{x}^{(k+1)}), \\
& \b{\nu}^{(k+1)} := \b{\nu}^{(k)} + \rho\, \b{h}(\b{x}^{(k+1)}).
\end{align*}
Note that as the Lagrangian is completely decomposable by the $i$ indices, the optimization for every $i$-th primal or dual variable does not depend on other indices; in other words, the terms of other indices become constant for every index.
The last terms in the augmented Lagrangian, Eq. (\ref{equation_ADMM_augmented_Lagrangian_general}), can be restated as:
\begin{align*}
&\b{\lambda}^\top \b{y}'(\b{x}) + \b{\nu}^\top \b{h}(\b{x}) + \frac{\rho}{2} \|\b{y}'(\b{x})\|_2^2 + \frac{\rho}{2} \|\b{h}(\b{x})\|_2^2 \\
&=\b{\lambda}^\top \b{y}'(\b{x}) + \frac{\rho}{2} \|\b{y}'(\b{x})\|_2^2 + \frac{1}{2\rho} \|\b{\lambda}\|_2^2 - \frac{1}{2\rho} \|\b{\lambda}\|_2^2 \\
&+ \b{\nu}^\top \b{h}(\b{x}) + \frac{\rho}{2} \|\b{h}(\b{x})\|_2^2 + \frac{1}{2\rho} \|\b{\nu}\|_2^2 - \frac{1}{2\rho} \|\b{\nu}\|_2^2 \\
&\!\!= \frac{\rho}{2} \Big(\|\b{y}'(\b{x})\|_2^2 + \frac{1}{\rho^2} \|\b{\lambda}\|_2^2 + \frac{2}{\rho} \b{\lambda}^\top \b{y}'(\b{x}) \Big) - \frac{1}{2\rho} \|\b{\lambda}\|_2^2 \\
&+ \frac{\rho}{2} \Big(\|\b{h}(\b{x})\|_2^2 + \frac{1}{\rho^2} \|\b{\nu}\|_2^2 + \frac{2}{\rho} \b{\nu}^\top \b{h}(\b{x}) \Big) - \frac{1}{2\rho} \|\b{\nu}\|_2^2 \\
&= \frac{\rho}{2} \big\|\b{y}'(\b{x}) + \frac{1}{\rho} \b{\lambda}\big\|_2^2 - \frac{1}{2\rho} \|\b{\lambda}\|_2^2 \\
&+ \frac{\rho}{2} \big\|\b{h}(\b{x}) + \frac{1}{\rho} \b{\nu}\big\|_2^2 - \frac{1}{2\rho} \|\b{\nu}\|_2^2 \\
&\overset{(a)}{=} \frac{\rho}{2} \big\|\b{y}'(\b{x}) + \b{u}_\lambda\big\|_2^2 + \frac{\rho}{2} \big\|\b{h}(\b{x}) + \b{u}_\nu\big\|_2^2 - \text{constant},
\end{align*}
where $(a)$ is because we define $\b{u}_\lambda := (1/\rho) \b{\lambda}$ and $\b{u}_\nu := (1/\rho) \b{\nu}$. 
Hence, the Lagrangian can be restated as:
\begin{align*}
&\mathcal{L}_\rho(\{\b{x}_i\}_{i=1}^m, \b{u}_\lambda, \b{u}_\nu) = \sum_{i=1}^m f_i(\b{x}_i) + \\
&~~~~+ \frac{\rho}{2} \big\|\b{y}'(\b{x}) + \b{u}_\lambda\big\|_2^2 + \frac{\rho}{2} \big\|\b{h}(\b{x}) + \b{u}_\nu\big\|_2^2 + \text{constant} \\
&= \sum_{i=1}^m f_i(\b{x}_i) + \frac{\rho}{2} \sum_{i=1}^m \big[(y'_i(\b{x}_i) + u_{\lambda,i})^2 \\
&~~~~~~~~~~~~~~~~~~~~~~~~~~~+ (h_i(\b{x}_i) + u_{\nu,i})^2\big] + \text{constant},
\end{align*}
where $u_{\lambda,i} = (1/\rho) \lambda_i$ and $u_{\nu,i} = (1/\rho) \nu_i$ are the $i$-th elements of $\b{u}_\lambda$ and $\b{u}_\nu$, respectively. 
Hence, updating variables can be restated as:
\begin{align}
& \b{x}_i^{(k+1)} := \arg\min_{\b{x}_i} \Big(f_i(\b{x}_i) + \frac{\rho}{2} \big[(y'_i(\b{x}_i) + u_{\lambda,i}^{(k)})^2 \nonumber \\
&~~~~~~~~~~~~~~~+ (h_i(\b{x}_i) + u_{\nu,i}^{(k)})^2\big] \Big), \forall i \in \{1, \dots, m\}, \label{equation_ADMM_x_i_update2_general} \\
& u_{\lambda,i}^{(k+1)} := u_{\lambda,i}^{(k)} + \rho\, y'_i(\b{x}_i^{(k+1)}),\, \forall i \in \{1, \dots, m\} \label{equation_ADMM_u_lambda_update2_general} \\
& u_{\nu,i}^{(k+1)} := u_{\nu,i}^{(k)} + \rho\, h_i(\b{x}_i^{(k+1)}),\, \forall i \in \{1, \dots, m\}. \label{equation_ADMM_u_nu_update2_general}
\end{align}


\hfill\break
\textbf{-- Use of ADMM for Distributed Optimization:}
ADMM is one of the most well-known algorithms for distributed optimization. If the problem can be divided into several disjoint blocks (i.e., several primal variables), we can solve the optimization for each primal variable on a separate core or server (see Eq. (\ref{equation_ADMM_x_i_update2_general}) for every $i$). Hence, in every iteration of ADMM, the update of primal variables can be performed in parallel by distributed servers. At the end of each iteration, the updated primal variables are gathered in a central server so that the update of dual variable(s) is performed (see Eqs. (\ref{equation_ADMM_u_lambda_update2_general}) and (\ref{equation_ADMM_u_nu_update2_general})). Then, the updated dual variable(s) is sent to the distributed servers so they update their primal variables. This procedure is repeated until convergence of primal and dual variables. 
In this sense, ADMM is performed similar to the approach of \textit{federated learning} \cite{konevcny2015federated,li2020federated}. 

\subsubsection{Making Optimization Problem Distributed}

We can convert a non-distributed optimization problem to a distributed optimization problem to solve it using ADMM. 
Many recent machine learning and signal processing papers are using this technique. 

\hfill\break
\textbf{-- Univariate optimization problem:}
Consider a regular non-distributed problem with one optimization variable $\b{x}$:
\begin{equation}\label{equation_ADMM_optimization_problem_general_notDistributed}
\begin{aligned}
& \underset{\b{x}}{\text{minimize}}
& & \sum_{i=1}^m f_i(\b{x}) \\
& \text{subject to}
& & y_i(\b{x}) \leq 0, \; i \in \{1, \ldots, m\}, \\
& & & h_i(\b{x}) = 0, \; i \in \{1, \ldots, m\}.
\end{aligned}
\end{equation}
This problem can be stated as:
\begin{equation}\label{equation_ADMM_optimization_problem_general_Distributed}
\begin{aligned}
& \underset{\{\b{x}_i\}_{i=1}^m}{\text{minimize}}
& & \sum_{i=1}^m f_i(\b{x}_i) \\
& \text{subject to}
& & y_i(\b{x}_i) \leq 0, \; i \in \{1, \ldots, m\}, \\
& & & h_i(\b{x}_i) = 0, \; i \in \{1, \ldots, m\}, \\
& & & \b{x}_i = \b{z}, \; i \in \{1, \ldots, m\},
\end{aligned}
\end{equation}
where we introduce $m$ variables $\{\b{x}_i\}_{i=1}^m$ and use the trick $\b{x}_i = \b{z}, \forall i$ to make them equal to one variable. 
Eq. (\ref{equation_ADMM_optimization_problem_general_Distributed}) is similar to Eq. (\ref{equation_ADMM_optimization_problem_general}) except that it has $2m$ equality constraints rather than $m$ equality constraints. Hence, we can use ADMM updates similar to Eqs. (\ref{equation_ADMM_x_i_update2_general}), (\ref{equation_ADMM_u_lambda_update2_general}), and (\ref{equation_ADMM_u_nu_update2_general}) but with slight change because of the additional $m$ constraints. We introduce $m$ new dual variables for constraints $\b{x}_i = \b{z}, \forall i$ and update those dual variables as well as other variables. The augmented Lagrangian also has some additional terms for the new constraints. 
We do not write down the Lagrangian and ADMM updates because of its similarity to the previous equations. This is a good technique to make a problem distributed, use ADMM for solving it, and solving it in parallel servers. 

\hfill\break
\textbf{-- Multivariate optimization problem:}
Consider a regular non-distributed problem with multiple optimization variables $\{\b{x}_i\}_{i=1}^m$:
\begin{equation}\label{equation_ADMM_optimization_problem_general_notDistributed2}
\begin{aligned}
& \underset{\{\b{x}\}_{i=1}^m}{\text{minimize}}
& & \sum_{i=1}^m f_i(\b{x}_i) \\
& \text{subject to}
& & \b{x}_i \in \mathcal{S}_i, \; i \in \{1, \ldots, m\},
\end{aligned}
\end{equation}
where $\b{x}_i \in \mathcal{S}_i$ can be any constraint such as belonging to a set $\mathcal{S}_i$, an equality constraint, or an inequality constraint. 
We can embed the constraint in the objective function using an indicator function:
\begin{equation*}
\begin{aligned}
& \underset{\{\b{x}\}_{i=1}^m}{\text{minimize}}
& & \sum_{i=1}^m \big(f_i(\b{x}_i) + \phi_i(\b{x}_i)\big),
\end{aligned}
\end{equation*}
where $\phi_i(\b{x}_i) := \mathbb{I}(\b{x}_i \in \mathcal{S}_i)$ is zero if $\b{x}_i \in \mathcal{S}_i$ and is infinity otherwise. 
This problem can be stated as:
\begin{equation}\label{equation_ADMM_optimization_problem_general_Distributed2}
\begin{aligned}
& \underset{\{\b{x}_i\}_{i=1}^m}{\text{minimize}}
& & \sum_{i=1}^m \big(f_i(\b{x}_i) + \phi_i(\b{z}_i)\big) \\
& \text{subject to}
& & \b{x}_i = \b{z}_i, \; i \in \{1, \ldots, m\},
\end{aligned}
\end{equation}
where we introduce a variable $\b{z}_i$ for every $\b{x}_i$, use the introduced variable for the second term in the objective function, and we equate them in the constraint. 

As the constraints $\b{x}_i - \b{z}_i = \b{0}, \forall i$ are equality constraints, we can use Eqs. (\ref{equation_ADMM_x1_update2}), (\ref{equation_ADMM_x2_update2}), and (\ref{equation_ADMM_nu_update2}) as ADMM updates for this problem:
\begin{align}
& \b{x}_i^{(k+1)} := \arg\min_{\b{x}_i} \Big(f_i(\b{x}_i) + \frac{\rho}{2} \big\|\b{x}_i - \b{z}_i^{(k)} + \b{u}_i^{(k)}\big\|_2^2 \Big), \nonumber \\
&~~~~~~~~~~~~~~~~~~~~~~~~~~~~~~~~~~~~~~~~~~~~~~~~ \forall i \in \{1, \dots, m\},  \label{equation_ADMM_MadeDistirbuted_multivariate_x_update} \\
& \b{z}_i^{(k+1)} := \arg\min_{\b{z}_i} \Big(\phi_i(\b{z}_i) + \frac{\rho}{2} \big\|\b{x}_i^{(k+1)} - \b{z}_i + \b{u}_i^{(k)}\big\|_2^2 \Big), \nonumber \\
&~~~~~~~~~~~~~~~~~~~~~~~~~~~~~~~~~~~~~~~~~~~~~~~~ \forall i \in \{1, \dots, m\},  \label{equation_ADMM_MadeDistirbuted_multivariate_z_update} \\
& \b{u}_i^{(k+1)} := \b{u}_i^{(k)} + \rho (\b{x}_i^{(k+1)} + \b{z}_i^{(k+1)}),\, \forall i \in \{1, \dots, m\}. \nonumber
\end{align}
Comparing Eqs. (\ref{equation_ADMM_MadeDistirbuted_multivariate_x_update}) and (\ref{equation_ADMM_MadeDistirbuted_multivariate_z_update}) with Eq. (\ref{equation_proximal_mapping_scaled}) shows that these ADMM updates can be written as proximal mappings:
\begin{align}
& \b{x}_i^{(k+1)} := \textbf{prox}_{\frac{1}{\rho} f_i}(\b{z}_i^{(k)} - \b{u}_i^{(k)}),\, \forall i \in \{1, \dots, m\}, \nonumber\\
& \b{z}_i^{(k+1)} := \textbf{prox}_{\frac{1}{\rho} \phi_i}(\b{x}_i^{(k+1)} + \b{u}_i^{(k)}),\, \forall i \in \{1, \dots, m\}, \label{equation_ADMM_MadeDistirbuted_multivariate_z_update_prox}\\
& \b{u}_i^{(k+1)} := \b{u}_i^{(k)} + \rho (\b{x}_i^{(k+1)} + \b{z}_i^{(k+1)}),\, \forall i \in \{1, \dots, m\}, \nonumber
\end{align}
if we notice that $\|\b{x}_i^{(k+1)} - \b{z}_i + \b{u}_i^{(k)}\|_2^2 = \|\b{z}_i - \b{x}_i^{(k+1)} - \b{u}_i^{(k)}\|_2^2$.
Note that in many papers, such as \cite{otero2018alternate}, we only have $m=1$. In that case, we only have two primal variables $\b{x}$ and $\b{z}$. 

According to Lemma \ref{lemma_projection_onto_set}, as the function $\phi_i(.)$ is an indicator function, Eq. (\ref{equation_ADMM_MadeDistirbuted_multivariate_z_update_prox}) can be implemented by projection onto the set $\mathcal{S}_i$:
\begin{align*}
\b{z}_i^{(k+1)} := \Pi_{\mathcal{S}_i}(\b{x}_i^{(k+1)} + \b{u}_i^{(k)}),\, \forall i \in \{1, \dots, m\}.
\end{align*}
As an example, assume the variables are all matrices so we have $\b{X}_i$, $\b{Z}_i$, and $\b{U}_i$.
if the set $\mathcal{S}_i$ is the cone of orthogonal matrices, the constraint $\b{X}_i \in \mathcal{S}_i$ would be $\b{X}_i^\top \b{X}_i = \b{I}$. In this case, the update of matrix variable $\b{Z}_i$ would be done by setting the singular values of $(\b{x}_i^{(k+1)} + \b{u}_i^{(k)})$ to one (see Lemma \ref{lemma_projection_onto_orthogonal_cone}).

\section{Additional Notes}

There exist some other optimization methods which we have not covered in this paper, for brevity. Here, we review some of them. 

\subsection{Cutting-Plane Methods}
Cutting plane methods, also called the localization methods, are a family of methods which start with a large feasible set containing the solution to be found. Then, iteratively they reduce the feasible set by cutting off some piece of it \cite{boyd2007localization}. For example, a cutting-plane method starts with a polyhedron feasible set. It finds a plane at every iteration which divides the feasible sets into two disjoint parts one of which contains the ultimate solution. It gets rid of the part without solution and reduces the volume of the polyhedron. This is repeated until the polyhedron feasible set becomes very small and converges to the solution. This is somewhat a generalization of the \textit{bisection method}, also called the \textit{binary search method}, which was used for root-finding \cite{burden1963numerical} but later it was used for convex optimization. The bisection method halves the feasible set and removes the part without the solution, at every iteration (see Algorithm \ref{algorithm_bisection}). 
Some of the important cutting-plane methods are \textit{center of gravity method}, \textit{Maximum Volume Ellipsoid (MVE) cutting-plane method}, \textit{Chebyshev center cutting-plane method}, and \textit{Analytic Center Cutting-Plane Method (ACCPM)} \cite{goffin1993computation,nesterov1995cutting,atkinson1995cutting}. 
Similar to subgradient methods, cutting-plane methods can be used for optimizing non-smooth functions.

\subsection{Ellipsoid Method}
Ellipsoid method was developed by several people \cite{wolfe1980invited,rebennack2009ellipsoid}.
It was proposed in \cite{shor1977cut,yudin1976informational,yudin1977evaluation,yudin1977optimization} and it was initially applied to liner programming in a famous paper \cite{khachiyan1979polynomial}. It is similar to cutting-plane methods in cutting some part of feasible set iteratively. At every iteration, it finds an ellipsoid centered at the current solution:
\begin{align*}
\mathcal{E}(\b{x}^{(k)}, \b{P}) := \{\b{z}\, |\, (\b{z} - \b{x}^{(k)})^\top \b{P}^{-1} (\b{z} - \b{x}^{(k)}) \leq 1\},
\end{align*}
where $\b{P} \in \mathbb{S}_{++}^d$ (is symmetric positive definite). It removes half of the ellipsoid which does not contain the solution. Again, another ellipsoid is found at the updated solution. This is repeated until the ellipsoid of iteration is very small and converges to the solution.





\subsection{Minimax and Maximin Problems}


Consider a function of two variables, $f(\b{x}_1, \b{x}_2)$, and the following optimization problem:
\begin{equation}\label{equation_optimization_problem_minimax}
\begin{aligned}
& \underset{\b{x}_1}{\text{minimize}}~\Big(\underset{\b{x}_2}{\text{maximize}}~f(\b{x}_1, \b{x}_2)\Big).
\end{aligned}
\end{equation}
In this problem, we want to minimize the function w.r.t. one of the variables and maximize it w.r.t the other variable. This optimization problem is called the \textit{minimax problem}. 
We can change the order of this problem to have the so-called \textit{maximin problem}:
\begin{equation}\label{equation_optimization_problem_maximin}
\begin{aligned}
& \underset{\b{x}_1}{\text{maximize}}~\Big(\underset{\b{x}_2}{\text{minimize}}~f(\b{x}_1, \b{x}_2)\Big).
\end{aligned}
\end{equation}
Note that under certain conditions, the minimax and maximin problems are equivalent if the variables of maximization (or minimization) stay the same. In other words, under some conditions, we have \cite{du2013minimax}:
\begin{align*}
\underset{\b{x}_1}{\text{minimize}}\,(&\underset{\b{x}_2}{\text{maximize}}\,f(\b{x}_1, \b{x}_2)) \\
&= \underset{\b{x}_2}{\text{maximize}}\,(\underset{\b{x}_1}{\text{minimize}}\,f(\b{x}_1, \b{x}_2)).
\end{align*}
In the minimax and maximin problems, the two variables have conflicting or contrastive desires; one of them wants to maximize the function while the other wants to minimize it. Hence, they are widely used in the field of game theory as important strategies \cite{aumann1972some}. 

\SetAlCapSkip{0.5em}
\IncMargin{0.8em}
\begin{algorithm2e}[!t]
\DontPrintSemicolon
    Input: $l$ and $u$\;
    \For{iteration $k = 0, 1, \dots$}{
        $\b{x}^{(k+1)} := \frac{l+u}{2}$\;
        \uIf{$\nabla f(\b{x}) < 0$}{
            $l := \b{x}$\;
        }
        \Else{
            $\b{u} := \b{x}$\;
        }
        Check the convergence criterion\;
        \If{converged}{
            \textbf{return} $\b{x}^{(k+1)}$\;
        }
    }
\caption{The bisection algorithm}\label{algorithm_bisection}
\end{algorithm2e}
\DecMargin{0.8em}

\subsection{Riemannian Optimization}

In this paper, we covered optimization methods in the Euclidean space. The Euclidean optimization methods can be slightly revised to have optimization on (possibly curvy) Riemannian manifolds. 
\textit{Riemannian optimization} \cite{absil2009optimization,boumal2020introduction} optimizes a cost function while the variable lies on a Riemannian manifold $\mathcal{M}$. The optimization variable in the Riemannian optimization is usually matrix rather than vector; hence, Riemannian optimization is also called \textit{optimization on matrix manifolds}. The Riemannian optimization problem is formulated as:
\begin{equation}\label{equation_Riemannian_optimization}
\begin{aligned}
& \underset{\b{X}}{\text{minimize}}
& & f(\b{X}) \\
& \text{subject to}
& & \b{X} \in \mathcal{M}.
\end{aligned}
\end{equation}
Most of the Euclidean first-order and second-order optimization methods have their Riemannian optimization variants by some changes in the formulation of methods. 
In the Riemannian optimization methods, the solution lies on the manifold. At every iteration, the descent direction is calculated in the tangent space on the manifold at the current solution. Then, the updated solution in the tangent space is retracted (i.e., projected) onto the curvy manifold \cite{hosseini2020recent,hu2020brief}. This procedure is repeated until convergence to the final solution. 
Some well-known Riemannian manifolds which are used for optimization are Symmetric Positive Definite (SPD) \cite{sra2015conic}, quotient \cite{lee2013quotient}, Grassmann \cite{bendokat2020grassmann}, and Stiefel \cite{edelman1998geometry} manifolds.

\subsection{Metaheuristic Optimization}\label{section_metaheuristic_optimization}

In this paper, we covered classical optimization methods. 
There are some other optimization methods, called \textit{metaheuristic optimization} \cite{talbi2009metaheuristics}, which are a family of methods finding the optimum of a cost function using efficient, and not brute-force, search. 
They fall in the field of \textit{soft computing} and can be used in highly non-convex optimization with many constraints, where classical optimization is a little difficult and slow to perform. 
Some very well-known categories of the Metaheuristic optimization methods are \textit{nature-inspired optimization} \cite{yang2010nature}, \textit{evolutionary computing} \cite{simon2013evolutionary}, and \textit{particle-based optimization}. 
The general shared idea of metaheuristic methods is as follows. We perform local search by some particles in various parts of the feasible set. Wherever a smaller cost was found, we tend to do more local search in that area; although, we should also keep exploring other areas because that better area might be just a local optimum. Hence, both local and global searches are used for exploitation and exploration of the cost function, respectively. 

There exist many metaheuristic methods.
Two popular and fundamental ones are Genetic Algorithm (GA) \cite{holland1992adaptation} and Particle Swarm Optimization (PSO) \cite{kennedy1995particle}. GA, which is an evolutionary method inspired by natural selection, uses chromosomes as particles where the particles tend to cross-over (or marry) with better particles in terms of smaller cost function. This results in a better next generation of particles. Mutations are also done for global search and exploration. Iterations of making new generations results in very good particles finding the optimum. 

PSO is a nature-inspired method whose particles can be seen as a herd fishes or birds. For better understanding, assume particles are humans digging ground in a vast area to find treasure. They find various things which differ in terms of value. When someone finds a roughly good thing, people tend to move toward that person bu also tend to search more than before, around themselves at the same time. Hence, they move in a combined direction tending a little more toward the better found object. The fact that they still dig around themselves is that the found object may not be the real treasure (i.e., may be a local optimum) so they also search around for the sake of exploration. 

\section{Conclusion}\label{section_conclusion}

This paper was a tutorial and survey paper on KKT conditions, numerical optimization (both first-order and second-order methods), and distributed optimization. We covered various optimization algorithms in this paper. Reading it can be useful for different people in different fields of science and engineering. We did not assume much on the background of reader and explained methods in detail. 

\section*{Acknowledgement}

The authors hugely thank Prof. Stephen Boyd for his great courses Convex Optimization 1 and 2 of Stanford University available on YouTube (The course Convex Optimization 1 mostly focuses on second-order and interior-point methods and the course Convex Optimization 2 focuses on more advanced non-convex optimization, non-smooth optimization, ellipsoid method, distributed optimization, proximal algorithms, and some first-order methods). 
They also thank Prof. Kimon Fountoulakis, Prof. James Geelen, Prof. Oleg Michailovich, Prof. Massoud Babaie-Zadeh, Prof. Lieven Vandenberghe, Prof. Mark Schmidt, Prof. Ryan Tibshirani, Prof. Reshad Hosseini, Prof. Saeed Sharifian, and some other professors whose lectures partly covered some materials and proofs mentioned in this tutorial paper. 
The great books of Prof. Yurii Nesterov \cite{nesterov1998introductory,nesterov2003introductory} were also influential on some of proofs.

\appendix

\section{Proofs for Section \ref{section_preliminaries}}

\subsection{Proof for Lemma \ref{lemma_fundamental_theorem_calculus_corollary}}\label{app_fundamental_theorem_calculus_corollary}

\begin{align*}
&f(\b{y}) \overset{(\ref{equation_fundamental_theorem_calculus})}{=} f(\b{x}) + \nabla f(\b{x})^\top (\b{y} - \b{x}) \\
&~~~~~~+ \int_0^1 \Big(\nabla f\big(\b{x} + t(\b{y} - \b{x})\big) - \nabla f(\b{x}) \Big)^\top (\b{y} - \b{x}) dt \\
&\overset{(a)}{\leq} f(\b{x}) + \nabla f(\b{x})^\top (\b{y} - \b{x}) \\
&~~~+ \int_0^1 \|\nabla f\big(\b{x} + t(\b{y} - \b{x})\big) - \nabla f(\b{x}) \|_2 \|\b{y} - \b{x}\|_2 dt \\
&\overset{(b)}{\leq} f(\b{x}) + \nabla f(\b{x})^\top (\b{y} - \b{x}) + \int_0^1 L t \|\b{y} - \b{x}\|_2^2 dt \\
&= f(\b{x}) + \nabla f(\b{x})^\top (\b{y} - \b{x}) + L \|\b{y} - \b{x}\|_2^2 \int_0^1 t dt \\
&= f(\b{x}) + \nabla f(\b{x})^\top (\b{y} - \b{x}) + \frac{L}{2} \|\b{y} - \b{x}\|_2^2,
\end{align*}
where $(a)$ is because of the Cauchy-Schwarz inequality and $(b)$ is because, according to Eq. (\ref{equation_gradient_L_smooth}), we have $\|\nabla f(\b{x} + t(\b{y} - \b{x})) - \nabla f(\b{x}) \|_2 = L \|\b{x} + t(\b{y} - \b{x}) - \b{x}\|_2 = L t \|\b{y} - \b{x}\|_2$. Q.E.D.

\subsection{Proof for Lemma \ref{lemma_function_and_gradient_difference_bounds}}\label{app_function_and_gradient_difference_bounds}

\textbf{-- Proof for the first equation:}
As $f(.)$ is convex, according to Eq. (\ref{equation_convex_function_firstDerivative}), 
\begin{align}
&f(\b{z}) - f(\b{y}) \geq \nabla f(\b{y})^\top (\b{z} - \b{y}) \nonumber \\
&\implies f(\b{y}) - f(\b{z}) \leq \nabla f(\b{y})^\top (\b{y} - \b{z}). \label{equation_fz_fy_gradf_z_y}
\end{align}
\begin{align}\label{equation_fy_fz_fz_fx}
&f(\b{y}) - f(\b{x}) = \big(f(\b{y}) - f(\b{z})\big) + \big(f(\b{z}) - f(\b{x})\big).
\end{align}
Also, according to Eq. (\ref{equation_fundamental_theorem_calculus_Lipschitz}), we have:
\begin{align}\label{equation_fz_fx_gradf_z_x_L_z_x}
&f(\b{z}) - f(\b{x}) \leq \nabla f(\b{x})^\top (\b{z} - \b{x}) + \frac{L}{2} \|\b{z} - \b{x}\|_2^2.
\end{align}
Using Eqs. (\ref{equation_fz_fy_gradf_z_y}) and (\ref{equation_fz_fx_gradf_z_x_L_z_x}) in Eq. (\ref{equation_fy_fz_fz_fx}) gives:
\begin{align*}
f(\b{y}) - f(\b{x}) \leq &\,\nabla f(\b{y})^\top (\b{y} - \b{z}) + \nabla f(\b{x})^\top (\b{z} - \b{x}) \\
&+ \frac{L}{2} \|\b{z} - \b{x}\|_2^2.
\end{align*}
For this, we can minimize this upper-bound (the right-hand side) by setting its derivative w.r.t. $\b{z}$ to zero. It gives:
\begin{align*}
\b{z} = \b{x} - \frac{1}{L} \big(\nabla f(\b{x}) - \nabla f(\b{y})\big).
\end{align*}
Putting this in the upper-bound gives:
\begin{align*}
&f(\b{y}) - f(\b{x}) \leq \nabla f(\b{y})^\top (\b{y} - \b{x}) \\
&+ \frac{1}{L} \nabla f(\b{y})^\top (\nabla f(\b{x}) - \nabla f(\b{y})) \\
&- \frac{1}{L} \nabla f(\b{x})^\top (\nabla f(\b{x}) - \nabla f(\b{y})) \\
&+ \frac{L}{2} \frac{1}{L^2} \|\nabla f(\b{x}) - \nabla f(\b{y})\|_2^2 \\
&\overset{(a)}{=} \nabla f(\b{y})^\top (\b{y} - \b{x}) - \frac{1}{L} \|\nabla f(\b{x}) - \nabla f(\b{y})\|_2^2 \\
&+ \frac{1}{2L} \|\nabla f(\b{x}) - \nabla f(\b{y})\|_2^2 \\
&= \nabla f(\b{y})^\top (\b{y} - \b{x}) - \frac{1}{2L} \|\nabla f(\b{x}) - \nabla f(\b{y})\|_2^2,
\end{align*}
were $(a)$ is because $\|\nabla f(\b{x}) - \nabla f(\b{y})\|_2^2 = (\nabla f(\b{x}) - \nabla f(\b{y}))^\top (\nabla f(\b{x}) - \nabla f(\b{y}))$. 

\textbf{-- Proof for the second equation:}
If we exchange the points $\b{x}$ and $\b{y}$ in Eq. (\ref{equation_lemma_fy_fx_grady_y_x}), we have:
\begin{align}
f(\b{x}) - f(\b{y}) \leq &\,\nabla f(\b{x})^\top (\b{x} - \b{y}) \nonumber \\
&- \frac{1}{2 L} \|\nabla f(\b{y}) - \nabla f(\b{x})\|_2^2. \label{equation_lemma_fy_fx_grady_y_x_2}
\end{align}
Adding Eqs. (\ref{equation_lemma_fy_fx_grady_y_x}) and (\ref{equation_lemma_fy_fx_grady_y_x_2}) gives:
\begin{align*}
0 \leq &\,(\nabla f(\b{y}) - \nabla f(\b{x}))^\top (\b{y} - \b{x}) - \frac{1}{L} \|\nabla f(\b{y}) - \nabla f(\b{x})\|_2^2.
\end{align*}
Q.E.D.

\subsection{Proof for Lemma \ref{lemma_global_min_convex_function}}\label{appendix_lemma_global_min_convex_function}

Consider any $\b{y} \in \mathcal{D}$. We define $\b{z} := \alpha \b{y} + (1 - \alpha) \b{x}$. We choose a small enough $\alpha$ to have $\|\b{z} - \b{x}\|_2 \leq \epsilon$. Hence:
\begin{align*}
\epsilon &\geq \|\b{z} - \b{x}\|_2 = \|\alpha \b{y} + (1 - \alpha) \b{x}-\b{x}\|_2 = \|\alpha \b{y} - \alpha \b{x}\|_2 \\
&= \alpha \|\b{y} - \b{x}\|_2 \implies \alpha \leq \frac{\epsilon}{\|\b{y} - \b{x}\|_2}.
\end{align*}
As $\alpha \in [0,1]$, we should have $0 \leq \alpha \min(\epsilon / \|\b{y} - \b{x}\|_2, 1)$.
As $\b{x}$ is a local minimizer, according to Eq. (\ref{equation_local_minimizer}), we have $\exists\, \epsilon > 0 : \forall \b{z} \in \mathcal{D},\, \|\b{z} - \b{x}\|_2 \leq \epsilon \implies f(\b{x}) \leq f(\b{z})$.
As the function is convex, according to Eq. (\ref{equation_convex_function}), we have $f(\b{z}) = f\big(\alpha \b{y} + (1-\alpha) \b{x}\big) \leq \alpha f(\b{y}) + (1-\alpha) f(\b{x})$ (n.b. we have exchanged the variables $\b{x}$ and $\b{y}$ in Eq. (\ref{equation_convex_function})).
Hence, overall, we have:
\begin{align*}
&f(\b{x}) \leq f(\b{z}) \leq \alpha f(\b{y}) + (1-\alpha) f(\b{x}) \\
&\implies f(\b{x}) - (1-\alpha) f(\b{x}) \leq \alpha f(\b{y}) \\
&\implies \alpha f(\b{x}) \leq \alpha f(\b{y}) \implies f(\b{x}) \leq f(\b{y}), \forall \b{y} \in \mathcal{D}.
\end{align*}
So, $\b{x}$ is the global minimizer. Q.E.D.

\subsection{Proof for Lemma \ref{lemma_minimizer_gradient_zero}}\label{appendix_lemma_minimizer_gradient_zero}

\textbf{-- Proof for side} ($\b{x}^*$ is minimizer $\implies$ $\nabla f(\b{x}^*) = \b{0}$):

According to the definition of directional derivative, we have:
\begin{align*}
\nabla f(\b{x})^\top (\b{y} - \b{x}) = \lim_{t \rightarrow 0} \frac{f(\b{x} + t(\b{y} - \b{x})) - f(\b{x})}{t}.
\end{align*}
For a minimizer $\b{x}^*$, we have $f(\b{x}^*) \leq f(\b{x}^* + t(\b{y} - \b{x}^*))$ because it minimizes $f(.)$ for a neighborhood around it (n.b. $\b{x}^* + t(\b{y} - \b{x}^*)$ is a neighborhood of $\b{x}^*$ because $t$ tends to zero). Hence, we have:
\begin{align*}
0 &\leq \lim_{t \rightarrow 0} \frac{f(\b{x}^* + t(\b{y} - \b{x}^*)) - f(\b{x}^*)}{t} \\
&= \nabla f(\b{x}^*)^\top (\b{y} - \b{x}^*).
\end{align*}
As $\b{y}$ can be any point in the domain $\mathcal{D}$, we can choose it to be $\b{y} = \b{x}^* - \nabla f(\b{x}^*)$ so we have $\nabla f(\b{x}^*) = \b{x}^* - \b{y}$ and therefore:
\begin{align*}
0 &\leq \nabla f(\b{x}^*)^\top (\b{y} - \b{x}^*) = - \nabla f(\b{x}^*)^\top \nabla f(\b{x}^*) \\
&= - \|\nabla f(\b{x}^*)\|_2^2 \leq 0 \implies \nabla f(\b{x}^*) = \b{0}. 
\end{align*}

\textbf{-- Proof for side} ($\nabla f(\b{x}^*) = \b{0}$ $\implies$ $\b{x}^*$ is minimizer):

As the function $f(.)$ is convex, according to Eq. (\ref{equation_convex_function_firstDerivative}), we have:
\begin{align*}
f(\b{y}) \geq f(\b{x}^*) + \nabla f(\b{x}^*)^\top (\b{y} - \b{x}^*),
\end{align*}
$\forall \b{y} \in \mathcal{D}$. As we have $\nabla f(\b{x}^*) = \b{0}$, we can say $f(\b{y}) \geq f(\b{x}^*), \forall \b{y}$. So, $\b{x}^*$ is the global minimizer.

\subsection{Proof for Lemma \ref{lemma_first_order_optimality_condition}}\label{app_first_order_optimality_condition}

As $\b{x}^*$ is a local minimizer, according to Eq. (\ref{equation_local_minimizer}), we have $\exists\, \epsilon > 0 : \forall \b{y} \in \mathcal{D},\, \|\b{y} - \b{x}^*\|_2 \leq \epsilon \implies f(\b{x}^*) \leq f(\b{y})$.
Also, by Eq. (\ref{equation_fundamental_theorem_calculus}), we have $f(\b{y}) = f(\b{x}^*) + \nabla f(\b{x}^*)^\top (\b{y} - \b{x}^*) + o(\b{y} - \b{x}^*)$. From these two, we have:
\begin{align*}
&f(\b{x}^*) \leq f(\b{x}^*) + \nabla f(\b{x}^*)^\top (\b{y} - \b{x}^*) + o(\b{y} - \b{x}^*) \\
&\implies \nabla f(\b{x}^*)^\top (\b{y} - \b{x}^*) \geq 0.
\end{align*}
As $\b{y}$ can be any point in the domain $\mathcal{D}$, we can choose it to be $\b{y} = \b{x}^* - \nabla f(\b{x}^*)$ so we have $\nabla f(\b{x}^*) = \b{x}^* - \b{y}$ and therefore $\nabla f(\b{x}^*)^\top (\b{y} - \b{x}^*) = -\|\nabla f(\b{x}^*)\|_2^2 \geq 0$. Hence, $\nabla f(\b{x}^*) = 0$. Q.E.D.

\section{Proofs for Section \ref{section_first_order_methods}}

\subsection{Proof for Lemma \ref{lemma_time_complexity_line_search}}\label{app_time_complexity_line_search}

Because we halve the step size every time, after $\tau$ internal iterations of line-search, we have:
\begin{align*}
\eta^{(\tau)} := (\frac{1}{2})^\tau \eta^{(\tau)} = (\frac{1}{2})^\tau,
\end{align*}
where $\eta^{(\tau)} = 1$ is the initial step size. 
According to Eq. (\ref{equation_GD_step_size_eta_L_2}), we have:
\begin{align*}
&\eta^{(\tau)} = (\frac{1}{2})^\tau < \frac{1}{L} \overset{(a)}{\implies} \log_{\frac{1}{2}} (\frac{1}{2})^\tau > \log_{\frac{1}{2}} \frac{1}{L} \\
&\implies \tau > \log_{\frac{1}{2}} \frac{1}{L} = \frac{\log \frac{1}{L}}{\log \frac{1}{2}} = -\frac{\log \frac{1}{L}}{\log 2} \\
&= -\frac{1}{\log 2} (\log 1 - \log L) = \frac{\log L}{\log 2}.
\end{align*}
Q.E.D.

\subsection{Proof for Theorem \ref{theorem_GD_convergence_rate}}\label{app_GD_convergence_rate}

We re-arrange Eq. (\ref{equation_GD_decrease_cost}):
\begin{align}
&\|\nabla f(\b{x}^{(k)})\|_2^2 \leq 2L \big(f(\b{x}^{(k)}) - f(\b{x}^{(k+1)})\big), \forall k, \nonumber \\
&\implies \sum_{k=0}^t \|\nabla f(\b{x}^{(k)})\|_2^2 \leq 2L \sum_{k=0}^t \big(f(\b{x}^{(k)}) - f(\b{x}^{(k+1)})\big). \label{equation_proof_GD_convergence_rate_1}
\end{align}
The right-hand side of Eq. (\ref{equation_proof_GD_convergence_rate_1}) is a telescopic summation:
\begin{align*}
&\sum_{k=0}^t \big(f(\b{x}^{(k)}) - f(\b{x}^{(k+1)})\big) = f(\b{x}^{(0)}) - f(\b{x}^{(1)}) \\
&~~~~~+ f(\b{x}^{(1)}) - f(\b{x}^{(2)}) + \dots - f(\b{x}^{(t)}) + f(\b{x}^{(t)}) \\
&~~~~~- f(\b{x}^{(t+1)}) = f(\b{x}^{(0)}) - f(\b{x}^{(t+1)}).
\end{align*}
The left-hand side of Eq. (\ref{equation_proof_GD_convergence_rate_1}) is larger than summation of its smallest term for $(t+1)$ times:
\begin{align*}
(t+1) \min_{0\leq k \leq t} \|\nabla f(\b{x}^{(k)})\|_2^2 \leq \sum_{k=0}^t \|\nabla f(\b{x}^{(k)})\|_2^2.
\end{align*}
Overall, Eq. (\ref{equation_proof_GD_convergence_rate_1}) becomes:
\begin{align}\label{equation_proof_GD_convergence_rate_2}
(t+1) \min_{0\leq k \leq t} \|\nabla f(\b{x}^{(k)})\|_2^2 \leq 2L \big( f(\b{x}^{(0)}) - f(\b{x}^{(t+1)}) \big).
\end{align}
As $f^*$ is the minimum of function, we have:
\begin{align*}
&f(\b{x}^{(t+1)}) \geq f^* \\
&\implies
2L \big( f(\b{x}^{(0)}) - f(\b{x}^{(t+1)}) \big) \leq 2L \big( f(\b{x}^{(0)}) - f^*) \big).
\end{align*}
Hence, Eq. (\ref{equation_proof_GD_convergence_rate_2}) becomes:
\begin{align*}
(t+1) \min_{0\leq k \leq t} \|\nabla f(\b{x}^{(k)})\|_2^2 \leq 2L \big( f(\b{x}^{(0)}) - f^*) \big),
\end{align*}
which gives Eq. (\ref{equation_GD_upperbound_min_norm_gradient}). The right-hand side of Eq. (\ref{equation_GD_upperbound_min_norm_gradient}) is of the order $\mathcal{O}(1/t)$, resulting in Eq. (\ref{equation_GD_norm_gradient_sublinear_rate}). 
Moreover, for convergence, we desire:
\begin{align*}
\min_{0\leq k \leq t} \|\nabla f(\b{x}^{(k)})\|_2^2 \leq \epsilon \overset{(\ref{equation_GD_upperbound_min_norm_gradient})}{\implies} \frac{2 L (f(\b{x}^{(0)}) - f^*)}{t+1} \leq \epsilon,
\end{align*}
which gives Eq. (\ref{equation_GD_t_lowerbound_for_convergence}) by re-arranging for $t$. Q.E.D.

\subsection{Proof for Theorem \ref{theorem_GD_convergence_rate_convexFunction}}\label{app_GD_convergence_rate_convexFunction}

\begin{align}
&\|\b{x}^{(k+1)} - \b{x}^*\|_2^2 \overset{(a)}{=} \big\|\b{x}^{(k)} - \b{x}^* - \frac{1}{L} \nabla f(\b{x}^{(k)})\big\|_2^2 \nonumber\\
&=\! \big(\b{x}^{(k)} - \b{x}^* - \frac{1}{L} \nabla f(\b{x}^{(k)})\big)^\top \big(\b{x}^{(k)} - \b{x}^* - \frac{1}{L} \nabla f(\b{x}^{(k)})\big) \nonumber\\
&= \|\b{x}^{(k)} - \b{x}^*\|_2^2 -\frac{2}{L} (\b{x}^{(k)} - \b{x}^*)^\top \nabla f(\b{x}^{(k)}) \nonumber\\
&~~~~~+ \frac{1}{L^2} \|\nabla f(\b{x}^{(k)})\|_2^2, \label{equation1_theorem_GD_convergence_rate_convexFunction}
\end{align}
where $(a)$ is because of Eqs. (\ref{equation_update_point_numerical_optimization}) and (\ref{equation_GD_step_by_L}).
Moreover, according to Eq. (\ref{equation_lemma_gradfy_gradfx_y_x}), we have:
\begin{align}
&\big( \nabla f(\b{x}^{(k)}) - \nabla f(\b{x}^*) \big)^\top (\b{x}^{(k)} - \b{x}^*) \nonumber\\
&~~~~~~~~~~~~~~~~~~~~~~~~~\geq \frac{1}{L} \|\nabla f(\b{x}^{(k)}) - \nabla f(\b{x}^*)\|_2^2 \nonumber\\
&\overset{(\ref{equation_first_order_optimality_condition})}{\implies} \nabla f(\b{x}^{(k)})^\top (\b{x}^{(k)} - \b{x}^*) \geq \frac{1}{L} \|\nabla f(\b{x}^{(k)})\|_2^2. \nonumber \\
&\implies -\frac{2}{L}\nabla f(\b{x}^{(k)})^\top (\b{x}^{(k)} - \b{x}^*) \leq -\frac{2}{L^2} \|\nabla f(\b{x}^{(k)})\|_2^2. \label{equation2_theorem_GD_convergence_rate_convexFunction}
\end{align}
Using Eq. (\ref{equation2_theorem_GD_convergence_rate_convexFunction}) in Eq. (\ref{equation1_theorem_GD_convergence_rate_convexFunction}) gives:
\begin{align*}
&\|\b{x}^{(k+1)} - \b{x}^*\|_2^2 \leq \|\b{x}^{(k)} - \b{x}^*\|_2^2 -\frac{2}{L^2} \|\nabla f(\b{x}^{(k)})\|_2^2 \\
&+ \frac{1}{L^2} \|\nabla f(\b{x}^{(k)})\|_2^2 = \|\b{x}^{(k)} - \b{x}^*\|_2^2 -\frac{1}{L^2} \|\nabla f(\b{x}^{(k)})\|_2^2.
\end{align*}
This shows that at every iteration of gradient descent, the distance of point to $\b{x}^*$ is decreasing; hence:
\begin{align}\label{equation3_theorem_GD_convergence_rate_convexFunction}
&\|\b{x}^{(k)} - \b{x}^*\|_2^2 \leq \|\b{x}^{(0)} - \b{x}^*\|_2^2.
\end{align}
As the function is convex, according to Eq. (\ref{equation_convex_function_firstDerivative}), we have:
\begin{align}
&f(\b{x}^*) \geq f(\b{x}^{(k)}) + \nabla f(\b{x}^{(k)})^\top (\b{x}^* - \b{x}^{(k)}) \nonumber\\
&\implies f(\b{x}^{(k)}) - f(\b{x}^*) \leq \nabla f(\b{x}^{(k)})^\top (\b{x}^{(k)} - \b{x}^*) \nonumber\\
&~~~~~~~~~~~~~~~~~~~~~~~~~~~~\overset{(a)}{\leq} \|\nabla f(\b{x}^{(k)})\|_2 \|\b{x}^{(k)} - \b{x}^*\|_2 \nonumber\\
&~~~~~~~~~~~~~~~~~~~~~~~~~~~~\overset{(\ref{equation3_theorem_GD_convergence_rate_convexFunction})}{\leq} \|\nabla f(\b{x}^{(k)})\|_2 \|\b{x}^{(0)} - \b{x}^*\|_2 \nonumber\\
&\implies \frac{-1}{2L} \frac{\big(f(\b{x}^{(k)}) - f(\b{x}^*)\big)^2}{\|\b{x}^{(0)} - \b{x}^*\|_2^2} \geq \frac{-1}{2L} \|\nabla f(\b{x}^{(k)})\|_2^2, \label{equation4_theorem_GD_convergence_rate_convexFunction}
\end{align}
where $(a)$ is because of the Cauchy-Schwarz inequality. 
Also, from Eq. (\ref{equation_GD_decrease_cost}), we have:
\begin{align}
&f(\b{x}^{(k+1)}) \leq f(\b{x}^{(k)}) -\frac{1}{2L}\|\nabla f(\b{x}^{(k)})\|_2^2 \nonumber\\
&\implies f(\b{x}^{(k+1)}) - f(\b{x}^*) \nonumber\\
&\leq f(\b{x}^{(k)}) - f(\b{x}^*) -\frac{1}{2L}\|\nabla f(\b{x}^{(k)})\|_2^2 \nonumber\\
&\overset{(\ref{equation4_theorem_GD_convergence_rate_convexFunction})}{\leq} f(\b{x}^{(k)}) - f(\b{x}^*) - \frac{1}{2L} \frac{\big(f(\b{x}^{(k)}) - f(\b{x}^*)\big)^2}{\|\b{x}^{(0)} - \b{x}^*\|_2^2}. \label{equation5_theorem_GD_convergence_rate_convexFunction}
\end{align}
We define $\delta_k := f(\b{x}^{(k)}) - f(\b{x}^*)$ and $\mu := 1 /  (2L \|\b{x}^{(0)} - \b{x}^*\|_2^2)$. 
According to Eq. (\ref{equation_GD_decrease_cost}), we have:
\begin{align}\label{equation6_theorem_GD_convergence_rate_convexFunction}
\delta_{k+1} \leq \delta_k \implies \mu\, \frac{\delta_k}{\delta_{k+1}} \geq \mu. 
\end{align}
Eq. (\ref{equation5_theorem_GD_convergence_rate_convexFunction}) can be restated as:
\begin{align*}
&\delta_{k+1} \leq \delta_k - \mu\, \delta_k^2 \implies \mu\, \frac{\delta_k}{\delta_{k+1}} \leq \frac{1}{\delta_{k+1}} - \frac{1}{\delta_k} \\
&\overset{(\ref{equation6_theorem_GD_convergence_rate_convexFunction})}{\implies} \mu \leq \frac{1}{\delta_{k+1}} - \frac{1}{\delta_k} \implies \mu \sum_{k=0}^t (1) \leq \sum_{k=0}^t (\frac{1}{\delta_{k+1}} - \frac{1}{\delta_k}).
\end{align*}
The last term is a telescopic summation; hence:
\begin{align*}
&\mu (t+1) \leq \sum_{k=0}^t (\frac{1}{\delta_{k+1}} - \frac{1}{\delta_k}) \\
&= \frac{1}{\delta_1} - \frac{1}{\delta_0} + \frac{1}{\delta_2} - \frac{1}{\delta_1} + \dots + \frac{1}{\delta_{t+1}} = \frac{1}{\delta_{t+1}} - \frac{1}{\delta_0} \\
&\overset{(a)}{\leq} \frac{1}{\delta_{t+1}} \implies \delta_{t+1} \leq \frac{1}{\mu (t+1)} \\
&\implies f(\b{x}^{(t+1)}) - f^* \leq \frac{2 L \|\b{x}^{(0)} - \b{x}^*\|_2^2}{t+1},
\end{align*}
where $(a)$ is because $\delta_0 = f(\b{x}^{(0)}) - f(\b{x}^*) \geq 0$ because $f^*$ is the minimum in the convex function. Q.E.D.

\section{Proofs for Section \ref{section_second_order_methods}}

\subsection{Proof for Theorem \ref{theorem_sub_optimality_logBarrier}}\label{app_sub_optimality_logBarrier}

The Lagrangian for problem (\ref{equation_log_barrier_optimization_withLogBarrier}) is:
\begin{align*}
&\mathcal{L}(\b{x}, \b{\nu}) = f(\b{x}) - \frac{1}{t} \sum_{i=1}^{m_1} \log(-y_i(\b{x})) + \b{\nu}^\top (\b{A} \b{x} - \b{b}).
\end{align*}
According to Eq. (\ref{equation_KKT_x_dagger}), for parameter $t$, $\b{x}^\dagger(t)$ minimizes the Lagrangian:
\begin{align}
&\nabla_{\b{x}} \mathcal{L}(\b{x}^*(t),\b{\nu}) = \nabla_{\b{x}} f(\b{x}^*(t)) \nonumber\\
&~~~~~~~~~~+ \sum_{i=1}^{m_1} \frac{-1}{t\, y_i(\b{x}^*(t))} \nabla_{\b{x}} y_i(\b{x}^*(t)) + \b{A}^\top \b{\nu} \overset{\text{set}}{=} \b{0} \nonumber\\
&\implies \nabla_{\b{x}} \mathcal{L}(\b{x}^*(t),\b{\nu}) = \nabla_{\b{x}} f(\b{x}^*(t)) \nonumber\\
&~~~~~~~~~~+ \sum_{i=1}^{m_1} \lambda_i^*(t) \nabla_{\b{x}} y_i(\b{x}^*(t)) + \b{A}^\top \b{\nu} = \b{0}, \label{equation_theorem_sub_optimality_logBarrier_derivative_Lagrangian_wrt_x}
\end{align}
where we define $\lambda_i^*(t) := -1 / (t\, y_i(\b{x}^*(t)))$. Let $\b{\lambda}^*(t) := [\lambda_1^*(t), \dots, \lambda_{m_1}^*(t)]^\top$ and let $\b{\nu}^*(t))$ be the optimal dual variable $\b{\nu}$ for parameter $t$.
We take integral from Eq. (\ref{equation_theorem_sub_optimality_logBarrier_derivative_Lagrangian_wrt_x}) w.r.t. $\b{x}$ to retrieve the Lagrangian again:
\begin{align}
&\mathcal{L}(\b{x},\b{\lambda}^*(t), \b{\nu}^*(t)) = f(\b{x}^*(t)) \nonumber \\
&~~~~~~~~~~ + \sum_{i=1}^{m_1} \lambda_i^*(t) \nabla_{\b{x}} y_i(\b{x}) + \b{\nu}^*(t)^\top (\b{A} \b{x} - \b{b}). \label{equation_theorem_sub_optimality_logBarrier_Lagrangian}
\end{align}
The dual function is:
\begin{align*}
g(\b{\lambda}, \b{\nu}) = \inf_{\b{x}} \mathcal{L}(\b{x},\b{\lambda}^*(t), \b{\nu}^*(t)).
\end{align*}
The dual problem is:
\begin{align*}
&\sup_{\b{\lambda}, \b{\nu}}\, g(\b{\lambda}, \b{\nu}) = \sup_{\b{\lambda}, \b{\nu}}\, \inf_{\b{x}}\, \mathcal{L}(\b{x},\b{\lambda}^*(t), \b{\nu}^*(t)) \\
&= \mathcal{L}(\b{x}^*(t),\b{\lambda}^*(t), \b{\nu}^*(t)) \overset{(\ref{equation_theorem_sub_optimality_logBarrier_Lagrangian})}{=} f(\b{x}^*(t)) \\
&+ \sum_{i=1}^{m_1} \lambda_i^*(t) \nabla_{\b{x}} y_i(\b{x}^*(t)) + \b{\nu}^*(t)^\top (\underbrace{\b{A} \b{x}^*(t) - \b{b}}_{=\,\b{0}}) \\
&\overset{(a)}{=} f(\b{x}^*(t)) -\frac{1}{t} \sum_{i=1}^{m_1} (1) = f(\b{x}^*(t)) - \frac{m_1}{t} = f^* - \frac{m_1}{t},
\end{align*}
where $(a)$ is because we had defined $\lambda_i^*(t) = -1 / (t\, y_i(\b{x}^*(t)))$ and $\b{A} \b{x}^*(t) - \b{b} = \b{0}$ because the point is feasible. 
According to Eq. (\ref{equation_weak_duality}) for weak duality, we have $\sup_{\b{\lambda}, \b{\nu}}\, g(\b{\lambda}, \b{\nu}) \leq f_r^*$ where $f_r^*$ is the optimum of problem (\ref{equation_log_barrier_optimization_withLogBarrier}); hence, $f^* - (m_1/t) \leq f_r^*$. On the other hand, we $f_r^*$ is the optimum, we have:
\begin{align*}
f_r^* \overset{(\ref{equation_log_barrier_optimization_withLogBarrier})}{=} f(\b{x}^*) - \frac{1}{t} \sum_{i=1}^{m_1} \log(-y_i(\b{x}^*)) \leq f(\b{x}^*) = f^*.
\end{align*}
Hence, overall, Eq. (\ref{equation_sub_optimality_logBarrier}) is obtained. Q.E.D.

\bibliography{References}
\bibliographystyle{icml2016}

\end{document}